\documentclass[reqno,a4paper]{siamart1116}
\usepackage[top=1.5cm, bottom=1.5cm, left=2cm, right=2cm]{geometry}

\newtheorem{remark}{Remark}[section]
\usepackage{amsmath,amssymb,epsfig,float,color,mathtools,enumitem,comment}
\usepackage{arydshln}
\usepackage{booktabs}
\usepackage{mathrsfs}
\usepackage{stmaryrd}
\usepackage{hyperref}
\usepackage{cancel}
\usepackage{multirow}
\usepackage{graphicx,epstopdf} 
\usepackage[caption=false]{subfig} 

\newcommand{\ds}{\,\mbox{d}s}
\numberwithin{equation}{section}
\numberwithin{figure}{section}
\numberwithin{theorem}{section}

\definecolor{lightblue}{rgb}{0.22,0.45,0.70}
\definecolor{mygray}{rgb}{0.7,0.7,0.7}

\allowdisplaybreaks

\title{Equal order stabilized finite elements with Nitsche for stationary Navier-Stokes problem with slip boundary conditions : a priori and  a posteriori error analysis }
\author{Aparna Bansal\thanks{Department of Mathematics, Indian Institute of Technology Roorkee, Roorkee 247667, India. Email: \email{a\_bansal@ma.iitr.ac.in}.}
\and Nicol\'as A. Barnafi$^\P$\thanks{Instituto de Ingeniería Matemática y Computacional \& Facultad de Ciencias Biológicas, Pontificia Universidad Católica de Chile, Av Vicuña Mackenna 4860, Santiago, Chile, Email: \email{nicolas.barnafi@uc.cl}. \\
\indent$^\P$Centro de Modelamiento Matemático (CNRS IRL2807), Santiago, Chile.}
\and 
Rodolfo Araya \thanks{Departamento de Ingeniería Matemática and CI\textsuperscript{2}MA, Universidad de Concepción, Concepción, Chile \email{rodolfo.araya@udec.cl}.}
\and Dwijendra Narain Pandey\thanks{Department of Mathematics, Indian Institute of Technology Roorkee, Roorkee 247667, India. Email: \email{dwij@ma.iitr.ac.in}.}}
\date{\today}

\begin{document}

\maketitle
\begin{abstract}
In this work, we extend the equal-order stabilized scheme discussed in [Franca et al., Comput. Methods Appl. Mech. Engrg. 99 (1992) 209-233] to accommodate slip (i.e., Navier) boundary conditions for the stationary Navier-Stokes equations. Our analysis presents a robust formulation for implementing slip boundary conditions using Nitsche's method on arbitrarily complex boundaries. The well-posedness of the discrete problem is established under mild assumptions together with  optimal convergence rates for the approximation error. Furthermore, we establish the efficiency and reliability of residual-based a posteriori error estimators for the stationary discrete problem. Several well-known numerical tests validate our theoretical findings. The proposed method fits naturally within the framework of finite element implementation, offering both accuracy and enhanced flexibility in the selection of finite element pairs.
\end{abstract}

\section{Introduction}
Slip (or Navier ) boundary conditions naturally arise in the context of Stokes and Navier–Stokes equations, particularly when modeling biological surfaces \cite{bechert2000fluid}, slide coating processes \cite{christodoulou1989fluid}, or in turbulence modeling \cite{mohammadi1994analysis} and are essential for accurately capturing fluid behavior at small scales.
In statistical mechanics, the phenomenon of slip is well understood, as the Navier-Stokes equations have been shown to approximate the Boltzmann equation for small Knudsen numbers \cite{sone2002kinetic}, leading to the development of generalized slip-flow theory. These slip boundary conditions, often viewed as generalized Dirichlet conditions, are challenging to implement in standard finite element frameworks in the presence of complex geometries. 

There is an extensive literature for the imposition of slip boundary conditions. One of the earliest approaches involves enforcing slip conditions weakly through a Lagrange multiplier, as discussed in \cite{MR1707832, verfurth1986finite, MR1124131}, which are effective but they increase computational costs by introducing an extra variable. Penalty methods, which add a regularization term, have also been explored in \cite{MR351118, MR853660}. Despite their inconsistency and higher regularity requirements, they remain widely used because of their simplicity.  Nitsche's method is considered as a simple, consistent and primal technique to take into account the slip condition ( see \cite{MR4744101, MR2192480, MR2501054, MR341903, MR1365557, MR744045, MR3759094} ). Araya et al. extended this approach by discussing symmetric and non-symmetric variants of Nitsche’s method for the Stokes equations with slip conditions, using stabilized finite elements in \cite{MR4744101}. Additionally, a recent study \cite{MR4812237} combines symmetric Nitsche methods with inf-sup stable element pairs and Variational MultiScale (VMS) stabilization to solve the Navier-Stokes equations at high Reynolds numbers. Moreover, Poza et al. derived an equal-order stabilized formulation for the Stokes equations with variable viscosities in \cite{MR4669348}.

In adaptive finite element methods, a posteriori error estimators provide a way to quantify the local distribution of errors. Reliable estimator not only control the true error but also serve as stopping criterion in an adaptive refinement process. Furthermore, the efficiency of these estimators ensures that their convergence rate matches that of the true error. However, most existing error estimation results for finite element methods (see \cite{MR1885308, MR3059294} ) are focused on pure diffusion problems, with fewer studies addressing a posteriori estimators for advection-diffusion problems. Verfürth, in \cite{MR2182149}, presented robust a posteriori error estimates for stationary advection-diffusion equations, later extending these estimates to time-dependent cases in \cite{MR2182150}. Regarding the stationary Navier-Stokes or  Boussinesq equations it is worth mentioning the residual-based estimators proposed in \cite{MR3765706, MR2111747, MR4053090, MR3293502, MR1467674, MR1833723,MR4839142, MR4392237, MR2763864}. 

This work presents both a priori and a posteriori error analyses for the stationary Navier-Stokes equations with slip boundary conditions, employing Nitsche’s method in conjunction with a stabilized equal-order finite element formulation. Both symmetric and non-symmetric variants of Nitsche’s method are considered for their distinct advantages. In stabilized schemes, the introduction of appropriate stabilization terms ensures global inf-sup stability without the need for inf-sup compatible spaces for velocity and pressure. This also contributes to establish the coercivity with respect to a chosen norm. This approach guarantees the existence and uniqueness of the solution without relying on the inf-sup condition and facilitates the derivation of error bounds in this norm. Furthermore, a well-balanced numerical diffusion is incorporated through a stabilization parameter, which is crucial for resolving advection-dominated flows, especially at high Reynolds numbers (see \cite{ MR2278180, MR1186727, MR1377246}). The combination of stabilized methods with a posteriori error estimators improves the accuracy of numerical solutions while maintaining a low computational cost (see \cite{MR2398346}). This approach is particularly effective for approximating solutions involving multiple scales, as commonly encountered in non-linear Navier-Stokes equations. 
\subsection*{Outline}
The paper is organized as follows: Section \ref{section2} introduces the model and its continuous analysis. In Section \ref{section3}, we present the stabilized discrete formulation using equal-order stabilized finite elements and Nitsche’s method. The existence of solutions is established for the symmetric formulation under weak conditions on the given data and parameters using Brouwer’s fixed-point theorem. For all formulations i.e. symmetric, skew-symmetric, and incomplete, the existence and uniqueness of solutions are proven under strict conditions using Banach’s fixed-point theorem, as detailed in Section \ref{section4}. Section \ref{section5} derives a priori error estimates and proves the optimal convergence of the method. In Section \ref{section6}, we develop and analyze the efficiency and reliability of a residual-based a posteriori error estimator for the stationary problem. Finally, Section \ref{section8} presents numerical experiments that validate the theoretical convergence rates and demonstrate the robustness of the proposed a posteriori error estimators.

\section{The model problem}\label{section2}
The stationary incompressible Navier-Stokes equations are stated as follows
\begin{equation}\label{1}
\begin{aligned}
- 2 \nu \nabla \cdot  \varepsilon (\boldsymbol{u})  + \boldsymbol{u} \cdot \nabla \boldsymbol{u}+\nabla p  =\boldsymbol{f}, \quad 
\nabla \cdot \boldsymbol{u} &=0 \quad \text {in}\,\Omega , \\ 
\boldsymbol{u} &= \boldsymbol{0} \quad \text {on}\, \Gamma_D,
\end{aligned}
\end{equation}
posed on a spatial domain  $\Omega \subseteq \mathbb{R}^d, d \in\{2,3\}$ with Lipschitz boundary $\Gamma = \overline{\Gamma}_D \cup \overline{\Gamma}_{\mathrm{Nav}}$ and $\Gamma_D \cap \Gamma_{\mathrm{Nav}} = \emptyset$, where  $\boldsymbol{u}$ is the fluid velocity, $\nu>0$ is the viscosity, $p$ is the fluid pressure and  $\boldsymbol{f} $ represents the external body force on $\Omega$.
The Navier boundary condition on $\Gamma_{\mathrm{Nav}}$ is defined as
\begin{equation}\label{2}
\begin{aligned}
\boldsymbol{u} \cdot \boldsymbol{n} =0, \quad
2 \nu \boldsymbol{n}^t \varepsilon ( \boldsymbol{u}) \boldsymbol{\tau}^i + \beta \boldsymbol{u} \cdot \boldsymbol{\tau}^i =0 \quad \text {on}\,  \quad \Gamma_{\mathrm{Nav}}, \quad i =1,...,(d-1). 
\end{aligned}
\end{equation}
The Navier boundary condition allows the fluid to slip along the boundary and $\beta>0$ is a friction coefficient. It requires that the tangential component of the stress vector at the boundary be proportional to the tangential velocity. The strain tensor is defined as $\varepsilon(\boldsymbol{u}) = \frac{1}{2}(\nabla \boldsymbol{u} + \nabla \boldsymbol{u}^\top)$, with $\boldsymbol{n}$ and $\boldsymbol{\tau}^i$ representing the unit normal and tangent vectors to the boundary $\Gamma$, respectively.

For the sake of simplicity, throughout the analysis, $C$ will denote a generic positive constant independent of the mesh size $h$ but possibly dependent on the model parameters. We will also abuse notation by denoting $\delta_i, \gamma_j >0$ with $i,j\in \mathbb N$ as an arbitrary constant with different values at different occurrences, arising from the use of  Young's inequality. Additionally, whenever an inequality holds for positive constants independent of the mesh size and dependent on the parameters, we will use the symbols $\lesssim$ or $\gtrsim$ and omit specific constants. The assumption of homogeneity in the boundary conditions is made to simplify the subsequent analysis, as lifting operators have already been established \cite{MR3974685}. Non-homogeneous boundary conditions are utilized in the numerical tests in section \ref{section8}. 

 Define the spaces:
 $$
 \begin{aligned}
 \mathrm{V} &  \coloneqq \left\{\boldsymbol{v} \in \boldsymbol{H}^1(\Omega): \,\boldsymbol{v} \cdot \boldsymbol{n}=0 \,\text{ on }\,  \Gamma_{\mathrm{Nav}} , \boldsymbol{v} = \boldsymbol{0} \, \text{ on }\, \Gamma_D \right\}, \qquad
 \Pi  \coloneqq {L}^2_{0}(\Omega),
 \end{aligned}
 $$
 with their natural inner products, and consider the product space $\mathrm{V} \times \Pi$ equipped with the norm
$$
\|(\boldsymbol{v},q)\|^2 = \nu \|\varepsilon(\boldsymbol{v}) \|^2_{0,\Omega} +  \| q \|^2_{0,\Omega}  
$$
Now, we recall some classical results which will be needed in the forthcoming analysis sections.
\begin{lemma}{\cite{MR851383, MR1318914}}\label{lemma2.1}
There are positive constants $\alpha$ and $\beta$, depending only on $\Omega$, such that for all $\boldsymbol{v}, \boldsymbol{w} \in \mathrm{V}$ and $q \in \Pi$, it holds

$$
\begin{aligned}
 \sup _{\boldsymbol{v} \in \mathrm{V} \backslash(\boldsymbol{0}\}} \frac{(q, \nabla \cdot \boldsymbol{v})}{|\boldsymbol{v}|_{1, \Omega}} & \geq \beta\|q\|_{0, \Omega}, \\
( \boldsymbol{w} \cdot \nabla \boldsymbol{v} , \boldsymbol{v}) & \le \alpha|\boldsymbol{u}|_{1, \Omega}|\boldsymbol{w}|_{1, \Omega}|\boldsymbol{v}|_{1, \Omega}.
\end{aligned}
$$
Moreover, for all $\boldsymbol{u}, \boldsymbol{v}, \boldsymbol{w} \in \mathrm{V}$, it holds $( \boldsymbol{w} \cdot \nabla \boldsymbol{v} , \boldsymbol{v})=-\frac{1}{2}(\nabla \cdot \boldsymbol{w}, \boldsymbol{v} \cdot \boldsymbol{v}).$
\end{lemma}

The standard weak formulation of \eqref{1} and \eqref{2} is given by: Find $\left(\boldsymbol{u}, p \right) \in \mathrm{V} \times \Pi $, such that
	\begin{align}\label{bb}
	\mathcal{A}^{\mathrm{NS}}\left(\boldsymbol{u}, p  ; \boldsymbol {v}, q\right)=\mathcal{F}^{\mathrm{NS}}(\boldsymbol{v})  \quad \forall  \left(\boldsymbol{v}, q \right) \in \mathrm{V} \times \Pi,
	\end{align}
	where
	$$
	\mathcal{A}^{\mathrm{NS}}\left(\boldsymbol{u}, p  ; \boldsymbol {v}, q\right) \coloneqq  2 \nu \left(\varepsilon (\boldsymbol{u}), \varepsilon  (\boldsymbol{v})\right)_{\Omega} +(\boldsymbol{u} \cdot \nabla \boldsymbol{u}, \boldsymbol{v})_{\Omega} -(p, \nabla \cdot \boldsymbol{v})_{\Omega} -(q, \nabla \cdot \boldsymbol{u})_{\Omega} +\int_{\Gamma_{\mathrm{Nav}}} \beta \sum_i^{d-1}\left(\boldsymbol{\tau}^i \cdot \boldsymbol{v}\right)\left(\boldsymbol{\tau}^i \cdot \boldsymbol{u}\right) d s,
	$$
	$$
	\mathcal{F}^{\mathrm{NS}}(\boldsymbol{v}) \coloneqq \langle \boldsymbol{f},\boldsymbol{v} \rangle_{\Omega},
	$$
    where $\langle\cdot, \cdot\rangle_{\Omega}$ represents the duality pairing between $\mathrm{V}$ and its dual $\mathrm{V}^{\prime}$.

\section{Discrete stabilized scheme}\label{section3}
We assume that the polygonal computational domain ${\Omega}$ is discretized using a collection of regular partitions, denoted as $\{\mathcal{K}_h\}_{h>0}$, where $\Omega \subset \mathbb{R}^d$ is divided into simplices $K$ (triangles in 2D or tetrahedra in 3D) with a diameter $h_{{K}}$. The characteristic length of the finite element mesh $\mathcal{K}_h$ is denoted as $h \coloneqq \max _{{K} \in \mathcal{K}_h} h_{{K}}$. For a given triangulation $\mathcal{K}_h$, we define $\mathcal{E}_h$ as the set of all faces in $\mathcal{K}_h$, with the following partitioning
  		$$
  		\mathcal{E}_h \coloneqq \mathcal{E}_{\Omega} \cup \mathcal{E}_D \cup \mathcal{E}_{\mathrm{Nav}}
  		$$
  		where $\mathcal{E}_{\Omega}$ represents the faces  lying in the interior of $\Omega$, $\mathcal{E}_{\mathrm{Nav}}$ represents the faces lying on the boundary $\Gamma_{\mathrm{Nav}}$, and $\mathcal{E}_D$ represents the  faces lying on the boundary $\Gamma_D$. Additionally, $h_E$ denotes the $(d-1)$ dimensional diameter of a face. Here \emph{faces} loosely refers to the geometrical entities of co-dimension 1.
        
        Now, let us introduce the discrete finite element pair.
  		$$
  		\begin{aligned}
  		& \mathrm{V}_h \coloneqq \left\{\boldsymbol{v}_h \in \boldsymbol{C}(\overline{\Omega}) : \boldsymbol{v}_h =0 ~\, \text{on}\,~ E \in \mathcal{E}_D, \left.\boldsymbol{v}_h\right|_{K} \in \mathbb {P}_k({K}), k \geq  1 \quad \forall {K} \in \mathcal{K}_h,  \right\}, \\
  		& \mathrm{\Pi}_h \coloneqq \left\{q_h \in \mathrm{C}(\overline{\Omega}) :\left.q_h\right|_{{K}} \in \mathbb {P}_{k}({K}), k \geq 1 \quad \forall {K} \in \mathcal{K}_h\right\} \cap \Pi ,
  		\end{aligned}
  		$$
  		where $\mathbb{P}_k(K)$ is the space of polynomials of degree $k$ defined on $K$.
Let  $\boldsymbol{I}_h:\mathrm{V}\to \mathrm{V}_h$  and  $J_h: \Pi \to \Pi_h$  denote the vectorial and scalar versions of the Scott–Zhang interpolant, respectively. The following results regarding the approximation properties of these operators are established.
\begin{lemma}\label{interpolation_estimates}
For each element \( K \in \mathcal{K}_h \) and \( k \geq 1 \), there exist constants \( C_{SZ} \) and \( \tilde{C}_{SZ} \), independent of \( h_K \), such that the following inequalities hold:
\begin{align*}
    \left\|\boldsymbol{u} - \boldsymbol{I}_h \boldsymbol{u}\right\|_{0, K} + h_K \left|\boldsymbol{u} - \boldsymbol{I}_h \boldsymbol{u}\right|_{1, K} + h_K^2 \left|\boldsymbol{u} - \boldsymbol{I}_h \boldsymbol{u}\right|_{2, K} &\leq C_{SZ} h_K^{k+1} |\boldsymbol{u}|_{k+1, \omega_K}, \quad \forall \boldsymbol{u} \in \boldsymbol{H}^{k+1}(\omega_K), \\
    \left\|q - J_h q\right\|_{0, K} + h_K \left|q - J_h q\right|_{1, K} &\leq \tilde{C}_{SZ} h_K^k |q|_{k, \omega_K}, \quad \forall q \in H^k(\omega_K),
\end{align*}
where \( \omega_K \) is the patch defined as
   $ \omega_K := \bigcup_{\bar{K} \cap \bar{K}' \neq \emptyset} K'.$
\begin{proof}
    See \cite{MR2882148}. Note that for any \( q \in \Pi \), the Scott-Zhang interpolant \( J_h q \) does not typically belong to \( \Pi = L_0^2(\Omega) \). However, its modified version $J_h q - \frac{1}{|\Omega|} \int_{\Omega} J_h q \in \Pi,$
    which we denote as \( J_h q \) (with slight abuse of notation), also satisfy the above estimates.
\end{proof}
\end{lemma}
  	\begin{remark}
  	It is noted that $\mathrm{\Pi}_h$ is a subspace of $\mathrm{\Pi}$, but $\mathrm{V}_h$ is not a subspace of $\mathrm{V}$. In that sense, Nitsche's method can be considered as a non-conforming finite element approximation.
  	 \end{remark}
We consider the discrete norms
\begin{align*}
& \| \boldsymbol{v}_h \|_{1, \mathcal{K}_h}^2 := \nu \|\varepsilon(\boldsymbol{v}_h) \|^2_{0,\Omega} + \sum_{E \in \mathcal{E}_{\mathrm{Nav}}} \frac{\nu}{h_E} \|\boldsymbol{v}_h \cdot \boldsymbol{n} \|_{0,E}^2 \\
& \left|\!\left|\!\left| (\boldsymbol{v}_h, q_h )\right|\!\right|\!\right|^2 :=  \nu \|\varepsilon(\boldsymbol{v}_h) \|^2_{0,\Omega} + \sum_{E \in \mathcal{E}_{\mathrm{Nav}}} \frac{\nu}{h_E} \|\boldsymbol{v}_h \cdot \boldsymbol{n} \|_{0,E}^2 + \sum_{K \in \mathcal{K}_h} \frac{h_K^2}{\nu} \|\nabla q_h \|^2_{0,K}
\end{align*}
Next, we equip the dual of the discrete spaces with the following norm:
$$
\|(\boldsymbol{v}_h, q_h)\|_{\left(\mathrm{V}_h \times \Pi_h\right)^{\prime}}:=\sup _{\left\|\left(\boldsymbol{w}_h, r_h\right)\right\| \leq 1}\left\{\left\langle\boldsymbol{v}_h, \boldsymbol{w}_h\right\rangle+\left(q_h, r_h\right)\right\}
$$
Consider the parameter $\theta \in \{-1, 0, 1\}$ , a Nitsche penalty parameter $\gamma > 0$, and stabilization parameters $\tau$ and $\delta$. The stabilized discrete scheme used builds upon the framework presented in \cite{MR1186727}, and  combines it with the Nitsche approach. The formulation is given by: \textit{Find $\left(\boldsymbol{u}_h, p_h \right) \in \mathrm{V}_h \times \Pi_h $, such that}
  		\begin{align}\label{DSF}
  		\mathcal{A}_h^{\mathrm{NS}}\left(\boldsymbol{u}_h, p_h ; \boldsymbol{v}_h, q_h\right)=\mathcal{F}_h^{\mathrm{NS}}(\boldsymbol{v}_h)  \quad \forall  \left(\boldsymbol{v}_h, q_h \right) \in \mathrm{V}_h \times \Pi_h,
  		\end{align}
  with 
  $$
\begin{aligned}
\mathcal{A}_h^{\mathrm{NS}}\left(\boldsymbol{u}_h, p_h ; \boldsymbol{v}_h, q_h\right)   &\coloneqq \sum_{K \in \mathcal{K}_{h}} \bigg( 2 \nu \left(\varepsilon (\boldsymbol{u}_h), \varepsilon (\boldsymbol{v}_h) \right)_{K} + (\boldsymbol{u}_h \cdot \nabla \boldsymbol{u}_h, \boldsymbol{v}_h)_{K} - (p_h, \nabla \cdot \boldsymbol{v}_h)_{K} - (q_h, \nabla \cdot \boldsymbol{u}_h)_{K} \bigg) \\ 
&\quad + \sum_{E \in \mathcal{E}_{\mathrm{Nav}}} \bigg( -\int_{E} \boldsymbol{n}^t( 2 \nu \varepsilon (\boldsymbol{u}_h) - p_h I) \boldsymbol{n} (\boldsymbol{n} \cdot \boldsymbol{v}_h) \, ds  - \theta \int_{E} \boldsymbol{n}^t(2 \nu \varepsilon (\boldsymbol{v}_h) - q_h I) \\ 
&\quad  \boldsymbol{n} (\boldsymbol{n} \cdot \boldsymbol{u}_h) \, ds + \int_{E} \beta \sum_{i=1}^{d-1} (\boldsymbol{\tau}^i \cdot \boldsymbol{v}_h)(\boldsymbol{\tau}^i \cdot \boldsymbol{u}_h)\, ds  + \gamma \nu \int_{E} h_E^{-1} (\boldsymbol{u}_h \cdot \boldsymbol{n})(\boldsymbol{v}_h \cdot \boldsymbol{n})  \, ds \bigg) \\& \quad + \sum_{K \in \mathcal{K}_{h}}  \int_K  \tau\bigg(-2 \nu \nabla \cdot  \varepsilon (\boldsymbol{u}_h) + \boldsymbol{u}_h \cdot \nabla \boldsymbol{u}_h + \nabla p_h \bigg) \bigg(-2 \nu \nabla \cdot  \varepsilon (\boldsymbol{v}_h)  + \boldsymbol{u}_h \cdot \nabla \boldsymbol{v}_h \\ & \quad  - \nabla q_h \bigg) +  \sum_{K \in \mathcal{K}_{h}} \delta  \int_K \nabla \cdot \boldsymbol{u}_h \nabla \cdot \boldsymbol{v}_h, \\
  \mathcal{F}^{\mathrm{NS}}_h(\boldsymbol{v}_h,q_h) &\coloneqq  \left(\boldsymbol{f} ,\boldsymbol{v}_h \right)_K + \sum_{K \in \mathcal{K}_{h}} \tau\left( \boldsymbol{f} , - 2 \nu \nabla \cdot  \varepsilon (\boldsymbol{v}_h) + \boldsymbol{u}_h \cdot \nabla \boldsymbol{v}_h - \nabla q_h  \right)_K  ,
   \end{aligned}
$$
where $\theta = \{1,0,-1\}$ represents the symmetric, incomplete and skew-symmetric variants of Nitsche, respectively. The stabilization parameters \( \tau \) and \( \delta \) are defined as
\begin{align*}
    \delta &:= \lambda |\boldsymbol{u}_h(\boldsymbol{x})|_p h_K \xi(\operatorname{Re}_K(\boldsymbol{x})), \qquad
    \tau := \frac{h_K}{2|\boldsymbol{u}_h(\boldsymbol{x})|_p} \xi(\operatorname{Re}_K(\boldsymbol{x})),
\end{align*}
where 
\[
    \operatorname{Re}_K(\boldsymbol{x}) := \frac{m_K |\boldsymbol{u}_h(\boldsymbol{x})|_p h_K}{4\nu}, \quad
    \xi(y) := 
    \begin{cases} 
        y, & 0 \leq y < 1, \\
        1, & y \geq 1,
    \end{cases} \quad
    m_K := \min\left\{ \frac{1}{3}, 2 C_K \right\},
\]
and \( \lambda > 0 \) is a positive parameter. Here, \( |\boldsymbol{u}_h(\boldsymbol{x})|_p \) denotes the \( l_p \)-norm, with \( 1 \leq p \leq \infty \), and \( C_K \) is the inverse inequality constant.
\section{Well-posedness and stability analysis}\label{section4}
Before working on our results, we sequence the following technical lemmas.
\begin{lemma}{\cite[Proposition 1.135]{MR2050138}}\label{local_inverse}
There exists a positive constant $C$, independent of $h$, such for all $\boldsymbol{v}_h \in \mathrm{V}_h$ and all $q_h \in \Pi_h$, we have
$$
\begin{aligned}
& \left\|\boldsymbol{v}_h\right\|_{l, p, K} \leq C h_K^{m-l+d(1 / p-1 / q)}\left\|\boldsymbol{v}_h\right\|_{m, q, K},
\end{aligned}
$$
where $0 \leq m \leq l$ and $1 \leq p, q \leq \infty$.
\end{lemma}
\begin{lemma}{\cite[Lemma 12.8]{MR4242224}}\label{Trace Inequality $-I$}
Let $\boldsymbol{v}_h \in \mathrm{V}_h$ then for each $K \in \mathcal{K}_h, E \subset \partial K$, there exists a positive constant $C_{\mathrm{tr}}$, independent of $K$, such that

$$
\left\|\boldsymbol{v}_h\right\|_{0, E} \leq C_{\mathrm{tr}} h_K^{-\frac{1}{2}}\left\|\boldsymbol{v}_h\right\|_{0, K}.
$$
\end{lemma}
\begin{lemma}{\cite[Lemma 12.15]{MR4242224}}\label{Trace Inequality $- II$} Let $\boldsymbol{v}_h \in \boldsymbol{H}^1(K)$, with $K \in \mathcal{K}_h$. Then for any facet $E \subset \partial K$, there exists a positive constant $C_{\mathrm{tr}}^{\prime}$, independent of $K$, such that

$$
\|\boldsymbol{v}_h\|_{0, E} \leq C_{\mathrm{tr}}^{\prime}\left\{h_K^{-1 / 2}\|\boldsymbol{v}_h\|_{0, K}+h_K^{1 / 2}\|\nabla \boldsymbol{v}_h\|_{0, K}\right\}.
$$
\end{lemma}
\begin{lemma}{\cite[Lemma 1.138]{MR2050138}}\label{H1Hinfty}
    Let \( \boldsymbol{v}_h \in \mathrm{V}_h \). Then, the following inequalities hold:
    \begin{align*}
        \|\boldsymbol{v}_h\|_{\infty, K} & \leq C h_K^{-1} \|\boldsymbol{v}_h\|_{0, K}, \\
        h_K |\boldsymbol{v}_h|_{1, K} & \leq C \|\boldsymbol{v}_h\|_{0, K}.
    \end{align*}
\end{lemma}
\begin{remark}
Let us first note that the stability parameter $\tau$ is bounded by a constant in each element domain $K$. By definition
\begin{align}
\tau  &= \frac{h_K}{2|\boldsymbol{u}_h(\boldsymbol{x})|_p}, \quad \operatorname{Re}_K(\boldsymbol{x}) \geq 1 \qquad
\tau = \frac{m_k h_K^2}{8 \nu}, \quad 0 \leq \operatorname{Re}_K(\boldsymbol{x}) < 1   \label{rem2}
\end{align}
Therefore, for $\operatorname{Re}_K(\boldsymbol{x}) \geq 1$,
\begin{align}\label{rem3}
\tau  & =\frac{h_K}{2|\boldsymbol{u}_h(\boldsymbol{x})|_p} \frac{1}{\operatorname{Re}_K(\boldsymbol{x})} \frac{m_k|\boldsymbol{u}_h(\boldsymbol{x})|_p h_K}{4 \nu} \leq \frac{m_k h_K^2}{8 \nu}
\end{align}
and combining with the definition \eqref{rem2}, we conclude that the bound \eqref{rem3} is valid for all values of $\operatorname{Re}_K(\boldsymbol{x})$.
\end{remark}
Before heading to the proof of the existence and the uniqueness of a solution for discrete problem \eqref{DSF}, we need some auxiliary results. 
We define the operator $\mathscr{P}: \mathrm{V}_h \longrightarrow \Pi_h$ by
\begin{equation}\label{ope_P}
\begin{aligned}
\sum_{K \in \mathcal{K}_h} \tau\left(\nabla \mathscr{P}\left(\boldsymbol{w} \right), \nabla r \right)_K  = -\left(r, \nabla \cdot \boldsymbol{w} \right) + \theta \left(r, \boldsymbol{w} \cdot \boldsymbol{n} \right)- \sum_{K \in \mathcal{K}_h} \tau \left( - 2 \nu \nabla \cdot  \varepsilon (\boldsymbol{w})+ \boldsymbol{w} \cdot \nabla \boldsymbol{w}   - \boldsymbol{f} , \nabla r \right),
\end{aligned}
\end{equation}

for all $\boldsymbol{w} \in \mathrm{V}_h,\, r \in \Pi_h$. Observe that $\mathscr{P}$ is well-defined from Lax-Milgram's Theorem with the norm
$
\left\|r\right\|_*^2 := \sum_{K \in \mathcal{K}_h} \tau \left\|\nabla r \right\|_{0, K}^2 
$
Also, define the mapping $\mathscr{N}: \mathrm{V}_h \longrightarrow \mathrm{V}_h$ by
$$
\begin{aligned}
\left(\mathscr{N}\left(\boldsymbol{w}\right), \boldsymbol{v} \right)& = 2 \nu \left(\varepsilon (\boldsymbol{w}), \varepsilon  (\boldsymbol{v}) \right)+\left(\boldsymbol{w} \cdot \nabla \boldsymbol{w}, \boldsymbol{v}\right)-\left(\mathscr{P}\left(\boldsymbol{w}\right), \nabla \cdot \boldsymbol{v}\right)-\left(\boldsymbol{f} , \boldsymbol{v}\right) \\& \quad  + \sum_{E \in \mathcal{E}_{\mathrm{Nav}}} \bigg( -\int_{E} \boldsymbol{n}^t(2 \nu \varepsilon ( \boldsymbol{w} ) - \mathscr{P}\left( \boldsymbol{w} \right) I) \boldsymbol{n} (\boldsymbol{n} \cdot \boldsymbol{v}) \, ds  - \theta \int_{E} \boldsymbol{n}^t \left(2 \nu \varepsilon ( \boldsymbol{v}) \right) \boldsymbol{n} (\boldsymbol{n} \cdot \boldsymbol{w}) \, ds   \\& \quad + \int_{E} \beta \sum_{i=1}^{d-1} (\boldsymbol{\tau}^i \cdot \boldsymbol{v})(\boldsymbol{\tau}^i \cdot \boldsymbol{w}) \, ds  + \gamma \nu \int_{E} h_E^{-1} (\boldsymbol{w} \cdot \boldsymbol{n})(\boldsymbol{v} \cdot \boldsymbol{n}) \, ds \bigg)  +\sum_{K \in \mathcal{K}_h} \delta \left( \nabla \cdot \boldsymbol{w}, \nabla \cdot \boldsymbol{v} \right)_K  \\& \quad  + \sum_{K \in \mathcal{K}_h} \tau\left( -2 \nu \nabla \cdot  \varepsilon (\boldsymbol{w}) + \boldsymbol{w} \cdot \nabla \boldsymbol{w}  + \nabla \mathscr{P} \left(\boldsymbol{w}\right) - \boldsymbol{f} , -2 \nu \nabla \cdot  \varepsilon (\boldsymbol{v}) +  \boldsymbol{w} \cdot \nabla \boldsymbol{v} \right).
\end{aligned}
$$
The next result provides a characterization of the solution of problem \eqref{DSF} with respect to the operators $\mathscr{P}$ and $\mathscr{N}$.

\begin{lemma}
The pair $\left(\boldsymbol{u}_h, p_h\right) \in \mathrm{V}_h \times \Pi_h$ is a solution of problem \eqref{DSF} if and only if $\mathscr{N}\left(\boldsymbol{u}_h\right)=\boldsymbol{0}$ and $p_h=\mathscr{P}\left(\boldsymbol{u}_h\right)$.
\end{lemma}
\begin{proof}
The proof is similar to Lemma 3.5 in \cite{MR2914281}.
\end{proof}
We are now ready to prove the well-posedness of problem \eqref{DSF}. The proof follows closely the arguments presented in \cite{MR2914281}.
\subsection{Existence of the discrete solution under weak conditions}
\begin{theorem}\label{existence_theorem}
There exists a positive constant \( \tilde{C} \), independent of \( h \) and \( \nu \), such that the problem \eqref{DSF} has at least one solution \( (\boldsymbol{u}_h, p_h) \) for \( \theta = 1 \), provided that the Nitsche parameter \( \gamma \) is sufficiently large and 
\begin{equation}\label{assumption}
h^{1-\kappa}\left\{\left\|\boldsymbol{f} \right\|_{-1, \Omega}^2 + \sum_{K \in \mathcal{K}_h} \tau\left\| \boldsymbol{f}  \right\|_{0, K}^2\right\}^{1 / 2} \leq \tilde{C}, \qquad 0 < \kappa < 1.
\end{equation}
\begin{proof}
To prove the existence of the solution of the nonlinear formulation \eqref{DSF}, we use a Corollary of Brouwer's Fixed Point Theorem (see \cite[Chapter IV, Corollary 1.1]{MR851383}). Let $\boldsymbol{u}_h \in \mathrm{V}_h$, with $\| \boldsymbol{u}_h \|_{1,\mathcal{K}_h}=R$, where $R$ is a positive number that will be choosen later. Denote
$$
\begin{aligned}
& a_1:=\left\{\sum_{K \in \mathcal{K}_h} \tau\left\| -2 \nu \nabla \cdot  \varepsilon (\boldsymbol{u}_h) + \boldsymbol{u}_h \cdot \nabla \boldsymbol{u}_h  + \nabla \mathscr{P}\left(\boldsymbol{u}_h \right) \right\|_{0, K}^2\right\}^{1 / 2},
 \\
&a_2:=\left\{\left\|\boldsymbol{f} \right\|_{-1, \Omega}^2+\sum_{K \in \mathcal{K}_h} \tau\left\|\boldsymbol{f}  \right\|_{0, K}^2\right\}^{1 / 2}, \qquad a_3:=\left\{\sum_{K \in \mathcal{K}_h}  \delta \left\|\nabla \cdot \boldsymbol{u}_h \right\|_{0, K}^2\right\}^{1 / 2}.
\end{aligned}
$$

Taking $r=\mathscr{P}\left(\boldsymbol{u}_h\right)$ and $\boldsymbol{w} = \boldsymbol{u}_h$ in  \eqref{ope_P} give us
\[
\begin{aligned}
-(\mathscr{P}(\boldsymbol{u}_h), \nabla \cdot \boldsymbol{u}_h) 
&= -\theta \sum_{E \in \mathcal{E}_{\mathrm{Nav}}} 
\big( \mathscr{P}(\boldsymbol{u}_h), \boldsymbol{u}_h \cdot \boldsymbol{n} \big)_E + \sum_{K \in \mathcal{K}_h} 
\tau \Big( -2 \nu \nabla \cdot \varepsilon(\boldsymbol{u}_h) 
+ \boldsymbol{u}_h \cdot \nabla \boldsymbol{u}_h - \boldsymbol{f} \\
&\quad + \nabla \mathscr{P}(\boldsymbol{u}_h), 
\nabla \mathscr{P}(\boldsymbol{u}_h) \Big)_K.
\end{aligned}
\]
Using the above identity in the definition of operator $\mathscr{N}$, we get
\begin{align*}
\left(\mathscr{N}\left(\boldsymbol{u}_h\right), \boldsymbol{u}_h\right) 
&= 2 \nu \left\|\varepsilon(\boldsymbol{u}_h)\right\|_{0, \Omega}^2 
+ \left( \boldsymbol{u}_h \cdot \nabla \boldsymbol{u}_h, \boldsymbol{u}_h\right) 
- \left(\boldsymbol{f}, \boldsymbol{u}_h\right) \\
&\quad + \sum_{E \in \mathcal{E}_{\mathrm{Nav}}} \bigg( 
- \int_{E} \boldsymbol{n}^t (2 \nu \varepsilon (\boldsymbol{u}_h)) \boldsymbol{n} (\boldsymbol{n} \cdot \boldsymbol{u}_h) \, ds 
- \theta \int_{E} \boldsymbol{n}^t (2 \nu \varepsilon (\boldsymbol{u}_h)) \boldsymbol{n} (\boldsymbol{n} \cdot \boldsymbol{u}_h) \, ds \\
&\quad + (1- \theta) \int_E \mathscr{P}(\boldsymbol{u}_h) (\boldsymbol{n} \cdot \boldsymbol{u}_h) \, ds 
+ \int_{E} \beta \sum_{i=1}^{d-1} (\boldsymbol{\tau}^i \cdot \boldsymbol{u}_h)^2 \, ds 
+ \gamma \nu \int_{E} h_E^{-1} (\boldsymbol{u}_h \cdot \boldsymbol{n})^2 \, ds \bigg) \\
&\quad + \sum_{K \in \mathcal{K}_h} \delta \left( \nabla \cdot \boldsymbol{u}_h, \nabla \cdot \boldsymbol{u}_h \right)_K 
+ \sum_{K \in \mathcal{K}_h} \tau \big( -2 \nu \nabla \cdot \varepsilon (\boldsymbol{u}_h) + \boldsymbol{u}_h \cdot \nabla \boldsymbol{u}_h \\
&\quad + \nabla \mathscr{P} (\boldsymbol{u}_h) - \boldsymbol{f}, 
-2 \nu \nabla \cdot \varepsilon (\boldsymbol{u}_h) + \boldsymbol{u}_h \cdot \nabla \boldsymbol{u}_h + \nabla \mathscr{P} (\boldsymbol{u}_h) \big)_K \\
&= \mathcal{I} + \left( \boldsymbol{u}_h \cdot \nabla \boldsymbol{u}_h, \boldsymbol{u}_h\right) 
- \left(\boldsymbol{f}, \boldsymbol{u}_h\right) 
+ \sum_{K \in \mathcal{K}_h} \delta \left( \nabla \cdot \boldsymbol{u}_h, \nabla \cdot \boldsymbol{u}_h \right)_K \\
&\quad + \sum_{K \in \mathcal{K}_h} \tau \big( -2 \nu \nabla \cdot \varepsilon (\boldsymbol{u}_h) 
+ \boldsymbol{u}_h \cdot \nabla \boldsymbol{u}_h 
+ \nabla \mathscr{P} (\boldsymbol{u}_h) - \boldsymbol{f}, \\
&\quad -2 \nu \nabla \cdot \varepsilon (\boldsymbol{u}_h) + \boldsymbol{u}_h \cdot \nabla \boldsymbol{u}_h 
+ \nabla \mathscr{P} (\boldsymbol{u}_h) \big)_K.
\end{align*}
where 
\begin{align}\label{I}
\mathcal{I} : & = 2 \nu \left\| \varepsilon (\boldsymbol{u}_h) \right\|_{0, \Omega}^2 + \sum_{E \in \mathcal{E}_{\mathrm{Nav}}} \bigg( (1- \theta ) \int_E \mathscr{P}(\boldsymbol{u}_h) \left( \boldsymbol{n} \cdot \boldsymbol{u}_h \right) \, ds   - (1+ \theta) \int_{E}  \boldsymbol{n}^t(2 \nu \varepsilon  (\boldsymbol{u}_h) ) \boldsymbol{n} (\boldsymbol{n} \cdot \boldsymbol{u}_h)  \, ds  \nonumber  \\ & \quad + \int_{E} \beta \sum_{i=1}^{d-1} (\boldsymbol{\tau}^i \cdot \boldsymbol{u}_h)(\boldsymbol{\tau}^i \cdot \boldsymbol{u}_h)  \, ds  + \gamma \nu \int_{E} h_E^{-1} (\boldsymbol{u}_h \cdot \boldsymbol{n})(\boldsymbol{u}_h \cdot \boldsymbol{n})   \, ds \bigg).
\end{align}
The terms in \eqref{I} without $\theta$ can be bounded as follows:
\begin{align}\label{I1}
& 2 \nu \left\|\varepsilon ( \boldsymbol{u}_h ) \right\|_{0, \Omega}^2 + \sum_{E \in \mathcal{E}_{\mathrm{Nav}}}\left( \int_{E} \beta \sum_{i=1}^{d-1} (\boldsymbol{\tau}^i \cdot \boldsymbol{u}_h)(\boldsymbol{\tau}^i \cdot \boldsymbol{u}_h)  \, ds  + \frac{\gamma \nu} {h_E} \int_{E} (\boldsymbol{u}_h \cdot \boldsymbol{n})(\boldsymbol{u}_h \cdot \boldsymbol{n})  \, ds \right) \nonumber \\
\nonumber \\ 
& \geq 2 \nu \left\|\varepsilon( \boldsymbol{u}_h) \right\|_{0, \Omega}^2 + \sum_{E \in \mathcal{E}_{\mathrm{Nav}}}  \frac{\gamma \nu}{ h_E}\left\|\boldsymbol{u}_h \cdot \boldsymbol{n}\right\|_{0, E}^2  
\end{align}
We now bound the remaining terms in \eqref{I} from below. By choosing a positive parameter $\delta_1 > 0$ and applying trace, Cauchy-Schwarz and Young's inequality, we obtain
\begin{align}\label{I2}
2 \nu \left(\varepsilon \left(\boldsymbol{u}_h\right) \boldsymbol{n} \cdot \boldsymbol{n}, \boldsymbol{u}_h \cdot \boldsymbol{n}\right)_{\Gamma_{\mathrm{Nav}}} & \leq  2 \nu \sum_{E \in \mathcal{E}_{\mathrm{Nav}}}\left\|\varepsilon \left(\boldsymbol{u}_h \right) \boldsymbol{n}\right\|_{0, E}\left\|\boldsymbol{u}_h \cdot \boldsymbol{n}\right\|_{0, E} \nonumber \\
& = 2 \nu \sum_{E \in \mathcal{E}_{\mathrm{Nav}}} h_E^{1 / 2}\left\|\varepsilon \left(\boldsymbol{u}_h\right) \boldsymbol{n}\right\|_{0, E} h_E^{-1 / 2}\left\|\boldsymbol{u}_h \cdot \boldsymbol{n}\right\|_{0, E} \nonumber \\
& \leq  2 \nu \sum_{E \in \mathcal{E}_{\mathrm{Nav}}}\left(\frac{h_E \delta_1}{2}\left\|\varepsilon (\boldsymbol{u}_h) \right\|_{0, E}^2+\frac{h_E^{-1}}{2 \delta_1}\left\|\boldsymbol{u}_h \cdot \boldsymbol{n}\right\|_{0, E}^2\right) \nonumber \\
& \leq 2 \nu \delta_1 C_{\mathrm{tr}}^2 \sum_{K \in \mathcal{K}_h}\left\|\varepsilon (\boldsymbol{u}_h )\right\|_{0, K}^2+\frac{1}{\delta_1} \sum_{E \in \mathcal{E}_{\mathrm{Nav}}} \frac{\nu}{2 h_E}\left\|\boldsymbol{u}_h \cdot \boldsymbol{n}\right\|_{0, E}^2 .
\end{align}

On the other hand, for $\delta_2 > 0$, and using trace, inverse, and H\"older's inequalities, we have that
\begin{align}\label{I3}
\left(\mathscr{P} \left(\boldsymbol{u}_h\right) , \boldsymbol{u}_h \cdot \boldsymbol{n}\right)_{\Gamma_{\mathrm{Nav}}} & \leq \sum_{E \in \mathcal{E}_{\mathrm{Nav}}}\left\| \mathscr{P} \left(\boldsymbol{u}_h\right)  \right\|_{0, E}\left\|\boldsymbol{u}_h \cdot \boldsymbol{n}\right\|_{0, E} \nonumber \\
& =\sum_{E \in \mathcal{E}_{\mathrm{Nav}}} \nu^{-\frac{1}{2}} h_E^{1 / 2}\left\| \mathscr{P} \left(\boldsymbol{u}_h\right)  \right\|_{0, E} \nu^{\frac{1}{2}} h_E^{-1 / 2}\left\|\boldsymbol{u}_h \cdot \boldsymbol{n}\right\|_{0, E} \nonumber \\
& \leq \sum_{E \in \mathcal{E}_{\mathrm{Nav}}} \frac{h_E \delta_2}{2 \nu}\left\| \mathscr{P} \left(\boldsymbol{u}_h\right)  \right\|_{0, E}^2+\sum_{E \in \mathcal{E}_{\mathrm{Nav}}} \nu \frac{h_E^{-1}}{2 \delta_2}\left\|\boldsymbol{u}_h \cdot \boldsymbol{n}\right\|_{0, E}^2 \nonumber \\
& \leq \sum_{K \in \mathcal{K}_h} \frac{C_{\mathrm{tr}}^2 \delta_2}{2 \nu}\left\| \mathscr{P} \left(\boldsymbol{u}_h\right)  \right\|_{0, K}^2+\frac{1}{2 \delta_2} \sum_{E \in \mathcal{E}_{\mathrm{Nav}}} \frac{\nu}{h_E}\left\|\boldsymbol{u}_h \cdot \boldsymbol{n}\right\|_{0, E}^2  \nonumber \\
& \leq \frac{C_{\mathrm{in}}^2 C_{\mathrm{tr}}^2 \delta_2}{2} \sum_{K \in \mathcal{K}_h} \frac{h_K^2}{\nu }\left\|\nabla \mathscr{P} \left(\boldsymbol{u}_h\right) 
 \right\|_{0, K}^2+\frac{1}{ \delta_2} \sum_{E \in \mathcal{E}_{\mathrm{Nav}}} \frac{\nu}{2 h_E}\left\|\boldsymbol{u}_h \cdot \boldsymbol{n}\right\|_{0, E}^2.
\end{align}
Now, using \eqref{I1}, \eqref{I2}, \eqref{I3} in \eqref{I}, we get
$$
\begin{aligned}
 \mathcal{I}  & \geq  2 \nu \left\|\varepsilon (\boldsymbol{u}_h)\right\|_{0, \Omega}^2 -2 (1+ \theta)\nu \delta_1 C_{\mathrm{tr}}^2 \sum_{K \in \mathcal{K}_h}\left\|\varepsilon (\boldsymbol{u}_h) \right\|_{0, K}^2 + \sum_{E \in \mathcal{E}_{\mathrm{Nav}}} \left(\gamma - \frac{(1+ \theta)}{2 \delta_1} - \frac{(1- \theta)}{2 \delta_2} \right) \frac{ \nu}{ h_E}\left\|\boldsymbol{u}_h \cdot \boldsymbol{n}\right\|_{0, E}^2 \\ & \quad - \frac{C_{\mathrm{in}}^2 C_{\mathrm{tr}}^2 \delta_2 (1- \theta) }{2} \sum_{K \in \mathcal{K}_h} \frac{h_K^2}{\nu }\left\|\nabla \mathscr{P} \left(\boldsymbol{u}_h\right)  \right\|_{0, K}^2
\\ & \geq  \left( 2 \nu  - 2 (1+ \theta)\nu \delta_1 C_{\mathrm{tr}}^2 \right) \left\|\varepsilon (\boldsymbol{u}_h) \right\|_{0, \Omega}^2  + \sum_{E \in \mathcal{E}_{\mathrm{Nav}}} \left(\gamma - \frac{(1+ \theta)}{2 \delta_1} - \frac{(1- \theta)}{2 \delta_2} \right) \frac{ \nu}{ h_E}\left\|\boldsymbol{u}_h \cdot \boldsymbol{n}\right\|_{0, E}^2 \\& \quad - \frac{C_{\mathrm{in}}^2 C_{\mathrm{tr}}^2 \delta_2 (1- \theta) }{2} \sum_{K \in \mathcal{K}_h} \frac{h_K^2}{\nu }\left\|\nabla \mathscr{P} \left(\boldsymbol{u}_h\right)  \right\|_{0, K}^2
 \end{aligned}
$$
If $ \theta = 1 $ and we choose $\gamma > 1/\delta_1$ then we have 
$$
\begin{aligned}
 \mathcal{I} & \geq 2 \nu (1 - 2 \delta_1 C_{\mathrm{tr}}^2)  \left\| \varepsilon (\boldsymbol{u}_h ) \right\|_{0, \Omega}^2  + \sum_{E \in \mathcal{E}_{\mathrm{Nav}}} \left(\gamma - \frac{1}{\delta_1}\right) \frac{ \nu }{ h_E}\left\|\boldsymbol{u}_h \cdot \boldsymbol{n}\right\|_{0, E}^2 \\ &
 \geq \min\left\{2 (1 - 2 \delta_1 C_{\mathrm{tr}}^2) , \gamma - \frac{1}{\delta_1}\right\} \| \boldsymbol{u}_h \|^2_{1,\mathcal{K}_h}
 \end{aligned}
$$
Finally, setting $C_S := \min\left\{2 (1 - 2 \delta_1 C_{\mathrm{tr}}^2) , \gamma - \frac{1}{\delta_1}\right\}$, we get 
\begin{align*}
\left(\mathscr{N}\left(\boldsymbol{u}_h\right), \boldsymbol{u}_h\right) & \geq\, C_S \| \boldsymbol{u}_h \|^2_{1,\mathcal{K}_h}  + \left( \boldsymbol{u}_h \cdot \nabla \boldsymbol{u}_h, \boldsymbol{u}_h\right)  - \left(\boldsymbol{f}  , \boldsymbol{u}_h\right)  + \sum_{K \in \mathcal{K}_h} \delta \left( \nabla \cdot \boldsymbol{u}_h,  \nabla \cdot \boldsymbol{u}_h \right)_K \\ & \quad + \sum_{K \in \mathcal{K}_h} \tau\left(  -2 \nu \nabla \cdot  \varepsilon (\boldsymbol{u}_h) + \boldsymbol{u}_h \cdot \nabla \boldsymbol{u}_h + \nabla \mathscr{P} (\boldsymbol{u}_h)  , -2 \nu \nabla \cdot  \varepsilon (\boldsymbol{u}_h) + \boldsymbol{u}_h \cdot \nabla \boldsymbol{u}_h \right. \\  & \left. \quad + \nabla \mathscr{P} (\boldsymbol{u}_h) \right)_K  - \sum_{K \in \mathcal{K}_h} \tau \left(  \boldsymbol{f} ,  -2 \nu \nabla \cdot  \varepsilon (\boldsymbol{u}_h) +  \boldsymbol{u}_h  \cdot \nabla \boldsymbol{u}_h + \nabla \mathscr{P} (\boldsymbol{u}_h)\right)  \\&  \geq (C_S - \frac{1}{2}) R^2  + a_3^2 + (1 - \delta_3 )a_1^2 - \frac{1}{2} \left\|  \boldsymbol{f}  \right\|^2_{-1, \Omega} - \frac{1}{4 \delta_3}  \sum_{K \in \mathcal{K}_h} \tau \left\| \boldsymbol{f} \right\|_{0,\Omega}^2 +  \left(\boldsymbol{u}_h \cdot \nabla \boldsymbol{u}_h , \boldsymbol{u}_h\right) \\& \geq (C_S - \frac{1}{2}) R^2 + a_3^2 + (1- \delta_3 )a_1^2 - \min \left\{ \frac{1}{2}, \frac{1}{4 \delta_3}  \right\} a_2^2  +  \left(\boldsymbol{u}_h \cdot \nabla \boldsymbol{u}_h, \boldsymbol{u}_h\right)
\end{align*}
Choosing $\delta_2 = \frac{1}{4 C_{\mathrm{tr}}^2}$ and $\delta_3 = \frac{1}{M}$, with $M \gg 1$, we obtain
\begin{align}\label{choosen_values}
\left(\mathscr{N}\left(\boldsymbol{u}_h\right), \boldsymbol{u}_h\right) \geq (C_S - \frac{1}{2}) R^2 + a_3^2 + ( 1 - \frac{1}{M})a_1^2 - \min \left\{ \frac{1}{2}, \frac{M}{4}   \right\} a_2^2   +  \left(\boldsymbol{u}_h \cdot \nabla \boldsymbol{u}_h, \boldsymbol{u}_h\right)
\end{align}
Now, using the identity  
$\int_{\Omega} \left( \boldsymbol{u}_h \cdot \nabla \boldsymbol{u}_h \right) \boldsymbol{u}_h \, \mathrm{d}x = \frac{1}{2} \int_{\Gamma_{\mathrm{Nav}}} \left( \boldsymbol{u}_h \cdot \boldsymbol{u}_h \right) \left( \boldsymbol{u}_h \cdot \boldsymbol{n} \right) \, \mathrm{d}s$ 
$- \frac{1}{2} \int_{\Omega} \left(\nabla \cdot \boldsymbol{u}_h \right) \boldsymbol{u}_h \cdot \boldsymbol{u}_h \, \mathrm{d}x, $ gives
\begin{align}\label{NNNN}
\left(\mathscr{N}\left(\boldsymbol{u}_h\right), \boldsymbol{u}_h\right) & \geq (C_S - \frac{1}{2}) R^2 + a_3^2 + ( 1 - \frac{1}{M})a_1^2 - \min \left\{ \frac{1}{2}, \frac{M}{4}   \right\} a_2^2   + \frac{1}{2} \int_{\Gamma_{\mathrm{Nav}}} \left( \boldsymbol{u}_h \cdot \boldsymbol{u}_h \right) \left( \boldsymbol{u}_h \cdot \boldsymbol{n} \right) \, \mathrm{d}s \nonumber \\ & \quad 
- \frac{1}{2} \int_{\Omega} \left(\nabla \cdot \boldsymbol{u}_h \right) \boldsymbol{u}_h \cdot \boldsymbol{u}_h \, \mathrm{d}x \nonumber \\ & \geq C \left( R^2 + a_3^2 + a_1^2 - a_2^2 \right)  + \frac{1}{2} \int_{\Gamma_{\mathrm{Nav}}} \left( \boldsymbol{u}_h \cdot \boldsymbol{u}_h \right) \left( \boldsymbol{u}_h \cdot \boldsymbol{n} \right) \, \mathrm{d}s  - \frac{1}{2} \int_{\Omega} \left(\nabla \cdot \boldsymbol{u}_h \right) \boldsymbol{u}_h \cdot \boldsymbol{u}_h \, \mathrm{d}x
\end{align}
where $C$ is a generic positive constant, as the coefficient of all the terms are positive.

If $ \theta = 0 $ then we have 
$$
\begin{aligned}
 \mathcal{I} & \geq  2 \nu \left( 1- \delta_1 C_{\mathrm{tr}}^2 \right) \left\|\varepsilon (\boldsymbol{u}_h) \right\|_{0, \Omega}^2  + \sum_{E \in \mathcal{E}_{\mathrm{Nav}}} \left(\gamma - \frac{1}{2 \delta_1} - \frac{1}{2 \delta_2} \right) \frac{ \nu}{ h_E}\left\|\boldsymbol{u}_h \cdot \boldsymbol{n}\right\|_{0, E}^2  \\ & \quad - \frac{C_{\mathrm{in}}^2 C_{\mathrm{tr}}^2 \delta_2 }{2} \sum_{K \in \mathcal{K}_h} \frac{h_K^2}{\nu }\left\|\nabla \mathscr{P} \left(\boldsymbol{u}_h\right)  \right\|_{0, K}^2 
 \end{aligned}
$$
Choose $ \gamma > \frac{1}{2\delta_1} + \frac{1}{2\delta_2}  $ and $ \frac{1}{\delta_1} > C_{\mathrm{tr}}^2$ we get 
$$
\begin{aligned}
\mathcal{I} & \geq \min\left\{ 1-\delta_1 C_{\mathrm{tr}}^2 , \gamma - \frac{1}{2 \delta_1} - \frac{1}{2 \delta_2} \right\}\left( 2 \nu  \left\|\varepsilon (\boldsymbol{u}_h)\right\|_{0, \Omega}^2  + \sum_{E \in \mathcal{E}_{\mathrm{Nav}}}  \frac{ \nu}{ h_E}\left\|\boldsymbol{u}_h \cdot \boldsymbol{n}\right\|_{0, E}^2 \right) \\ & \quad - \frac{C_{\mathrm{in}}^2 C_{\mathrm{tr}}^2 \delta_2}{2}  \sum_{K \in \mathcal{K}_h} \frac{h_K^2}{\nu }\left\|\nabla \mathscr{P} \left(\boldsymbol{u}_h\right)  \right\|_{0, K}^2
\\& \quad \geq C_S  \| \boldsymbol{u}_h\|_{1,\mathcal{K}_h}^2 - \frac{C_{\mathrm{in}}^2 C_{\mathrm{tr}}^2 \delta_2}{2}  \sum_{K \in \mathcal{K}_h} \frac{h_K^2}{\nu }\left\|\nabla \mathscr{P} \left(\boldsymbol{u}_h\right)  \right\|_{0, K}^2
 \end{aligned}
$$
The subsequent steps involve calculations similar to those described earlier.
\begin{align}\label{incomplete_case}
 \left(\mathscr{N}\left(\boldsymbol{u}_h\right), \boldsymbol{u}_h\right) & \geq (C_S - \frac{1}{2}) R^2 + a_3^2 + ( 1 - \frac{1}{M})a_1^2 - \min \left\{ \frac{1}{2}, \frac{M}{4}   \right\} a_2^2   + \frac{1}{2} \int_{\Gamma_{\mathrm{Nav}}} \left( \boldsymbol{u}_h \cdot \boldsymbol{u}_h \right) \left( \boldsymbol{u}_h \cdot \boldsymbol{n} \right) \, \mathrm{d}s \nonumber \\ & \quad 
- \frac{1}{2} \int_{\Omega} \left(\nabla \cdot \boldsymbol{u}_h \right) \boldsymbol{u}_h \cdot \boldsymbol{u}_h \, \mathrm{d}x  - \frac{C_{\mathrm{in}}^2 C_{\mathrm{tr}}^2 \delta_2}{2}  \sum_{K \in \mathcal{K}_h} \frac{h_K^2}{\nu }\left\|\nabla \mathscr{P} \left(\boldsymbol{u}_h\right)  \right\|_{0, K}^2    
\end{align}
Now, if we take $r =\boldsymbol{J}_h\left(\boldsymbol{u}_h \cdot \boldsymbol{u}_h\right)$  in \eqref{ope_P}, use Cauchy-Schwarz's inequality, trace inequality, mesh regularity, and following closely what was done in \cite{MR2914281, MR3463441}, we get
\begin{equation}\label{a2}
\left|\left(\nabla \cdot \boldsymbol{u}_h, \boldsymbol{u}_h \cdot \boldsymbol{u}_h\right)\right| \leq C\left\{ R+a_1+a_2+a_4\right\}\left\{\sum_{K \in \mathcal{K}_h} h_K^2\left|\boldsymbol{u}_h \cdot \boldsymbol{u}_h\right|_{1, K}^2\right\}^{1 / 2}
\end{equation}
By applying Lemma \ref{local_inverse} with \( l = m = 0 \), \( p = \infty \), and \( 1 \leq q \leq \infty \), we obtain
$\|\boldsymbol{u}_h\|_{\infty, K} \leq C h_K^{-d / q} \|\boldsymbol{u}_h\|_{0, q, K}.$
Next, using the Sobolev embedding \( H^1(\Omega) \hookrightarrow L^q(\Omega) \), which holds for all \( 2 \leq q < \infty \) if \( d = 2 \) and for \( 2 \leq q \leq 6 \) if \( d = 3 \), we derive a similar result as in \cite[Theorem 3.4]{MR3463441}.
\begin{align}\label{a3}
\left|\boldsymbol{u}_h \cdot \boldsymbol{u}_h\right|_{1, K} &  \leq C h_K^{-d / q}\left|\boldsymbol{u}_h\right|_{1, K}\left|\boldsymbol{u}_h\right|_{1, \Omega}.
\end{align}
Using Cauchy-Schwarz and trace inequalities, we get
\begin{align}\label{a4}
\int_{\Gamma_{\mathrm{Nav}}} \left( \boldsymbol{u}_h \cdot \boldsymbol{u}_h \right) \left( \boldsymbol{u}_h \cdot \boldsymbol{n} \right)\ds & \leq \| \boldsymbol{u}_h \cdot    \boldsymbol{u}_h \|_{0, \Gamma_{\mathrm{Nav}}}  \| \boldsymbol{u}_h \cdot    \boldsymbol{n} \|_{0, \Gamma_{\mathrm{Nav}}}  \nonumber \\ & \leq
h_K \| \boldsymbol{u}_h \cdot    \boldsymbol{u}_h \|_{1,K} h_E^{-1/2} \| \boldsymbol{u}_h \cdot    \boldsymbol{n} \|_{0, \Gamma_{\mathrm{Nav}}} \nonumber \\ & \leq 
R ( \sum_{K \in \mathcal{K}_h}   h_K^2 \|\boldsymbol{u}_h \cdot    \boldsymbol{u}_h \|_{1,K}^2 )^{1/2}
\end{align}
and then from \eqref{NNNN}, \eqref{a2}, \eqref{a3} and \eqref{a4}, and by choosing $\kappa:=\frac{d}{q}$, it holds
$$
\begin{aligned}
\left(\mathscr{N}\left(\boldsymbol{u}_h\right), \boldsymbol{u}_h\right) & \gtrsim   R^2+ a_1^2+a_3^2 -a_2^2 + \frac{1}{2} \int_{\Gamma_{\mathrm{Nav}}} \left( \boldsymbol{u}_h \cdot \boldsymbol{u}_h \right) \left( \boldsymbol{u}_h \cdot \boldsymbol{n} \right)\ds -\frac{1}{2}\left(\nabla \cdot \boldsymbol{u}_h, \boldsymbol{u}_h \cdot \boldsymbol{u}_h\right)  \\
& \gtrsim  R^2+ a_1^2+a_3^2 + a_2^2-C \left\{ R+a_1+a_2 + a_4\right\} \left\{\sum_{K \in \mathcal{K}_h} h_K^{2-2 \kappa}\left|\boldsymbol{u}_h\right|_{1, K}^2\right\}^{\frac{1}{2}}\left|\boldsymbol{u}_h\right|_{1, \Omega}  \\
& \gtrsim  R^2+ a_1^2+a_3^2+ -a_2^2-C h^{1-\kappa}\left\{ R+a_1+a_2 \right\} R^2 \\
& \gtrsim  R^2+  a_1^2+a_3^2 -a_2^2-C h^{1-\kappa} R^3-\frac{1}{2} a_1^2 -\frac{1}{2}a_2^2 - C^2 h^{2(1-\kappa)} R^4  \\
& \gtrsim  R^2+  a_1^2+a_3^2 -a_2^2 -C h^{1-\kappa} R^3- C^2 h^{2(1-\kappa)} R^4,
\end{aligned}
$$
where $R:=\frac{1}{6 C h^{1-\kappa}}$ and $\tilde{C}=\frac{1}{12 C}$. 
Note that, for only $\theta = 1 $, assumption \eqref{assumption} leads to $a_2 \leq \frac{R}{2} $, thus
$$
\begin{aligned}
\left(\mathscr{N}\left(\boldsymbol{u}_h\right), \boldsymbol{u}_h\right) & \gtrsim R^2 +a_1^2 +a_3^2 -a_2^2 \gtrsim R^2 + a_1^2 + a_3^2 -a_2^2 \geq 0 .
\end{aligned}
$$
Therefore, by applying Brouwer's Fixed Point Theorem (see \cite[Chapter IV, Corollary 1.1]{MR851383}) , we conclude that there exists a function \( \boldsymbol{u}_h \in \mathrm{V}_h \) such that the norm \( \| \boldsymbol{u}_h \|_{1,\mathcal{K}_h} \leq R \) and the fixed point operator \( \mathscr{N}(\boldsymbol{u}_h) = \boldsymbol{0} \). This indicates that \( \boldsymbol{u}_h \) is a fixed point of the operator \( \mathscr{N} \) within the constrained norm.
\end{proof}
\end{theorem} 
\begin{remark}
Equation \eqref{incomplete_case} corresponds to \( \theta = 0 \). The term $
\frac{C_{\mathrm{in}}^2 C_{\mathrm{tr}}^2 \delta_2}{2} \sum_{K \in \mathcal{K}_h} \frac{h_K^2}{\nu} \left\|\nabla \mathscr{P}(\boldsymbol{u}_h) \right\|_{0, K}^2$
presents a challenge due to its negative coefficient, preventing us from proving existence for \( \theta = 0 \). Similar difficulties arise in the skew-symmetric case, complicates the existence proof under weaker assumptions, such as Brouwer's fixed-point theorem. However, we establish existence for \( \theta = 1 \) (the symmetric case) and later for all \( \theta \) values under stricter conditions, using Banach's fixed-point theorem.
\end{remark}
\subsection{ Existence and uniqueness of the discrete solution under strict conditions}
We prove the existence and uniqueness result for formulation \eqref{DSF} under the diffusion dominated assumption (i.e. $\nu$ large enough). We first write \eqref{DSF} as a fixed-point problem. To this end, we define $T_h: \mathrm{V}^{\prime} \times$ $\Pi \longrightarrow \mathrm{V}_h \times \Pi_h$, a discrete Stokes operator, which contains the linear terms of \eqref{DSF}, for each $(\boldsymbol{w}, r) \in \mathrm{V}^{\prime} \times \Pi$, it associates the unique solution $\left(\boldsymbol{u}_h, p_h\right) \in \mathrm{V}_h \times \Pi_h$ of
\begin{equation}\label{fixed_op}
\mathcal{A}_h^{\mathrm{S}}\left(\boldsymbol{u}_h, p_h ; \boldsymbol{v}_h, q_h\right) =\left\langle \boldsymbol{w}, \boldsymbol{v}_h\right\rangle+\left(r, q_h\right),
\end{equation}
for all $\left(\boldsymbol{v}_h, q_h\right) \in \mathrm{V}_h \times \Pi_h$, where $\mathcal{A}^{\mathrm{S}}(\cdot, \cdot)$ is given by 
$$
\begin{aligned}
\mathcal{A}_h^{\mathrm{S}}\left(\boldsymbol{u}_h, p_h ; \boldsymbol{v}_h, q_h\right):& = \sum_{K \in \mathcal{K}_{h}} \bigg(  2 \nu (\varepsilon (\boldsymbol{u}_h), \varepsilon ( \boldsymbol{v}_h)) - (p_h, \nabla \cdot \boldsymbol{v}_h)_{K} - (q_h, \nabla \cdot \boldsymbol{u}_h)_{K} \bigg)  + \sum_{E \in \mathcal{E}_{\mathrm{Nav}}} \\ 
&\quad \bigg( -\int_{E} \boldsymbol{n}^t(2 \nu \varepsilon (\boldsymbol{u}_h) - p_h I) \boldsymbol{n} (\boldsymbol{n} \cdot \boldsymbol{v}_h)  \, ds - \theta \int_{E} \boldsymbol{n}^t(2 \nu \varepsilon ( \boldsymbol{v}_h)  - q_h I) \boldsymbol{n} (\boldsymbol{n} \cdot \boldsymbol{u}_h)  \, ds \\ 
&\quad + \int_{E} \beta \sum_{i=1}^{d-1} (\boldsymbol{\tau}^i \cdot \boldsymbol{v}_h)(\boldsymbol{\tau}^i \cdot \boldsymbol{u}_h) \, ds  + \gamma \nu \int_{E} h_E^{-1} (\boldsymbol{u}_h \cdot \boldsymbol{n})(\boldsymbol{v}_h \cdot \boldsymbol{n})  \, ds \bigg)  \\ &\quad + \sum_{K \in \mathcal{K}_{h}} \int_K \tau \bigg(-2 \nu \nabla \cdot  \varepsilon (\boldsymbol{u}_h)  + \nabla p_h \bigg) \bigg(-2 \nu \nabla \cdot  \varepsilon (\boldsymbol{v}_h)  - \nabla q_h \bigg) \\& \quad   + \sum_{K \in \mathcal{K}_{h}} \delta  \int_K  \nabla \cdot \boldsymbol{u}_h \nabla \cdot \boldsymbol{v}_h
\end{aligned}
$$

Additionally, we introduce the mapping $G_h: H^2\left(\mathcal{K}_h\right)^2 \times H^1\left(\mathcal{K}_h\right) \longrightarrow \mathrm{V}_h \times \Pi_h$, which contains the nonlinear terms of \eqref{DSF} for a fixed point, where $\left(\boldsymbol{w}_h, r_h\right):=G_h(\boldsymbol{z}, t)$ solves
$$
\begin{aligned}
(\boldsymbol{w}_h, \boldsymbol{v}_h) + (r_h, q_h) & = -\left(\boldsymbol{f} - \boldsymbol{z} \cdot \nabla \boldsymbol{z}, \boldsymbol{v}_h\right)  - \sum_{K \in \mathcal{K}_h} \tau\left(\boldsymbol{f}  -  \boldsymbol{z} \cdot \nabla \boldsymbol{z} , -2 \nu \nabla \cdot  \varepsilon (\boldsymbol{v}_h) - \nabla q_h \right)_K \\
& \quad  -  \sum_{K \in \mathcal{K}_h} \tau\left(\boldsymbol{f}  + 2 \nu \nabla \cdot  \varepsilon (\boldsymbol{v}_h) -  \boldsymbol{z} \cdot \nabla \boldsymbol{z} - \nabla t, \boldsymbol{z} \cdot \nabla \boldsymbol{v}_h\right)_K,
\end{aligned}
$$
for all $\left(\boldsymbol{v}_h, q_h\right) \in \mathrm{V}_h \times \Pi_h$. Combining these operators, problem \eqref{DSF} is written as the following fixed point problem
\begin{equation}\label{TG}
-T_h G_h\left(\boldsymbol{u}_h, p_h\right) = \left(\boldsymbol{u}_h, p_h\right).
\end{equation}

Before proving the uniqueness result for problem \eqref{DSF}, we need to establish the well-posedness of problem \eqref{fixed_op}. This result is presented in the next lemma.
\begin{lemma}\label{lemma_4.8}
The mapping $T_h$  is well-defined for sufficiently large values of $\gamma$.
\end{lemma} 
\begin{proof}
Defining the mesh dependent norm
$$
\| (\boldsymbol{u}_h, p_h )\|_h^2 = \nu \|\varepsilon(\boldsymbol{u}_h) \|^2_{0,\Omega} + \sum_{E \in \mathcal{E}_{\mathrm{Nav}}} \frac{\nu}{h_E} \|\boldsymbol{u}_h \cdot \boldsymbol{n} \|_{0,E}^2 + \sum_{K \in \mathcal{K}_h} \left( \frac{h_K^2}{\nu} \|\nabla p_h \|^2_{0,K} + \| \delta^{1/2} \nabla \cdot \boldsymbol{u}_h \|^2_{0, K} \right),
$$
and it is proven in a similar way to the proof in Theorem \ref{existence_theorem} for sufficiently large $\gamma$ and $\forall (\boldsymbol{u}_h,p_h) \in \mathrm{V}_h \times \Pi_h$ 
\begin{align*}
\mathcal{A}_h^{\mathrm{S}}\left(\boldsymbol{u}_h, p_h ; \boldsymbol{u}_h, - p_h\right) & =  2 \nu (\varepsilon (\boldsymbol{u}_h), \varepsilon ( \boldsymbol{u}_h))  + \sum_{E \in \mathcal{E}_{\mathrm{Nav}}} \bigg( -(1+\theta)  \int_{E} \boldsymbol{n}^t 2 \nu \varepsilon (\boldsymbol{u}_h) \boldsymbol{n} (\boldsymbol{n} \cdot \boldsymbol{u}_h)  \, ds \\& \quad + (1- \theta)\int_E p_h (\boldsymbol{n} \cdot \boldsymbol{u}_h)  + \int_{E} \beta \sum_{i=1}^{d-1} (\boldsymbol{\tau}^i \cdot \boldsymbol{u}_h)^2 \, ds + \gamma \nu \int_{E} h_E^{-1} (\boldsymbol{u}_h \cdot \boldsymbol{n})^2  \, ds \bigg)  \\& \quad   + \sum_{K \in \mathcal{K}_{h}} \int_K \tau \left(-2 \nu \nabla \cdot  \varepsilon (\boldsymbol{u}_h)  + \nabla p_h \right)^2  + \sum_{K \in \mathcal{K}_{h}}   \int_K \delta ( \nabla \cdot \boldsymbol{u}_h )^2 
\end{align*}
Consider the following bound
\begin{align}\label{tau_bound}
\sum_{K \in \mathcal{K}_h} \| \tau^{1/2} 2 \nu \nabla \cdot \varepsilon (\boldsymbol{u}_h )\|_{0,K}^2 \leq \frac{m_K 2 \nu}{4}    \sum_{K \in \mathcal{K}_h} h_K^2 \| \nabla \cdot \varepsilon (\boldsymbol{u}_h) \|^2_{0,K} \leq \frac{2 \nu}{2} \| \varepsilon (\boldsymbol{u}_h) \|^2_{0,\Omega}
\end{align}
For \( \theta = 1 \). Using equations \eqref{I2}, \eqref{I3} (with \( \mathscr{P}(\boldsymbol{u}_h) \) replaced by \( p_h \)), and \eqref{tau_bound}, we get
\begin{align*}
    \mathcal{A}_h^{\mathrm{S}}\left(\boldsymbol{u}_h, p_h ; \boldsymbol{u}_h, - p_h\right) &\geq 2 \nu (\varepsilon (\boldsymbol{u}_h), \varepsilon ( \boldsymbol{u}_h)) 
    + \sum_{E \in \mathcal{E}_{\mathrm{Nav}}} \bigg( 
    - 4 \nu \delta_1 C_{\mathrm{tr}}^2 \sum_{K \in \mathcal{K}_h}\left\|\varepsilon (\boldsymbol{u}_h )\right\|_{0, K}^2 \\
    &\quad 
    - \frac{2}{\delta_1} \sum_{E \in \mathcal{E}_{\mathrm{Nav}}} \frac{\nu}{2 h_E}\left\|\boldsymbol{u}_h \cdot \boldsymbol{n}\right\|_{0, E}^2 + \gamma \nu \int_{E} h_E^{-1} (\boldsymbol{u}_h \cdot \boldsymbol{n})^2 \, ds \bigg) 
  \\ &\quad   + \sum_{K \in \mathcal{K}_{h}} \int_K \tau \left(-2 \nu \nabla \cdot  \varepsilon (\boldsymbol{u}_h) + \nabla p_h \right)^2 + \sum_{K \in \mathcal{K}_{h}} \delta \int_K ( \nabla \cdot \boldsymbol{u}_h )^2 \\
    &\geq 2 \nu (1-2 \delta_1 C_{\mathrm{tr}}^2) \left\|\varepsilon (\boldsymbol{u}_h )\right\|_{0, \Omega}^2 
    + \sum_{E \in \mathcal{E}_{\mathrm{Nav}}} \frac{\nu}{h_E} \left(\gamma - \frac{1}{\delta_1}\right) \left\|\boldsymbol{u}_h \cdot \boldsymbol{n}\right\|_{0, E}^2 \\
    &\quad + \sum_{K \in \mathcal{K}_h} \bigg( 
    \| \delta^{1/2} \nabla \cdot \boldsymbol{u}_h \|^2_{0, K} 
    + \| \tau^{1/2} 2 \nu \nabla \cdot \varepsilon (\boldsymbol{u}_h) \|_{0,K}^2 
    + \| \tau^{1/2} \nabla p_h \|_{0,K}^2 \\
    &\quad - 2 \| \tau^{1/2} 2 \nu \nabla \cdot \varepsilon (\boldsymbol{u}_h) \|_{0,K}^2 
    - \frac{1}{2} \| \tau^{1/2} \nabla p_h \|_{0,K}^2 \bigg) \\
    &\geq 2 \nu \left(\frac{1}{2}-2 \delta_1 C_{\mathrm{tr}}^2\right) \left\|\varepsilon (\boldsymbol{u}_h )\right\|_{0, \Omega}^2 
    + \sum_{E \in \mathcal{E}_{\mathrm{Nav}}} \frac{\nu}{h_E} \left(\gamma - \frac{1}{\delta_1}\right) \left\|\boldsymbol{u}_h \cdot \boldsymbol{n}\right\|_{0, E}^2 \\
    &\quad + \sum_{K \in \mathcal{K}_h} \bigg( 
    \| \delta^{1/2} \nabla \cdot \boldsymbol{u}_h \|^2_{0, K} 
    + \frac{1}{2} \| \tau^{1/2} \nabla p_h \|_{0,K}^2 \bigg) 
    \gtrsim \| (\boldsymbol{u}_h, p_h )\|_h^2.
\end{align*}
The result holds for sufficiently large values of $\gamma$, specifically when $\gamma > 4 C_{\mathrm{tr}}^2$.
Next, for $\theta = -1$ and use \eqref{I2}, \eqref{I3} (with \( \mathscr{P}(\boldsymbol{u}_h) \) replaced by \( p_h \)), and \eqref{tau_bound}, we have
\begin{align*}
\mathcal{A}_h^{\mathrm{S}}\left(\boldsymbol{u}_h, p_h ; \boldsymbol{u}_h, - p_h\right)  & \geq   \nu \| \varepsilon (\boldsymbol{u}_h) \|^2_{0,\Omega} +  \frac{\nu}{h_E} \sum_{E \in \mathcal{E}_{\mathrm{Nav}}} \left( \gamma - \frac{1}{\delta_2} \right) \left\|\boldsymbol{u}_h \cdot \boldsymbol{n}\right\|_{0, E}^2 + \left( \frac{m_K}{16} - C_{\mathrm{in}}^2 C_{\mathrm{tr}}^2 \delta_2 \right) \\ & \quad   \sum_{K \in \mathcal{K}_h}  \frac{h_K^2}{\nu }\left\|\nabla p_h \right\|_{0, K}^2 + \sum_{K \in \mathcal{K}_h}  \| \delta^{1/2} \nabla \cdot \boldsymbol{u}_h \|^2_{0, K} \gtrsim \| (\boldsymbol{u}_h, p_h )\|_h^2 
\end{align*}
The result holds for sufficiently large values of $\gamma$, specifically when $\gamma > 1/\delta_2$.

Similarly, for $\theta = 0$ and use \eqref{I2}, \eqref{I3} (with \( \mathscr{P}(\boldsymbol{u}_h) \) replaced by \( p_h \)), and \eqref{tau_bound}, we get
\begin{align*}
\mathcal{A}_h^{\mathrm{S}}\left(\boldsymbol{u}_h, p_h ; \boldsymbol{u}_h, - p_h\right)  & \geq  2 \nu \left(\frac{1}{2} - \delta_1 C_{\mathrm{tr}}^2 \right) \| \varepsilon (\boldsymbol{u}_h) \|^2_{0,\Omega} +  \frac{\nu}{h_E} \left(  \gamma - \frac{1}{2 \delta_1} - \frac{1}{2 \delta_2} \right) \sum_{E \in \mathcal{E}_{\mathrm{Nav}}}  \left\|\boldsymbol{u}_h \cdot \boldsymbol{n}\right\|_{0, E}^2 \\ & \quad  + \left( \frac{m_K}{16} - \frac{C_{\mathrm{in}}^2 C_{\mathrm{tr}}^2 \delta_2}{2} \right) \sum_{K \in \mathcal{K}_h} \frac{h_K^2}{\nu }\left\|\nabla p_h \right\|_{0, K}^2 + \sum_{K \in \mathcal{K}_h}  \| \delta^{1/2} \nabla \cdot \boldsymbol{u}_h \|^2_{0, K} \gtrsim \| (\boldsymbol{u}_h, p_h )\|_h^2 
\end{align*}
The result holds for sufficiently large values of $\gamma$, specifically when $\gamma > \frac{1}{2 \delta_1} + \frac{1}{2 \delta_2}$. 
Hence, the problem \eqref{fixed_op} is well-posed for sufficiently large $\gamma$, and the operator $T_h$ is well-defined.
\end{proof}
\begin{lemma}\label{continuity_T}
The operator $T_h$ is continuous for sufficiently large values of $\gamma$. More precisely, there exists $C>0$, independent of $h$ and $\nu$, such that
\begin{align*}
\|T_h(\boldsymbol{w}, r)\|\leq C \left(\sqrt{\nu}+ \frac{h } {\sqrt{\nu}}\right)^2 \| (\boldsymbol{w}, r)\|_{\left(\mathrm{V}_h \times \Pi_h\right)^{\prime}},
\end{align*}
for all $(\boldsymbol{w}, r) \in(\mathrm{V} \times \Pi)^{\prime}$.
\end{lemma}
\begin{proof}
 The proof follows standard arguments, but we present it here for completeness. Let $\left(\boldsymbol{u}_h, p_h\right)=T_h(\boldsymbol{w}, r)$. From Lemma \ref{lemma_4.8} we see that 

\begin{align}\label{main_bound}
\left\|\left(\boldsymbol{u}_h, p_h\right)\right\|_h^2 & \leq \mathcal{A}_h^{\mathrm{S}}\left(\boldsymbol{u}_h, p_h ; \boldsymbol{u}_h, - p_h\right)  =\left\langle\boldsymbol{w}, \boldsymbol{u}_h\right\rangle+\left(r, - p_h\right)\nonumber \\
& \leq\|\boldsymbol{w}\|_{\mathrm{V}_h^{\prime}}\left|\boldsymbol{u}_h\right|_{1, \Omega}+\|r\|_{\Pi_h^{\prime}}\left\|p_h\right\|_{0, \Omega} .
\end{align}

To bound the $L^2(\Omega)$ norm of $p_h$, let $\boldsymbol{z} \in \mathrm{V}$ be such that 

$$
\beta_1 \left\|p_h\right\|_{0, \Omega}|\boldsymbol{z}|_{1, \Omega} \leq\left(p_h, \nabla \cdot \boldsymbol{z}\right).
$$
Let $\boldsymbol{z}_h$ denote the Cl\'ement interpolant of $\boldsymbol{z}$. By integrating by parts on the term $\left(p_h, \nabla \cdot \boldsymbol{z}\right)$ and utilizing the fact that $\left(\boldsymbol{u}_h, p_h\right)$ is the solution of \eqref{TG}, along with the Cl\'ement interpolant estimates \cite{MR400739}, we arrive at
\begin{align}\label{lemma_4.9}
\beta_1 \left\|p_h\right\|_{0, \Omega}|\boldsymbol{z}|_{1, \Omega} & \leq \left(p_h, \nabla \cdot\left(\boldsymbol{z}-\boldsymbol{z}_h\right)\right)+\left(p_h, \nabla \cdot \boldsymbol{z}_h\right) \nonumber  \\ &
=  2 \nu \left( \varepsilon ( \boldsymbol{u}_h), \varepsilon (\boldsymbol{z}_h)  \right)  - \sum_{K \in \mathcal{K}_h} \int_K \nabla p_h (\boldsymbol{z}-\boldsymbol{z}_h) + \sum_{E \in \Gamma_{\mathrm{Nav}}} \left[ \int_E (p_h \cdot \boldsymbol{n} )(\boldsymbol{z} - \boldsymbol{z}_h)  \right. \nonumber \\ & \quad \left. - \int_E  \boldsymbol{n}^t ( 2 \nu \varepsilon (\boldsymbol{u}_h )) \boldsymbol{n} (\boldsymbol{n} \cdot \boldsymbol{z}_h )  - \theta \int_E \boldsymbol{n}^t (2 \nu \varepsilon (\boldsymbol{z}_h) ) \boldsymbol{n} ( \boldsymbol{n} \cdot \boldsymbol{u}_h) + \int_E  p_h ( \boldsymbol{n} \cdot \boldsymbol{z}_h) \right. \nonumber \\ & \quad  \left. + \int_E \beta \sum_{i}^{d-1} (\tau^i \cdot \boldsymbol{z}_h) (\tau^i \cdot \boldsymbol{u}_h) + \gamma \nu \int_E h_E^{-1} (\boldsymbol{u}_h \cdot \boldsymbol{n})  (\boldsymbol{z}_h \cdot \boldsymbol{n}) \right]\nonumber  \\ & \quad + \sum_{K \in \mathcal{K}_h} \left[ \int_K \tau ( -2 \nu \nabla \cdot \varepsilon ( \boldsymbol{u}_h) + \nabla p_h ) ( - 2 \nu \nabla \cdot \varepsilon (z_h) ) + \delta \int_K (\nabla \cdot \boldsymbol{u}_h) (\nabla \cdot \boldsymbol{z}_h) \right] \nonumber \\ & \quad - \langle 
 \boldsymbol{w}, \boldsymbol{z}_h \rangle \nonumber  \\ & \leq  
 2 \nu \| \varepsilon (\boldsymbol{u}_h)\|_{0,\Omega} \| \varepsilon (\boldsymbol{z}_h)\|_{0,\Omega} + \sum_{K \in \mathcal{K}_h} h_K \| \nabla p_h \|_{0,K} | \boldsymbol{z} |_{1,\omega_K} \nonumber \\ & \quad  + \sum_{E \in \Gamma_{\mathrm{Nav}}} h_E^{1/2} \|p_h \boldsymbol{n} \|_{0,E} | \boldsymbol{z}|_{1,\omega_E} + \|\boldsymbol{w}\|_{\mathrm{V}_h^{\prime}} |\boldsymbol{z}_h|_{1,\Omega} + \mathcal{R}_1 +\mathcal{R}_2
\end{align}
where 
\begin{align*}
    \mathcal{R}_1  & =  \sum_{E \in \Gamma_{\mathrm{Nav}}} \left[ - \int_E  \boldsymbol{n}^t ( 2 \nu \varepsilon (\boldsymbol{u}_h )) \boldsymbol{n} (\boldsymbol{n} \cdot \boldsymbol{z}_h )  - \theta \int_E \boldsymbol{n}^t (2 \nu \varepsilon (\boldsymbol{z}_h) ) \boldsymbol{n} ( \boldsymbol{n} \cdot \boldsymbol{u}_h) + \int_E  p_h ( \boldsymbol{n} \cdot \boldsymbol{z}_h) \right. \\ & \quad  \left. + \int_E \beta \sum_{i}^{d-1} (\tau^i \cdot \boldsymbol{z}_h) (\tau^i \cdot \boldsymbol{u}_h) + \gamma \nu \int_E h_E^{-1} (\boldsymbol{u}_h \cdot \boldsymbol{n}) (\boldsymbol{z}_h \cdot \boldsymbol{n}) \right] \\ 
   \mathcal{R}_2 & =  \sum_{K \in \mathcal{K}_h} \left[ \int_K \tau ( -2 \nu \nabla \cdot \varepsilon ( \boldsymbol{u}_h) + \nabla p_h ) ( - 2 \nu \nabla \cdot \varepsilon (z_h) ) + \delta \int_K (\nabla \cdot \boldsymbol{u}_h) (\nabla \cdot \boldsymbol{z}_h) \right] 
\end{align*}
First, we compute a bound for $\mathcal{R}_1$. Using the Poincaré inequality, we have 
\begin{align*}
    |\mathcal{R}_1 | & \leq  2 \nu \sum_{E \in \mathcal{E}_{\mathrm{Nav}}} h_E^{1 / 2}\left\|\varepsilon \left(\boldsymbol{u}_h\right) \boldsymbol{n}\right\|_{0, E} h_E^{-1 / 2}\left\|\boldsymbol{z}_h \cdot \boldsymbol{n}\right\|_{0, E}  + \theta 2 \nu \sum_{E \in \mathcal{E}_{\mathrm{Nav}}} h_E^{1 / 2}\left\|\varepsilon \left(\boldsymbol{z}_h\right) \boldsymbol{n}\right\|_{0, E} \\
    & \quad h_E^{-1 / 2}\left\|\boldsymbol{u}_h \cdot \boldsymbol{n}\right\|_{0, E}  + \sum_{E \in \mathcal{E}_{\mathrm{Nav}}} \nu^{-\frac{1}{2}} h_E^{1 / 2}\left\| p_h \right\|_{0, E} \nu^{\frac{1}{2}} h_E^{-1 / 2}\left\|\boldsymbol{z}_h \cdot \boldsymbol{n}\right\|_{0, E} \\
    & \quad + \beta \| \varepsilon( \boldsymbol{u}_h) \|_{0,\Omega}  \| \varepsilon( \boldsymbol{z}_h) \|_{0,\Omega} + \sum_{E \in \Gamma_{\mathrm{Nav}}} \frac{\gamma \nu}{h_E}  \| \boldsymbol{u}_h \cdot \boldsymbol{n} \|_{0,E}  \| \boldsymbol{z}_h \cdot \boldsymbol{n} \|_{0,E}
\end{align*}
Similarly, for $\mathcal{R}_2$ we have 
\begin{align*}
    |\mathcal{R}_2| & \leq \sum_{K \in \mathcal{K}_h} \bigg(\|\tau^{1/2} (-2 \nu \nabla \cdot \varepsilon (\boldsymbol{u}_h) + \nabla p_h )\|_{0,K} \| \tau^{1/2} (-2 \nu \nabla \cdot \varepsilon (\boldsymbol{z}_h)) \|_{0,K} \\ & \quad + \| \delta^{1/2} \nabla \cdot \boldsymbol{u}_h \|_{0,K} \| \delta^{1/2} \nabla \cdot \boldsymbol{z}_h \|_{0,K}   \bigg)
\end{align*}
Next, using the generalized Poincar\'{e} inequality, we obtain 
$ \| \delta^{1/2} \nabla \cdot \boldsymbol{z}_h \|_{0,K}  \leq C \frac{h_K } {\sqrt{\nu}}
 \|\boldsymbol{z}_h\|_{1,\omega_K} $ and the bound for $\mathcal{R}_2$ becomes 
\begin{align*}
     \mathcal{R}_2 & \leq \sum_{K \in \mathcal{K}_h} \left(\|\tau^{1/2} (-2 \nu \nabla \cdot \varepsilon (\boldsymbol{u}_h) + \nabla p_h )\|_{0,K} \| \tau^{1/2} (-2 \nu \nabla \cdot \varepsilon (\boldsymbol{z}_h)) \|_{0,K} + \| \delta^{1/2} \nabla \cdot \boldsymbol{u}_h \|_{0,K} \frac{h_K } {\sqrt{\nu}} \|\boldsymbol{z}_h\|_{1,\omega_K}   \right)
\end{align*}
Substituting the bounds for $\mathcal{R}_1$ and $\mathcal{R}_2$ in \eqref{lemma_4.9} and applying trace and inverse inequalities, we obtain 
\begin{align}\label{p_inequality}
\beta\left\|p_h\right\|_{0, \Omega}|\boldsymbol{z}|_{1, \Omega} & \leq  C \left[ \nu \|\varepsilon(\boldsymbol{u}_h) \|^2_{0,\Omega} + \sum_{E \in \mathcal{E}_{\mathrm{Nav}}} \frac{\nu}{h_E} \|\boldsymbol{u}_h \cdot \boldsymbol{n} \|_{0,E}^2 + \sum_{K \in \mathcal{K}_h} \left( \frac{h_K^2}{\nu} \|\nabla p_h \|^2_{0,K} + \| \delta^{1/2} \nabla \cdot \boldsymbol{u}_h \|^2_{0, K} \right)  \right. \nonumber \\ & \left. \quad + \|\boldsymbol{w}_h \|^2_{\mathrm{V}_h^{\prime}}\right]^{1/2}  \left( \sqrt{\nu} +\frac{h } {\sqrt{\nu}}\right) \|\boldsymbol{z}\|_{1,\Omega} \nonumber  \\
\left\|p_h\right\|_{0, \Omega} & \leq C \left( \sqrt{\nu} + \frac{h} {\sqrt{\nu}}\right) \left\{\left\|\left(\boldsymbol{u}_h, p_h\right)\right\|_h^2+\|\boldsymbol{w}\|_{\boldsymbol{V}_h^{\prime}}^2\right\}^{1 / 2} .
\end{align}
Then, using \eqref{p_inequality} in \eqref{main_bound}, and $a b \leq a^2+\frac{1}{4} b^2$, we arrive at
$$
\begin{aligned}
\left\|\left(\boldsymbol{u}_h, p_h\right)\right\|_h^2 & \leq C\left(\|\boldsymbol{w}\|_{\boldsymbol{V}_h^{\prime}}^2+\left(\sqrt{\nu}+ \frac{h } {\sqrt{\nu}}\right)^2\|r\|_{\Pi_h^{\prime}}^2\right)  \leq C \left(\sqrt{\nu}+ \frac{h } {\sqrt{\nu}}\right)^2\left(\|\boldsymbol{w}\|_{\boldsymbol{V}_h^{\prime}}^2+\|r\|_{\Pi_h^{\prime}}^2\right)
\end{aligned}
$$
and then replacing this in \eqref{p_inequality} it holds that
$$
\begin{aligned}
\left\|p_h\right\|_{0, \Omega} \leq C\left(\sqrt{\nu}+ \frac{h } {\sqrt{\nu}}\right)^2\left(\|\boldsymbol{w}\|_{\boldsymbol{V}_h^{\prime}}+\|r\|_{\Pi_h^{\prime}}\right) .
\end{aligned}
$$
To conclude the proof, we note that 
$\left\|\left(\boldsymbol{u}_h, p_h\right)\right\| \leq \left(\left\|\left(\boldsymbol{u}_h, p_h\right)\right\|_h^2 + \left\|p_h\right\|_{0, \Omega}^2\right)^{1 / 2}$,
along with the inequality
\[
\|\boldsymbol{w}\|_{\boldsymbol{V}_h^{\prime}} + \|r\|_{\Pi_h^{\prime}} \leq C\|(\boldsymbol{w}, r)\|_{\left(\boldsymbol{V}_h \times \Pi_h\right)^{\prime}}.
\]
 \end{proof}
We are ready to prove the uniqueness result.
\begin{lemma}
For \( \theta \in \{-1, 0, 1\} \), and for \(\nu\) and \(\gamma\) sufficiently large such that  
\begin{equation}
    \frac{1}{\sqrt{\nu}}\left\{2+\frac{5}{\sqrt{\nu}}+\frac{h}{\sqrt{\nu}} \|\boldsymbol{f}\|_{0, \Omega} \right\}\left(1+\frac{h}{\nu}\right)^2 < \frac{1}{C},
\end{equation}
where \( C \) is a positive constant independent of \( h \), the solution to problem \eqref{fixed_op} exists and is unique.
  \begin{proof}
First, observe that a solution of \eqref{fixed_op} is a fixed point of the operator $-T_h G_h$ using \eqref{TG}. Thereby, the proof reduces to show that the operator $-T_h G_h$ is a strict contraction in the ball $$B:=\left\{\left(\boldsymbol{v}_h, q_h\right) \in \mathrm{V}_h \times \Pi_h:\|\left(\boldsymbol{v}_h, q_h\right)\| \leq 1\right\}$$ and use Banach's fixed point Theorem.

Let $\left(\boldsymbol{u}_h, p_h\right),\left(\boldsymbol{v}_h, q_h\right) \in B$. Using Lemma \ref{continuity_T}, the definition of operators $T_h$ and $G_h$, it holds
\begin{equation}\label{main_eq}
\begin{aligned}
&\|T_h G_h\left(\boldsymbol{u}_h, p_h\right) - T_h G_h\left(\boldsymbol{v}_h, q_h\right)\| 
= \|T_h\left(G_h\left(\boldsymbol{u}_h, p_h\right) - G_h\left(\boldsymbol{v}_h, q_h\right)\right)\| \\
& \leq C \left(\sqrt{\nu}+ \frac{h } {\sqrt{\nu}}\right)^2 \sup_{\left\|(\boldsymbol{w}_h, t_h)\right\| \leq 1} \Bigg\{ 
    \left( \boldsymbol{u}_h  \cdot  \nabla \boldsymbol{u}_h - \boldsymbol{v}_h  \cdot \nabla \boldsymbol{v}_h, \boldsymbol{w}_h \right) \\
& \quad - \sum_{K \in \mathcal{K}_h} \tau\left( \boldsymbol{f}  + 2 \nu \nabla \cdot  \varepsilon (\boldsymbol{u}_h) - \boldsymbol{u}_h \cdot \nabla \boldsymbol{u}_h  - \nabla p_h, \boldsymbol{u}_h \cdot \nabla \boldsymbol{w}_h  \right)_K \\
& \quad - \sum_{K \in \mathcal{K}_h} \tau\left( -\boldsymbol{f}  - 2 \nu \nabla \cdot  \varepsilon (\boldsymbol{v}_h) + \boldsymbol{v}_h  \cdot \nabla \boldsymbol{v}_h + \nabla q_h, \boldsymbol{v}_h \cdot  \nabla \boldsymbol{w}_h  \right)_K \\
& \quad +  \sum_{K \in \mathcal{K}_h} \tau\left( \left(\nabla \boldsymbol{v}_h \right) \boldsymbol{v}_h - \boldsymbol{u}_h \cdot \nabla \boldsymbol{u}_h , -2 \nu \nabla \cdot \varepsilon (\boldsymbol{w}_h) - \nabla t_h \right) 
\Bigg\} \\
& = C \left(\sqrt{\nu}+ \frac{h } {\sqrt{\nu}}\right)^2 \sup_{\left\|(\boldsymbol{w}_h, t_h)\right\| \leq 1} \left\{ \mathrm{I} + \mathrm{II} + \mathrm{III} + \mathrm{IV} \right\} ,
\end{aligned}
\end{equation}

where
\begin{align*}
    \mathrm{I} &= \left( \boldsymbol{u}_h  \cdot  \nabla \boldsymbol{u}_h - \boldsymbol{v}_h  \cdot \nabla \boldsymbol{v}_h, \boldsymbol{w}_h \right)\\
     \mathrm{II} &= - \sum_{K \in \mathcal{K}_h} \tau\left(\boldsymbol{f} + 2 \nu \nabla \cdot \varepsilon ( \boldsymbol{u}_h) - \boldsymbol{u}_h \cdot \nabla \boldsymbol{u}_h  - \nabla p_h, \left(\boldsymbol{u}_h - \boldsymbol{v}_h \right) \cdot \nabla \boldsymbol{w}_h  \right)_K, \\
    \mathrm{III} &= -  \sum_{K \in \mathcal{K}_h} \tau\left( 2 \nu \left( \nabla \cdot \varepsilon  (\boldsymbol{u}_h) - \nabla \cdot \varepsilon  (\boldsymbol{v}_h) \right) - \left( \boldsymbol{u}_h \cdot \nabla \boldsymbol{u}_h  -  \boldsymbol{v}_h \cdot \nabla \boldsymbol{v}_h \right) - \left(\nabla p_h - \nabla q_h \right), \boldsymbol{v}_h \cdot \nabla \boldsymbol{w}_h  \right)_K \\
    \mathrm{IV} &= \sum_{K \in \mathcal{K}_h} \tau\left( \boldsymbol{v}_h \cdot \nabla \boldsymbol{v}_h  - \boldsymbol{u}_h \cdot \nabla \boldsymbol{u}_h , -2 \nu \nabla \cdot \varepsilon  (\boldsymbol{w}_h) - \nabla t_h \right)_K.
\end{align*}
To bound $\mathrm{I}$, we use that Lemma \ref{lemma2.1}, and we have 
\begin{equation}\label{bound1}  
\mathrm{I} \leq \frac{C}{\nu \sqrt{\nu}}\|\left(\boldsymbol{u}_h, p_h\right)-\left(\boldsymbol{v}_h, q_h\right) \|
  \|\left(\boldsymbol{w}_h, t_h\right)\| .
\end{equation}

To bound  $\mathrm{II}, \mathrm{III} \text{ and } \mathrm{IV}$, we use the same arguments as in \cite{MR2914281, MR3463441}. Thus we obtain
\begin{align}
\mathrm{II} & \leq \frac{C}{\nu^2}\left\{ h \|\boldsymbol{f}\|_{0, \Omega}+ \frac{1}{\nu}+\sqrt{\nu}\right\}\|\left(\boldsymbol{u}_h, p_h\right)-\left(\boldsymbol{v}_h, q_h\right)\| \|\left(\boldsymbol{w}_h, t_h\right)\|, \label{bound2} \\
 \mathrm{ III } & \leq \frac{C}{\nu^2 \sqrt{\nu}}\left\{\frac{2}{\nu}+1\right\}\|\left(\boldsymbol{u}_h, p_h\right)-\left(\boldsymbol{v}_h, q_h\right)\|
\|\left(\boldsymbol{w}_h, t_h\right)\|, \label{bound3} \\
\mathrm{IV} & \leq \frac{C}{\nu^2} \| \left(\boldsymbol{u}_h, p_h\right)-\left(\boldsymbol{v}_h, q_h\right)\|
\|\left(\boldsymbol{w}_h, t_h\right)\| . \label{bound4}
\end{align}
Collecting the bounds from \eqref{bound1}, \eqref{bound2}, \eqref{bound3}, and \eqref{bound4}, and substituting them into \eqref{main_eq} yields
$$  
\begin{aligned}
\|T_h G_h\left(\boldsymbol{u}_h, p_h\right)-T_h G_h\left(\boldsymbol{v}_h, q_h\right) 
\|
& \leq \frac{C}{\sqrt{\nu}}\left\{2+\frac{5}{\sqrt{\nu}}+\frac{h}{\sqrt{\nu}} \|\boldsymbol{f}\|_{0, \Omega} \right\}\left(1+\frac{h}{\nu}\right)^2
\|\left(\boldsymbol{u}_h, p_h\right)-\left(\boldsymbol{v}_h, q_h\right)\|,
\end{aligned}
$$
and the result follows under the assumption that viscosity $\nu$ is such that 
$ \frac{C}{\sqrt{\nu}}\left\{2+\frac{5}{\sqrt{\nu}}+\frac{h}{\sqrt{\nu}} \|\boldsymbol{f}\|_{0, \Omega} \right\}\left(1+\frac{h}{\nu}\right)^2<1$.
\end{proof}
\end{lemma}
\subsection{Stability analysis}
\begin{theorem}\label{stability} 
Assume that  $\|\boldsymbol{u}_h \|_{L^4(\Omega)} < \hat{C} $ for some constant $\hat{C}>0$. For \( \theta \in \{-1, 0, 1\} \) and sufficiently large \(\gamma\), there exists a positive constant \( C_S(\gamma) \), independent of \( h \) and $\nu$, such that:
    \begin{align*}
    \mathcal{A}_h^{\mathrm{NS}}(\boldsymbol{u}_h, p_h; \boldsymbol{u}_h, -p_h ) & \geq C_S \left( \nu \| \varepsilon (\boldsymbol{u}_h ) \|_{0,\Omega}^2 + \sum_{E \in \mathcal{E}_{\mathrm{Nav}}} \frac{\nu}{h_E} \| \boldsymbol{u}_h \cdot \boldsymbol{n} \|_{0,E}^2 + \| \tau^{1/2}  \boldsymbol{u}_h \cdot \nabla \boldsymbol{u}_h  \|_{0,\Omega}^2 \right. \\&  \quad \left.  + \sum_{K \in \mathcal{K}_h} \frac{h_K^2}{\nu} \| \nabla p_h\|_{0,K}^2 + \| \delta^{1/2} \nabla \cdot \boldsymbol{u}_h \|_{0,\Omega}^2 \right), 
    \end{align*}
    $ \forall ~ (\boldsymbol{u}_h, p_h) \in \mathrm{V}_h \times \Pi_h $. The bounds on the parameter \(\gamma\) depend only on the constants associated with trace and inverse inequalities.
    \begin{proof}
      Substitute $(\boldsymbol{v}_h, q_h )= (\boldsymbol{u}_h, - p_h )$ in the formulation \eqref{DSF}, we have 
  \begin{align}\label{stability_formulation}
\mathcal{A}_h^{\mathrm{NS}}(\boldsymbol{u}_h, p_h; \boldsymbol{u}_h, -p_h ) & = \sum_{K \in \mathcal{K}_h} \bigg[ 2 \nu \left( \varepsilon (\boldsymbol{u}_h) ,  \varepsilon (\boldsymbol{u}_h) \right) + \left( \boldsymbol{u}_h \cdot \nabla \boldsymbol{u}_h, \boldsymbol{u}_h \right) \bigg]
+ \sum_{E \in \mathcal{E}_{\mathrm{Nav}}} \bigg[ - \int_E \boldsymbol{n}^t (2 \nu \varepsilon (\boldsymbol{u}_h) - p_h \boldsymbol{I}) \nonumber \\ & \quad  \boldsymbol{n} (\boldsymbol{n} \cdot \boldsymbol{u}_h )  \, ds  - \theta \int_E \boldsymbol{n}^t (2 \nu \varepsilon (\boldsymbol{u}_h) + p_h \boldsymbol{I}) \boldsymbol{n} (\boldsymbol{n} \cdot \boldsymbol{u}_h )  \, ds + \beta \sum_{i=1}^{d-1} \int_E ( \tau^i \cdot \boldsymbol{u}_h)^2   \, ds \nonumber \\ & \quad + \frac{\gamma \nu}{h_E} \int_E ( \boldsymbol{u}_h \cdot \boldsymbol{n})^2   \, ds  \bigg] 
+ \sum_{K \in \mathcal{K}_h} \| \tau^{1/2} (-2 \nu \nabla \cdot \varepsilon(\boldsymbol{u}_h) + \boldsymbol{u}_h \cdot \nabla \boldsymbol{u}_h + \nabla p_h) \|_{0,K}^2  \nonumber \\ & \quad + \sum_{K \in \mathcal{K}_h}  \| \delta^{1/2} \nabla \cdot \boldsymbol{u}_h \|^2_{0,K}  \nonumber \\& 
=  2 \nu \| \varepsilon (\boldsymbol{u}_h) \|^2_{0,\Omega} + \left( \boldsymbol{u}_h \cdot \nabla \boldsymbol{u}_h, \boldsymbol{u}_h \right) 
+ \sum_{E \in \mathcal{E}_{\mathrm{Nav}}} \bigg[ - (1+ \theta) \int_E \boldsymbol{n}^t 2 \nu \varepsilon (\boldsymbol{u}_h)  \boldsymbol{n} (\boldsymbol{n} \cdot \boldsymbol{u}_h )  \, ds \nonumber \\ & \quad + (1- \theta) \int_E p_h (\boldsymbol{n} \cdot \boldsymbol{u}_h )  \, ds   + \beta \sum_{i=1}^{d-1}\| \tau^i \cdot \boldsymbol{u}_h \|^2_{0,E} + \frac{\gamma \nu}{h_E} \|\boldsymbol{u}_h \cdot \boldsymbol{n} \|_{0,E}^2  \bigg]  \nonumber \\ & \quad  + \sum_{K \in \mathcal{K}_h}  \| \delta^{1/2} \nabla \cdot \boldsymbol{u}_h \|^2_{0,K}
 + \sum_{K \in \mathcal{K}_h} \bigg[\| \tau^{1/2} ( \boldsymbol{u}_h \cdot \nabla \boldsymbol{u}_h + \nabla p_h) \|_{0,K}^2   \nonumber \\ & \quad  + \| \tau^{1/2}  2 \nu \nabla \cdot \varepsilon (\boldsymbol{u}_h) \|^2_{0,K} -2 ( 2 \nu \nabla \cdot \varepsilon (\boldsymbol{u}_h), \tau (\boldsymbol{u}_h \cdot \nabla \boldsymbol{u}_h + \nabla p_h) \bigg]
\end{align}   
\begin{remark}\label{remark_4.3} 
Consider the following estimates: 
\begin{enumerate}
\item Using the Cauchy-Schwarz, inverse and Poincar\'e inequalities, along with the Sobolev embedding \\ $H^1(\Omega) \hookrightarrow L^4(\Omega)$ and the assumption $\|\boldsymbol{u}_h\|_{L^4(\Omega)} < \hat{C}$, we obtain  
\[
(\boldsymbol{u}_h \cdot \nabla \boldsymbol{u}_h, \boldsymbol{u}_h) 
\leq \|\boldsymbol{u}_h\|_{0,4,\Omega} \|\nabla \boldsymbol{u}_h\|_{0,\Omega} \|\boldsymbol{u}_h\|_{0,4,\Omega} 
\leq \frac{\hat{C} C_{\mathrm{emb}}(1 + C_{\mathrm{poin}})}{\nu} \nu \| \varepsilon(\boldsymbol{u}_h) \|_{0,\Omega}^2 
\leq \tilde{C} \nu \| \varepsilon(\boldsymbol{u}_h) \|_{0,\Omega}^2,
\]
where $\hat{C} = \frac{\nu}{M C_{\mathrm{emb}} C_{\mathrm{poin}}} > 0$ and $M >> 1$.
    \item To derive this estimate, we utilize the given assumptions along with the inverse , Cauchy-Schwarz,  Poincaré inequalities, and Sobolev embedding $L^4 \hookrightarrow L^2$. Specifically, we have
\begin{align*}
\sum_{K \in \mathcal{K}_h} \| \tau^{1/2} \boldsymbol{u}_h \cdot \nabla \boldsymbol{u}_h \|_{0,K}^2 
    &\leq \sum_{K \in \mathcal{K}_h}  \frac{m_K h_K^2}{8 \nu} \| \boldsymbol{u}_h \|_{0,K}^2 \| \nabla \boldsymbol{u}_h \|_{0,K}^2 
    \leq \sum_{K \in \mathcal{K}_h}  \frac{m_K}{8 \nu} C_{\mathrm{in}}^2 C_{\mathrm{emb}}^2 \| \boldsymbol{u}_h \|_{0,4,K}^2 \| \boldsymbol{u}_h \|_{0,K}^2 \\
    &\leq \sum_{K \in \mathcal{K}_h}  \frac{m_K}{8 \nu} \| \boldsymbol{u}_h \|_{0,K}^2 \| \nabla \boldsymbol{u}_h \|_{0,K}^2 
    \leq \sum_{K \in \mathcal{K}_h}  \frac{m_K}{8 \nu} C_{\mathrm{in}}^2 C_{\mathrm{emb}}^2 \hat{C}^2 C_{\mathrm{poin}}^2 \| \nabla \boldsymbol{u}_h \|_{0,K}^2 \\
    &\leq \sum_{K \in \mathcal{K}_h}  \frac{m_K}{8 \nu^2} C_{\mathrm{in}}^2 C_{\mathrm{emb}}^2 \hat{C}^2 C_{\mathrm{poin}}^2 \nu \| \varepsilon( \boldsymbol{u}_h) \|_{0,K}^2 \leq  \sum_{K \in \mathcal{K}_h} \frac{m_K C_{\mathrm{in}}^2}{M^2} \nu \| \varepsilon( \boldsymbol{u}_h) \|_{0,K}^2.
    \end{align*}
Thus, by appropriately choosing $M$, the coefficient can be made sufficiently small.
\end{enumerate}
\end{remark}
Using \eqref{I2}, \eqref{I3} (with \( \mathscr{P}(\boldsymbol{u}_h) \) replaced by \( p_h \)), \eqref{tau_bound}, and Remark~\ref{remark_4.3}, along with the Cauchy-Schwarz and Young's inequalities in \eqref{stability_formulation}, we obtain

\begin{align*}
    \mathcal{A}_h^{\mathrm{NS}}(\boldsymbol{u}_h, p_h; \boldsymbol{u}_h, -p_h ) & \geq (1- \tilde{C}) 2 \nu \| \varepsilon (\boldsymbol{u}_h) \|^2_{0,\Omega} - (1+ \theta) \bigg[ 2 \nu \delta_1 C_{\mathrm{tr}}^2 \|\varepsilon(\boldsymbol{u}_h)\|_{0,\Omega}^2  + \frac{1}{\delta_1} \sum_{E \in \mathcal{E}_{\mathrm{Nav}}} \frac{\nu}{2 h_E} \|\boldsymbol{u}_h \cdot \boldsymbol{n}\|^2_{0,\Gamma_{\mathrm{Nav}}} \bigg] \\
    & \quad - (1- \theta) \bigg[ \frac{C_{\mathrm{in}}^2 C_{\mathrm{tr}}^2 \delta_2}{2} \sum_{K \in \mathcal{K}_h} \frac{h_K^2}{\nu} \| \nabla p_h \|^2_{0,K}  + \frac{1}{\delta_2} \sum_{E \in \mathcal{E}_{\mathrm{Nav}}} \frac{\nu}{2 h_E} \| \boldsymbol{u}_h \cdot \boldsymbol{n} \|^2_{0,\Gamma_\mathrm{Nav}} \bigg] \\
    & \quad + \frac{\gamma \nu}{h_E} \| \boldsymbol{u}_h \cdot \boldsymbol{n} \|^2_{0,\Gamma_{\mathrm{Nav}}}  + \sum_{K \in \mathcal{K}_h} \| \delta^{1/2} \nabla \cdot \boldsymbol{u}_h \|_{0,K}  \\
    & \quad + \sum_{K \in \mathcal{K}_h} \bigg[ \frac{1}{2} \| \tau^{1/2} ( \boldsymbol{u}_h \cdot \nabla \boldsymbol{u}_h + \nabla p_h) \|_{0,K}^2 - \| \tau^{1/2}  2 \nu \nabla \cdot \varepsilon (\boldsymbol{u}_h) \|^2_{0,K} \bigg]  \\
    & \geq 2 \nu \left(\frac{1}{2} - \tilde{C} - (1+ \theta) \delta_1 C_{\mathrm{tr}}^2  \right)\| \varepsilon (\boldsymbol{u}_h) \|_{0,\Omega}^2 + \left( \gamma  - \frac{(1+ \theta)}{2 \delta_1}  \frac{(1- \theta)}{2 \delta_2}  \right) \frac{\nu}{h_E} \| \boldsymbol{u}_h \cdot \boldsymbol{n} \|^2_{0,\Gamma_{\mathrm{Nav}}} \\
    & \quad - (1- \theta) \frac{C_{\mathrm{in}}^2 C_{\mathrm{tr}}^2 \delta_2}{2} \sum_{K \in \mathcal{K}_h} \frac{h_K^2}{\nu} \| \nabla p_h \|^2_{0,K}  + \sum_{K \in \mathcal{K}_h} \| \delta^{1/2} \nabla \cdot \boldsymbol{u}_h \|_{0,K}  \\
    & \quad + \sum_{K \in \mathcal{K}_h} \bigg[ \frac{1}{2} \| \tau^{1/2}  \boldsymbol{u}_h \cdot \nabla \boldsymbol{u}_h  \|_{0,K}^2 +  \frac{1}{2} \| \tau^{1/2} \nabla p_h \|_{0,K}^2  \\
    & \quad - \| \tau^{1/2}  \boldsymbol{u}_h \cdot \nabla \boldsymbol{u}_h  \|_{0,K}^2 -  \frac{1}{4} \| \tau^{1/2} \nabla p_h \|_{0,K}^2 \bigg] \\
    & \geq 2 \nu \left(\frac{1}{2} - \tilde{C} - (1+ \theta) \delta_1 C_{\mathrm{tr}}^2  \right)\| \varepsilon (\boldsymbol{u}_h) \|_{0,\Omega}^2 + \left( \gamma  - \frac{(1+ \theta)}{2 \delta_1} - \frac{(1- \theta)}{2 \delta_2}  \right) \frac{\nu}{h_E} \| \boldsymbol{u}_h \cdot \boldsymbol{n} \|^2_{0,\Gamma_{\mathrm{Nav}}}  \\
    & \quad + \| \delta^{1/2} \nabla \cdot \boldsymbol{u}_h \|^2_{0,\Omega}  + \left( \frac{\tau}{4}  - \frac{(1- \theta) C_{\mathrm{in}}^2 C_{\mathrm{tr}}^2 \delta_2 h_K^2}{2 \nu} \right) \| \nabla p_h \|^2_{0,\Omega} \\
    & \quad - \frac{m_K C_{\mathrm{in}}^2}{2 M^2} 2 \nu \| \varepsilon( \boldsymbol{u}_h) \|_{0,K}^2 + \frac{1}{2} \| \tau^{1/2}  \boldsymbol{u}_h \cdot \nabla \boldsymbol{u}_h  \|_{0,\Omega}^2 \\
    & \geq C_S \left( 2 \nu \| \varepsilon (\boldsymbol{u}_h ) \|_{0,\Omega}^2 + \sum_{E \in \mathcal{E}_{\mathrm{Nav}}} \frac{\nu}{h_E} \| \boldsymbol{u}_h \cdot \boldsymbol{n} \|_{0,E}^2 + \sum_{K \in \mathcal{K}_h} \frac{h_K^2}{\nu} \| \nabla p_h\|_{0,K}^2 \right. \\
    & \quad ~ \left. + \| \delta^{1/2} \nabla \cdot \boldsymbol{u}_h \|_{0,\Omega}^2 + \| \tau^{1/2}  \boldsymbol{u}_h \cdot \nabla \boldsymbol{u}_h  \|_{0,\Omega}^2 \right)
\end{align*}
  where $C_S = \min\left\{ \left( \frac{1}{2} - \tilde{C} - (1+ \theta) \delta_1 C_{\mathrm{tr}}^2 - \frac{m_K C_{\mathrm{in}}^2}{2 M^2}  \right), \left(\gamma  - \frac{(1+ \theta)}{2 \delta_1} - \frac{(1- \theta)}{2 \delta_2} \right), \frac{1}{2},  \left(\frac{m_K}{32}  - \frac{(1- \theta) C_{\mathrm{in}}^2 C_{\mathrm{tr}}^2 \delta_2}{2} \right) \right\}$ is a 
 positive constant and the result holds for sufficiently large $\gamma$. 
\end{proof}
\end{theorem}
\section{a priori analysis}\label{section5}
This section is dedicated to the a priori error analysis based on the arguments presented in \cite{MR1935805}. We introduce the following notation:
\begin{align*} 
\boldsymbol{u} - \boldsymbol{u}_h = (\boldsymbol{u} - \boldsymbol{I}_h)+( \boldsymbol{I}_h \boldsymbol{u} - \boldsymbol{u}_h) = \eta^{\boldsymbol{u}_h} + e^{\boldsymbol{u}_h}, \text{ and }
p - p_h = (p - J_h p)+(J_h p - p_h) =\eta^{p_h} + e^{p_h},
\end{align*}
\begin{remark}
For a fixed $\tilde{\boldsymbol{u}} \in \mathrm{V}_h$, we define the bilinear form $\mathcal{A}_h^{\mathrm{NS}, \tilde{\boldsymbol{u}}} ( \cdot, \cdot)$ as 
$$
    \begin{aligned}
\mathcal{A}_h^{\mathrm{NS},\tilde{\boldsymbol{u}}}\left(\boldsymbol{u}_h, p_h ; \boldsymbol{v}_h, q_h\right)   &\coloneqq \sum_{K \in \mathcal{K}_{h}} \bigg( 2 \nu \left(\varepsilon (\boldsymbol{u}_h), \varepsilon (\boldsymbol{v}_h) \right)_{K} + (\tilde{\boldsymbol{u}} \cdot \nabla \boldsymbol{u}_h, \boldsymbol{v}_h)_{K} - (p_h, \nabla \cdot \boldsymbol{v}_h)_{K} - (q_h, \nabla \cdot \boldsymbol{u}_h)_{K} \bigg) \\ 
&\quad + \sum_{E \in \mathcal{E}_{\mathrm{Nav}}} \bigg( -\int_{E} \boldsymbol{n}^t( 2 \nu \varepsilon (\boldsymbol{u}_h) - p_h I) \boldsymbol{n} (\boldsymbol{n} \cdot \boldsymbol{v}_h) \, ds  - \theta \int_{E} \boldsymbol{n}^t(2 \nu \varepsilon (\boldsymbol{v}_h) - q_h I) \boldsymbol{n} (\boldsymbol{n} \cdot \boldsymbol{u}_h) \, ds \\ 
&\quad + \int_{E} \beta \sum_{i=1}^{d-1} (\boldsymbol{\tau}^i \cdot \boldsymbol{v}_h)(\boldsymbol{\tau}^i \cdot \boldsymbol{u}_h)\, ds  + \gamma \nu \int_{E} h_E^{-1} (\boldsymbol{u}_h \cdot \boldsymbol{n})(\boldsymbol{v}_h \cdot \boldsymbol{n})  \, ds \bigg) \\& \quad + \sum_{K \in \mathcal{K}_{h}}  \int_K  \tau\bigg(-2 \nu \nabla \cdot  \varepsilon (\boldsymbol{u}_h) + \tilde{\boldsymbol{u}} \cdot \nabla \boldsymbol{u}_h + \nabla p_h \bigg) \bigg(-2 \nu \nabla \cdot  \varepsilon (\boldsymbol{v}_h)  + \tilde{\boldsymbol{u}} \cdot \nabla \boldsymbol{v}_h - \nabla q_h \bigg) \\ & \quad +  \sum_{K \in \mathcal{K}_{h}} \delta  \int_K \nabla \cdot \boldsymbol{u}_h \nabla \cdot \boldsymbol{v}_h,
\end{aligned}
$$
\end{remark}
\begin{lemma}\label{interpolation_estimate}
 Assume that the solution to \eqref{bb} satisfies $\boldsymbol{u} \in \boldsymbol{H}^{k+1}(\Omega) \cap \mathrm{V}$ and $p \in$ $H^{k+1}(\Omega) \cap \Pi$. The interpolation errors satisfy
$$
\begin{aligned}
&\left\|\tau^{-1 / 2} \eta^{\boldsymbol{u}_h}\right\|_{0,\Omega}^2+2 \nu\left\|\varepsilon\left(\eta^{\boldsymbol{u}_h}\right)\right\|_{0,\Omega}^2+ \sum_{E \in \mathcal{E}_{\mathrm{Nav}}} \frac{\gamma \nu}{h_E} \| \eta^{\boldsymbol{u}_h} \cdot \boldsymbol{n} \|_{0,E}^2+ \|\tau^{1 / 2} \boldsymbol{u}_h \cdot \nabla \eta^{\boldsymbol{u}_h} \|_{0,\Omega}^2+\left\|\delta^{1 / 2} \nabla \cdot \eta^{\boldsymbol{u}_h}\right\|_{0,\Omega}^2 \\
&+\sum_{K \in \mathcal{K}_h}\left[\left\|\tau^{1 / 2} 2 \nu \nabla \cdot \varepsilon\left(\eta^{\boldsymbol{u}_h}\right)\right\|_{0, K}^2+H\left(\operatorname{Re}_K-1\right)\left\|\delta^{-1 / 2} \eta^{p_h}\right\|_{0, K}^2 +H\left(1-\operatorname{Re}_K\right)(2 \nu)^{-1}\left\|\eta^{p_h}\right\|_{0, K}^2\right] \\
& +\left\|\tau^{1 / 2} \nabla \eta^{p_h}\right\|_{0,\Omega}^2  \leq C\left[\sum_{K \in \mathcal{K}_h} h_K^{2 k}|\boldsymbol{u}|_{k+1, K}^2\left(H\left(\operatorname{Re}_K-1\right) h_K \sup _{x \in K}|\boldsymbol{u}_h|_p+H\left(1-\operatorname{Re}_K\right) 2 \nu\right)\right. \\
&\left. +\sum_{K \in \mathcal{K}_h} h_K^{2 k}|p|_{k+1, K}^2\left(H\left(\operatorname{Re}_K-1\right) h_K \sup _{x \in K}|\boldsymbol{u}_h|_p^{-1}+H\left(1-\operatorname{Re}_K\right) h_K^2(2 \nu)^{-1}\right)\right] 
\end{aligned}
$$
where $H(\cdot)$ is the Heaviside function given by $H(x-y)= \begin{cases}0, & x<y \\ 1, & x>y\end{cases}.$
\begin{proof}
To prove this lemma, we first consider $\mathcal{A}_h^{\mathrm{NS},\boldsymbol{u}_h}(\eta^{\boldsymbol{u}_h}, \eta^{p_h};\eta^{\boldsymbol{u}_h}, \eta^{p_h})$. Then, the velocity estimates related to the momentum equation follow as in \cite{MR1155924}, while the pressure and velocity estimates related to the continuity equation follow as in \cite{MR1186727}. The terms corresponding to the slip condition are handled as we did in the stability proof.
\end{proof}
\end{lemma}
We may now establish the following convergence estimate. 
\begin{theorem}
Under the same assumptions as in Theorem \ref{stability}, we consider a sufficiently small \( M \) that satisfies the condition $ \| \boldsymbol{u}\|_{1,\Omega} < M.$
Let \( (\boldsymbol{u}, p) \in \mathrm{V} \times \Pi \) represent the solution to Problem \eqref{bb}, and let \( (\boldsymbol{u}_h, p_h) \in \mathrm{V}_h \times \Pi_h \) denote the solution to Problem \eqref{DSF}. We assume that \( (\boldsymbol{u}, p) \in (\boldsymbol{H}^{k+1}(\Omega) \cap \mathrm{V}) \times (H^k(\Omega) \cap \Pi) \), where \( k \geq 1 \). Under these conditions, there exists a constant \( C > 0 \), such that
\begin{align*}
& \nu \| \varepsilon (\boldsymbol{u}-\boldsymbol{u}_h ) \|_{0,\Omega}^2 + \sum_{E \in \mathcal{E}_{\mathrm{Nav}}} \frac{\nu}{h_E} \| (\boldsymbol{u}-\boldsymbol{u}_h) \cdot \boldsymbol{n} \|_{0,E}^2 + \sum_{K \in \mathcal{K}_h} \frac{h_K^2}{\nu} \| \nabla (p-p_h)\|_{0,K}^2 + \| \delta^{1/2} \nabla \cdot (\boldsymbol{u}-\boldsymbol{u}_h) \|_{0,\Omega}^2 \\ & + \| \tau^{1/2} \boldsymbol{u}_h \cdot \nabla (\boldsymbol{u}-\boldsymbol{u}_h) \|_{0,\Omega}^2 
\leq C\left[\sum_{K \in \mathcal{K}_h} h_K^{2 k}|\boldsymbol{u}|_{k+1, K}^2\left(H\left(\operatorname{Re}_K-1\right) h_K \sup _{x \in K}|\boldsymbol{u}_h|_p+H\left(1-\operatorname{Re}_K\right) 2 \nu\right)\right. \\
&\left. +\sum_{K \in \mathcal{K}_h} h_K^{2 k}|p|_{k+1, K}^2\left(H\left(\operatorname{Re}_K-1\right) h_K \sup _{x \in K}|\boldsymbol{u}_h|_p^{-1}+H\left(1-\operatorname{Re}_K\right) h_K^2(2 \nu)^{-1}\right)\right] + h^2 \| \boldsymbol{f} \|^2_{0,\Omega},
\end{align*}
where $H( \cdot )$ is defined above.
\begin{proof}
The convergence proof is as follows: 
    \begin{align}\label{main_a_priori}
&  \nu \| \varepsilon (\boldsymbol{I}_h \boldsymbol{u}-\boldsymbol{u}_h ) \|_{0,\Omega}^2 
+ \sum_{E \in \mathcal{E}_{\mathrm{Nav}}} \frac{\nu}{h_E} \| (\boldsymbol{I}_h \boldsymbol{u}-\boldsymbol{u}_h) \cdot \boldsymbol{n} \|_{0,E}^2  + \sum_{K \in \mathcal{K}_h} \frac{h_K^2}{\nu} \| \nabla (I_h p-p_h)\|_{0,K}^2 \nonumber \\
&
+ \| \delta^{1/2} \nabla \cdot (\boldsymbol{I}_h \boldsymbol{u}-\boldsymbol{u}_h) \|_{0,\Omega}^2 + \| \tau^{1/2} \boldsymbol{u}_h \cdot \nabla ( \boldsymbol{I}_h \boldsymbol{u}-\boldsymbol{u}_h) \|_{0,\Omega}^2   \nonumber \\
& \quad 
\leq \mathcal{A}_h^{\mathrm{NS}, \boldsymbol{u}_h} \left( e^{\boldsymbol{u}_h}, e^{p_h}; e^{\boldsymbol{u}_h},- e^{p_h} \right) \nonumber  \\
& \quad =  \mathcal{A}_h^{\mathrm{NS}, \boldsymbol{u}_h} \left( \boldsymbol{I}_h \boldsymbol{u}_h, I_h p_h; e^{\boldsymbol{u}_h},- e^{p_h} \right) 
- \mathcal{A}_h^{\mathrm{NS}, \boldsymbol{u}_h} \left( \boldsymbol{u}_h, p_h; e^{\boldsymbol{u}_h},- e^{p_h} \right) \nonumber \\
& \  = \mathcal{A}_h^{\mathrm{NS}, \boldsymbol{u}_h} \left(  \boldsymbol{u}, p; e^{\boldsymbol{u}_h},- e^{p_h} \right) 
-  \mathcal{A}_h^{\mathrm{NS}, \boldsymbol{u}} \left(  \boldsymbol{u}, p; e^{\boldsymbol{u}_h},- e^{p_h} \right) -  \mathcal{A}_h^{\mathrm{NS}, \boldsymbol{u}_h} \left(  \eta^{\boldsymbol{u}_h}, \eta^{p_h}; e^{\boldsymbol{u}_h},- e^{p_h} \right)
\end{align}
Consider the third term and choose \( \theta = \{-1, 1\} \). Apply the Cauchy-Schwarz, Triangle, and Young's inequalities; we get
\begin{align*}
  &  \mathcal{A}_h^{\mathrm{NS}, \boldsymbol{u}_h} \left(  \eta^{\boldsymbol{u}_h}, \eta^{p_h}; e^{\boldsymbol{u}_h}, -e^{p_h} \right) \\ 
  & = \sum_{K \in \mathcal{K}_{h}} \bigg( 2 \nu \left(\varepsilon (\eta^{\boldsymbol{u}_h}), \varepsilon (e^{\boldsymbol{u}_h}) \right)_{K} + (\boldsymbol{u}_h \cdot \nabla \eta^{\boldsymbol{u}_h}, e^{\boldsymbol{u}_h})_{K} - (\eta^{p_h}, \nabla \cdot e^{\boldsymbol{u}_h})_{K} + ( e^{p_h}, \nabla \cdot \eta^{\boldsymbol{u}_h})_{K} \bigg) \\ 
  & \quad + \sum_{E \in \mathcal{E}_{\mathrm{Nav}}} \bigg( -\int_{E} \boldsymbol{n}^t( 2 \nu \varepsilon (\eta^{\boldsymbol{u}_h}) 
  - \eta^{p_h} I) \boldsymbol{n} (\boldsymbol{n} \cdot e^{\boldsymbol{u}_h}) \, ds  - \theta \int_{E} \boldsymbol{n}^t(2 \nu \varepsilon (e^{\boldsymbol{u}_h}) + e^{p_h} I) \boldsymbol{n} (\boldsymbol{n} \cdot \eta^{\boldsymbol{u}_h}) \, ds \\ 
  & \quad + \int_{E} \beta \sum_{i=1}^{d-1} (\boldsymbol{\tau}^i \cdot e^{\boldsymbol{u}_h})(\boldsymbol{\tau}^i \cdot \eta^{\boldsymbol{u}_h}) \, ds + \frac{\gamma \nu }{h_E} \int_{E}  (\eta^{\boldsymbol{u}_h} \cdot \boldsymbol{n})(e^{\boldsymbol{u}_h} \cdot \boldsymbol{n}) \, ds \bigg) \\ 
  & \quad + \sum_{K \in \mathcal{K}_{h}}  \int_K  \tau  \bigg(-2 \nu \nabla \cdot  \varepsilon (\eta^{\boldsymbol{u}_h}) + \boldsymbol{u}_h \cdot \nabla \eta^{\boldsymbol{u}_h} + \nabla \eta^{p_h} \bigg) \bigg(-2 \nu \nabla \cdot  \varepsilon (e^{\boldsymbol{u}_h})  + \boldsymbol{u}_h \cdot \nabla e^{\boldsymbol{u}_h} + \nabla  e^{p_h} \bigg)  \\ 
  & \quad +  \sum_{K \in \mathcal{K}_{h}} \delta  \int_K \nabla \cdot \eta^{\boldsymbol{u}_h} \nabla \cdot e^{\boldsymbol{u}_h} \\ 
  & \leq  2 \nu \left(\varepsilon (\eta^{\boldsymbol{u}_h}), \varepsilon (e^{\boldsymbol{u}_h}) \right)_{K} + (\boldsymbol{u}_h \cdot \nabla \eta^{\boldsymbol{u}_h}, e^{\boldsymbol{u}_h})_{K} - (\eta^{p_h}, \nabla \cdot e^{\boldsymbol{u}_h})_{K} + ( e^{p_h}, \nabla \cdot \eta^{\boldsymbol{u}_h})_{K} \\ 
  & \quad + \sum_{E \in \mathcal{E}_{\mathrm{Nav}}} \bigg( -\int_{E} \boldsymbol{n}^t( 2 \nu \varepsilon (\eta^{\boldsymbol{u}_h}) - \eta^{p_h} I) \boldsymbol{n} (\boldsymbol{n} \cdot e^{\boldsymbol{u}_h}) \, ds - \theta \int_{E} \boldsymbol{n}^t(2 \nu \varepsilon (e^{\boldsymbol{u}_h}) + e^{p_h} I) \boldsymbol{n} (\boldsymbol{n} \cdot \eta^{\boldsymbol{u}_h}) \, ds \\ 
  & \quad + \int_{E} \beta \sum_{i=1}^{d-1} (\boldsymbol{\tau}^i \cdot e^{\boldsymbol{u}_h})(\boldsymbol{\tau}^i \cdot \eta^{\boldsymbol{u}_h}) \, ds + \frac{\gamma \nu}{h_E} \int_{E} (\eta^{\boldsymbol{u}_h} \cdot \boldsymbol{n})(e^{\boldsymbol{u}_h} \cdot \boldsymbol{n}) \, ds \bigg)  + \left( \nabla \cdot \eta^{\boldsymbol{u}_h}, \delta \nabla \cdot e^{\boldsymbol{u}_h} \right) \\ 
  & \quad + \sum_{K \in \mathcal{K}_h} \left( \| \tau^{1/2} (\boldsymbol{u}_h \cdot \nabla e^{\boldsymbol{u}_h} + \nabla  e^{p_h}) \|_{0,K} + \| \tau^{1/2} 2 \nu \nabla \cdot  \varepsilon (e^{\boldsymbol{u}_h}) \|_{0,K} \right) \\ 
  & \quad \times  \left( \|\tau^{1/2} ( -2 \nu \nabla \cdot  \varepsilon (\eta^{\boldsymbol{u}_h})  + \boldsymbol{u}_h \cdot \nabla \eta^{\boldsymbol{u}_h} + \nabla  \eta^{p_h}) \|_{0,K} \right) \\ & \leq \frac{3}{ 4 \gamma_1} 2 \nu \| \varepsilon ( e^{\boldsymbol{u}_h}) \|_{0,\Omega}^2 + \frac{1}{ \gamma_2} \| \tau^{1/2} \boldsymbol{u}_h \cdot \nabla e^{\boldsymbol{u}_h} \|_{0,\Omega}^2 + \frac{1}{ \gamma_2} \| \tau^{1/2}  \nabla e^{p_h} \|_{0,\Omega}^2 + \frac{1}{2 \gamma_3} \| \delta^{1/2} \nabla \cdot e^{\boldsymbol{u}_h} \|_{0,\Omega}^2 \\& \quad  + \left( \eta^{p_h}, \nabla \cdot e^{\boldsymbol{u}_h} \right)  + (\gamma_1 + \gamma_2 ) \| \tau^{1/2} \boldsymbol{u}_h \cdot \nabla \eta^{\boldsymbol{u}_h} \|_{0,\Omega}^2 + 2 (\gamma_1 + \gamma_2) \| \tau^{1/2} \nabla \eta^{p_h} \|_{0,\Omega}^2 + 2 (\gamma_1 + \gamma_2) \\ & \qquad \| \tau^{1/2} 2 \nu \nabla \cdot \varepsilon (\eta^{\boldsymbol{u}_h}) \|_{0,\Omega}^2  + \frac{\gamma_3}{2} \| \delta^{1/2} \nabla \cdot \eta^{\boldsymbol{u}_h}
  \|_{0,\Omega}^2  + \frac{\gamma_1}{2 } 2 \nu \| \varepsilon ( \eta^{\boldsymbol{u}_h} ) \|_{0,\Omega}^2 + ( \boldsymbol{u}_h \cdot \nabla \eta^{\boldsymbol{u}_h} , e^{\boldsymbol{u}_h}) \\& \quad + \frac{1}{2 \gamma_2} \| \tau^{1/2} \nabla e^{p_h} \|_{0,\Omega}^2 + \frac{\gamma_2}{2} \| \tau^{-1/2} \eta^{\boldsymbol{u}_h} \|^2_{0,\Omega} + \sum_{E \in \mathcal{E}_{\mathrm{Nav}}} \bigg(  \int_E 2 e^{p_h} (\boldsymbol{n} \cdot \eta^{\boldsymbol{u}_h}) \, ds \\& \quad  + \int_E \boldsymbol{n}^t (2 \nu) \varepsilon (\eta^{\boldsymbol{u}_h}) \boldsymbol{n} (\boldsymbol{n} \cdot e^{\boldsymbol{u}_h}) \, ds  + \int_E \boldsymbol{n}^t (2 \nu) \varepsilon (e^{\boldsymbol{u}_h}) \boldsymbol{n} (\boldsymbol{n} \cdot \eta^{\boldsymbol{u}_h}) \, ds \\ 
  & \quad + \int_E \eta^{p_h} (\boldsymbol{n} \cdot  e^{\boldsymbol{u}_h}) \, ds   + \int_E \beta \sum_{i=1}^{d-1} (\boldsymbol{\tau}^i \cdot e^{\boldsymbol{u}_h})(\boldsymbol{\tau}^i \cdot \eta^{\boldsymbol{u}_h}) \, ds + \frac{\gamma \nu}{h_E} \int_{E} (\eta^{\boldsymbol{u}_h} \cdot \boldsymbol{n})(e^{\boldsymbol{u}_h} \cdot \boldsymbol{n}) \, ds \bigg)  
\end{align*}
Next, by applying the Cauchy-Schwarz and Young's inequalities, we address the boundary terms in $\mathcal{E}_{\mathrm{Nav}}$, similar to \eqref{I2}, \eqref{I3}, and \eqref{tau_bound}
\begin{align*}
    \int_E \boldsymbol{n}^t (2 \nu) \varepsilon (\eta^{\boldsymbol{u}_h}) \boldsymbol{n} (\boldsymbol{n} \cdot e^{\boldsymbol{u}_h}) \, ds    &\leq 2 \nu \gamma_1 C_{\mathrm{tr}}^2 \|\varepsilon(\eta^{\boldsymbol{u}_h)}\|_{0,\Omega}^2 + \frac{1}{\gamma_1} \sum_{E \in \mathcal{E}_{\mathrm{Nav}}} \frac{\nu}{2 h_E} \|e^{\boldsymbol{u}_h} \cdot \boldsymbol{n}\|^2_{0,\Gamma_{\mathrm{Nav}}} \\ \int_E \boldsymbol{n}^t (2 \nu) \varepsilon (e^{\boldsymbol{u}_h}) \boldsymbol{n} (\boldsymbol{n} \cdot \eta^{\boldsymbol{u}_h}) \, ds   & \leq  \gamma_1  \sum_{E \in \mathcal{E}_{\mathrm{Nav}}} \frac{2 \nu}{ h_E}  \|\eta^{\boldsymbol{u}_h}  \cdot \boldsymbol{n} \|^2_{0,\Gamma_{\mathrm{Nav}}} +   \frac{2 \nu}{4 \gamma_1} C_{\mathrm{tr}}^2 \|  \varepsilon( e^{\boldsymbol{u}_h})\|^2_{0,\Omega}
    \\ \int_E 2 e^{p_h}( \eta^{\boldsymbol{u}_h} \cdot \boldsymbol{n} ) \, ds & \leq   {\gamma_2} \sum_{E \in \mathcal{E}_{\mathrm{Nav}}} \frac{2 \nu}{ 2 h_E} \| \eta^{\boldsymbol{u}_h} \cdot \boldsymbol{n} \|^2_{0,\Gamma_\mathrm{Nav}} + \frac{ 2 C_{\mathrm{in}}^2 C_{\mathrm{tr}}^2}{2 \gamma_2}\sum_{K \in \mathcal{K}_h} \frac{h_K^2}{\nu} \| \nabla e^{p_h} \|^2_{0,K} \\ \int_E \eta^{p_h}( e^{\boldsymbol{u}_h} \cdot \boldsymbol{n} ) \, ds & \leq \gamma_2 \frac{C_{\mathrm{in}}^2 C_{\mathrm{tr}}^2}{2 }\sum_{K \in \mathcal{K}_h} \frac{h_K^2}{\nu} \| \nabla \eta^{p_h} \|^2_{0,K} +  \frac{1}{\gamma_2} \sum_{E \in \mathcal{E}_{\mathrm{Nav}}} \frac{\nu}{ 2 h_E} \| e^{\boldsymbol{u}_h} \cdot \boldsymbol{n} \|^2_{0,\Gamma_\mathrm{Nav}} \\
    \frac{\gamma \nu}{h_E} \int_{E} (\eta^{\boldsymbol{u}_h} \cdot \boldsymbol{n})(e^{\boldsymbol{u}_h} \cdot \boldsymbol{n}) \, ds & \leq {\gamma_1} \sum_{E \in \mathcal{E}_{\mathrm{Nav}}} \frac{\nu}{ 2 h_E} \| \eta^{\boldsymbol{u}_h} \cdot \boldsymbol{n} \|^2_{0,\Gamma_\mathrm{Nav}} + \frac{1}{\gamma_1} \sum_{E \in \mathcal{E}_{\mathrm{Nav}}} \frac{\nu}{ 2 h_E} \| e^{\boldsymbol{u}_h} \cdot \boldsymbol{n} \|^2_{0,\Gamma_\mathrm{Nav}} \\
    \int_E \beta \sum_{i=1}^{d-1} (\boldsymbol{\tau}^i \cdot e^{\boldsymbol{u}_h})(\boldsymbol{\tau}^i \cdot \eta^{\boldsymbol{u}_h}) \, ds  & \leq \frac{\beta \gamma_1}{2} \| \varepsilon (\eta^{\boldsymbol{u}_h})\|^2_{0,\Omega} + \frac{\beta}{2 \gamma_1} \| \varepsilon (e^{\boldsymbol{u}_h})\|^2_{0,\Omega} \\
      \| \tau^{1/2} 2 \nu \nabla \cdot \varepsilon (\eta^{\boldsymbol{u}_h})\|_{0,\Omega}^2 & \leq \frac{2 \nu}{2} \| \varepsilon (\eta^{\boldsymbol{u}_h}) \|_{0,\Omega}^2
\end{align*}
Now, the simplified equation is defined as 
  \begin{align}\label{subs1}
  &  \mathcal{A}_h^{\mathrm{NS}, \boldsymbol{u}_h} \left( \eta^{\boldsymbol{u}_h}, \eta^{p_h}; e^{\boldsymbol{u}_h}, -e^{p_h} \right) \nonumber \\
    & \quad \leq 2 \nu \left( \frac{3}{4 \gamma_1} + \frac{C_{\mathrm{tr}}^2}{4 \gamma_1} +\frac{\beta}{4 \nu \gamma_1} \right) \| \varepsilon (e^{\boldsymbol{u}_h}) \|_{0,\Omega}^2  + \left( \frac{3 \gamma_1}{2} + \gamma_2 + \gamma_1 C^2_{\mathrm{tr}} + \frac{\beta \gamma_1}{4 \nu} \right) 2 \nu \| \varepsilon (\eta^{\boldsymbol{u}_h}) \|_{0,\Omega}^2  \nonumber \\
    & \quad + \frac{\nu}{h_E} \sum_{E \in \mathcal{E}_{\mathrm{Nav}}} \left( \frac{1}{ \gamma_1} + \frac{1}{2 \gamma_2} \right) \| e^{\boldsymbol{u}_h} \cdot  \boldsymbol{n} \|_{0, \Gamma_{\mathrm{Nav}}}^2  + \frac{\nu}{h_E} \left( \frac{5}{2} \gamma_1 + \gamma_2 \right) \sum_{E \in \mathcal{E}_{\mathrm{Nav}}}  \|\eta^{\boldsymbol{u}_h} \cdot  \boldsymbol{n} \|_{0, \Gamma_{\mathrm{Nav}}}^2  \nonumber \\
    & \quad  + \left( \frac{3 }{2 \gamma_2} + \frac{C_{\mathrm{in}}^2 C_{\mathrm{tr}}^2}{\gamma_2} \right) \| \tau^{1/2} \nabla e^{p_h} \|_{0,\Omega}^2   + \left( 2 \gamma_1 + 2 \gamma_2 + \frac{\gamma_2 C_{\mathrm{in}}^2 C_{\mathrm{tr}}^2}{2} \right) \|  \tau^{1/2} \nabla \eta^{p_h} \|_{0,\Omega}^2 + \frac{1}{\gamma_2} \| \tau^{1/2} \boldsymbol{u}_h \cdot \nabla e^{\boldsymbol{u}_h} \|_{0,\Omega}^2 \nonumber \\
    & \quad  + ( \gamma_1 + \gamma_2) \| \tau^{1/2} \boldsymbol{u}_h \cdot \nabla \eta^{\boldsymbol{u}_h} \|_{0,\Omega}^2+ \frac{\gamma_2}{2} \| \tau^{-1/2} \eta^{\boldsymbol{u}_h} \|_{0,\Omega}^2 + \frac{\gamma_3}{2} \| \delta^{1/2} \nabla \cdot \eta^{\boldsymbol{u}_h} \|_{0,\Omega}^2 + \frac{1}{2 \gamma_3} \| \delta^{1/2} \nabla \cdot e^{\boldsymbol{u}_h} \|_{0,\Omega}^2   \nonumber \\ & \quad  + \left( \eta^{p_h}, \nabla \cdot e^{\boldsymbol{u}_h} \right) + \left( \boldsymbol{u}_h \cdot \nabla \eta^{\boldsymbol{u}_h}, e^{\boldsymbol{u}_h} \right) 
\end{align}
Consider the term $ \left(\eta^{p_h}, \nabla \cdot e^{\boldsymbol{u}_h}\right)$, and simplify it using the Cauchy-Schwarz and Young inequalities.
\begin{align}\label{subs2}
 \left(\eta^{p_h}, \nabla \cdot e^{\boldsymbol{u}_h}\right) & \leq  \sum_{K \in \mathcal{K}_h}\left[\frac{1}{2 \gamma_3} H\left(\operatorname{Re}_K-1\right)\left\|\delta^{1 / 2} \nabla \cdot e^{\boldsymbol{u}_h}\right\|_{0, K}^2+\frac{1}{2 \gamma_1} H\left(1-\operatorname{Re}_K\right) 2 \nu\left\|\varepsilon\left(e^{\boldsymbol{u}_h}\right)\right\|_{0, K}^2\right. \nonumber \\
& \quad  \left.+\frac{\gamma_3}{2} H\left(\operatorname{Re}_K-1\right)\left\|\delta^{-1 / 2} \eta^{p_h}\right\|_{0,K}^2+\frac{\gamma_1}{2} H\left(1-\operatorname{Re}_K\right)(2 \nu)^{-1}\left\|\eta^{p_h}\right\|_{0, K}^2\right] 
\end{align}
Next, Using Cauchy-Schwarz, Young's inequalities and the Sobolev embedding $H^1(\Omega) \hookrightarrow L^4(\Omega)$, along with the assumption of Theorem \ref{stability},  and  we get 
\begin{align}\label{subs3}
    \left( \boldsymbol{u}_h \cdot \nabla \eta^{\boldsymbol{u}_h}, e^{\boldsymbol{u}_h} \right) & \leq \|\boldsymbol{u}_h\|_{0,4,\Omega} \|\nabla \eta^{\boldsymbol{u}_h} \|_{2,\Omega}\|e^{\boldsymbol{u}_h}\|_{0,4,\Omega} \leq C \|\eta^{\boldsymbol{u}_h} \|_{1,\Omega} \|e^{\boldsymbol{u}_h}\|_{1,\Omega} \nonumber \\& \leq
     C \nu \left(\frac{\gamma_1}{2} \|\varepsilon (\eta^{\boldsymbol{u}_h})\|_{0,\Omega}^2 + \frac{1}{2 \gamma_1 } \| \varepsilon (e^{\boldsymbol{u}_h}) \|_{0,\Omega}^2 \right)
\end{align}
Next, we consider the first two terms from \eqref{main_a_priori} to establish the convergence of the nonlinear argument.
\begin{align}\label{a_priori_2}
   & \mathcal{A}_h^{\mathrm{NS}, \boldsymbol{u}_h} \left(  \boldsymbol{u}, p; e^{\boldsymbol{u}_h},- e^{p_h} \right) -  \mathcal{A}_h^{\mathrm{NS}, \boldsymbol{u}} \left(  \boldsymbol{u}, p; e^{\boldsymbol{u}_h},- e^{p_h} \right) = \nonumber  \\ & \quad  \left((\boldsymbol{u}_h - \boldsymbol{u}) \cdot \nabla \boldsymbol{u}, e^{\boldsymbol{u}_h}  \right) + \int_K \tau (-2 \nu \nabla \cdot \varepsilon (\boldsymbol{u}) + \nabla p) \left((\boldsymbol{u}_h - \boldsymbol{u}) \cdot \nabla e^{\boldsymbol{u}_h} \right) + \int_K \tau \left( (\boldsymbol{u}_h - \boldsymbol{u}) \cdot \nabla \boldsymbol{u}  \right)\nonumber \\ & \quad \times  ( -2 \nu \nabla \cdot \varepsilon (e^{\boldsymbol{u}_h}) + \nabla e^{p_h })  + \int_K \tau \left[ (\boldsymbol{u}_h \cdot \nabla \boldsymbol{u})( \boldsymbol{u}_h \cdot \nabla e^{\boldsymbol{u}_h}) - ( \boldsymbol{u} \cdot \nabla \boldsymbol{u}) ( \boldsymbol{u} \cdot \nabla e^{\boldsymbol{u}_h })  \right]
\end{align}
Now add and subtract $ \int_K \tau (\boldsymbol{u} \cdot \nabla \boldsymbol{u}) (\boldsymbol{u}_h \cdot \nabla e^{\boldsymbol{u}_h})$ in \eqref{a_priori_2}, we get 
\begin{align}\label{a_priori_3}
   \mathcal{A}_h^{\mathrm{NS}, \boldsymbol{u}_h} \left(  \boldsymbol{u}, p; e^{\boldsymbol{u}_h},- e^{p_h} \right) -  \mathcal{A}_h^{\mathrm{NS}, \boldsymbol{u}}\left(  \boldsymbol{u}, p; e^{\boldsymbol{u}_h},- e^{p_h} \right)  & = \mathcal{J}_1 + \mathcal{J}_2 + \mathcal{J}_3
 \end{align}
 where the $\mathcal{J}_i's$ are defined as follows 
\begin{align*}
    \mathcal{J}_1 :=&  \left((\boldsymbol{u}_h - \boldsymbol{u}) \cdot \nabla \boldsymbol{u}, e^{\boldsymbol{u}_h}  \right)   \\ 
    \mathcal{J}_2 :=& \int_K \tau \left( - 2 \nu \nabla \cdot \varepsilon (\boldsymbol{u}) + \boldsymbol{u} \cdot \nabla \boldsymbol{u} + \nabla p \right) \left( ( \boldsymbol{u}_h - \boldsymbol{u} )\cdot \nabla e^{\boldsymbol{u}_h} \right) \\
    \mathcal{J}_3  :=& \int_K \tau \left( ( \boldsymbol{u}_h - \boldsymbol{u}) \cdot \nabla \boldsymbol{u} \right) \left( - 2 \nu \nabla \cdot \varepsilon (e^{\boldsymbol{u}_h}) + \boldsymbol{u}_h \cdot \nabla e^{\boldsymbol{u}_h} + \nabla e^{p_h} \right)
\end{align*}
Now, we bound $\mathcal{J}_1$ as 
\begin{align*}
     \mathcal{J}_1 &
     \leq (\|\eta^{\boldsymbol{u}_h}\|_{1,\Omega} + \| e^{\boldsymbol{u}_h}\|_{1,\Omega} ) \|\nabla \boldsymbol{u} \|_{0,\Omega} \| e^{\boldsymbol{u}_h} \|_{1,\Omega} \\ & \leq  C M ( \| e^{\boldsymbol{u}_h}\|_{1,\Omega}^2 + \|\eta^{\boldsymbol{u}_h}\|_{1,\Omega}^2 )
\end{align*}
Since $(\boldsymbol{u}, p)$ satisfies the momentum equation, we apply the Cauchy–Schwarz and Young's inequalities, along with the embedding $H^1(\Omega) \hookrightarrow L^4(\Omega)$ and \eqref{rem2}, to bound $\mathcal{J}_2$ as
\begin{align*}
    \mathcal{J}_2  & \lesssim  h \| \boldsymbol{f} \|_{0,\Omega} \| \tau^{1/2} \boldsymbol{u}_h \cdot \nabla e^{\boldsymbol{u}_h} \|_{0,\Omega} + h_K \| \boldsymbol{f} \|_{0,K} \| \boldsymbol{u} \|_{0,4,K} \| e^{\boldsymbol{u}_h} \|_{0,4,K} \\ & \lesssim h \| \boldsymbol{f} \|_{0,\Omega} \left( \| \tau^{1/2} \boldsymbol{u}_h \cdot \nabla e^{\boldsymbol{u}_h} \|_{0,\Omega} + \| \boldsymbol{u} \|_{1,\Omega} \| e^{\boldsymbol{u}_h} \|_{1,\Omega} \right) \\
    & \lesssim h \| \boldsymbol{f} \|_{0,\Omega} \left( \| \tau^{1/2} \boldsymbol{u}_h \cdot \nabla e^{\boldsymbol{u}_h} \|_{0,\Omega} + M \| e^{\boldsymbol{u}_h} \|_{1,\Omega} \right) \\
    & \lesssim h^2 \| \boldsymbol{f} \|_{0,\Omega}^2 + \frac{1}{2} \| \tau^{1/2} \boldsymbol{u}_h \cdot \nabla e^{\boldsymbol{u}_h} \|_{0,\Omega}^2 + \frac{\nu}{2} \| \varepsilon(e^{\boldsymbol{u}_h}) \|_{0,\Omega}^2. 
    \end{align*}
Next, using Lemma \ref{local_inverse} with $m=0$, $l=1$, and $p=q=4$, we apply the Cauchy–Schwarz and Young's inequalities, along with  \eqref{tau_bound} and the Sobolev embedding $H^1(\Omega) \hookrightarrow L^4(\Omega)$. Consequently, we can bound $\mathcal{J}_3$ as
\begin{align*}
    \mathcal{J}_3  & \leq  h_K \|\boldsymbol{u}_h - \boldsymbol{u}\|_{0,4,K} \|\nabla \boldsymbol{u}\|_{0,4,K} \left( \| \tau^{1/2} 2 \nu \nabla \cdot \varepsilon (e^{\boldsymbol{u}_h}) \|_{0,2,K} + \| \tau^{1/2} \boldsymbol{u}_h \cdot \nabla e^{\boldsymbol{u}_h} \|_{0,2,K} + \| \tau^{1/2} \nabla e^{p_h} \|_{0,2,K} \right) \\ & \leq h_K \|\boldsymbol{u}_h - \boldsymbol{u}\|_{0,4,K}  h_K^{-1} \| \boldsymbol{u} \|_{0,4,K}  \left( \frac{2 \nu}{2} \|  \varepsilon (e^{\boldsymbol{u}_h}) \|_{0,2, K } + \| \tau^{1/2} \boldsymbol{u}_h \cdot \nabla e^{\boldsymbol{u}_h} \|_{0,2,K} + \| \tau^{1/2} \nabla e^{p_h} \|_{0,2,K} \right) \\ &  \leq \|\boldsymbol{u}_h - \boldsymbol{u}\|_{1,\Omega} \|\boldsymbol{u}\|_{1,\Omega} \left( \frac{2 \nu}{2} \| \varepsilon (e^{\boldsymbol{u}_h}) \|_{0,\Omega} + \| \tau^{1/2} \boldsymbol{u}_h \cdot \nabla e^{\boldsymbol{u}_h} \|_{0,\Omega} + \| \tau^{1/2} \nabla e^{p_h} \|_{0,\Omega} \right) \\ & \leq M  \|\boldsymbol{u}_h - \boldsymbol{u}\|_{1,\Omega} \left( \frac{2 \nu}{2} \| \varepsilon (e^{\boldsymbol{u}_h}) \|_{0,\Omega} + \| \tau^{1/2} \boldsymbol{u}_h \cdot \nabla e^{\boldsymbol{u}_h} \|_{0,\Omega} + \| \tau^{1/2} \nabla e^{p_h} \|_{0,\Omega} \right) \\ & \leq M \left( \left( \frac{\gamma_1}{2} + \frac{1}{\gamma_4} \right) \| \varepsilon (\eta^{\boldsymbol{u}_h}) \|_{0,\Omega}^2 + \left( \frac{1}{2 \gamma_1} + \frac{1}{\gamma_4} + 1 \right) \| \varepsilon (e^{\boldsymbol{u}_h}) \|_{0,\Omega}^2 + \frac{\gamma_4}{2} \| \tau^{1/2} \boldsymbol{u}_h \cdot \nabla e^{\boldsymbol{u}_h} \|_{0,\Omega}^2 \right. \\
    &  \left. + \frac{\gamma_4}{2} \| \tau^{1/2} \nabla e^{p_h} \|_{0,\Omega}^2 \right).
\end{align*}
Combine all the bounds on $\mathcal{J}_i's$ and substitute in \eqref{a_priori_3} and we have 
\begin{align}\label{subs4}
  & \mathcal{A}_h^{\mathrm{NS}, \boldsymbol{u}_h} \left(  \boldsymbol{u}, p; e^{\boldsymbol{u}_h}, - e^{p_h} \right) - \mathcal{A}_h^{\mathrm{NS}, \boldsymbol{u}} \left(  \boldsymbol{u}, p; e^{\boldsymbol{u}_h}, - e^{p_h} \right) 
    \leq M \bigg[ \left(\frac{\gamma_1}{2} + \frac{1}{\gamma_4} + 1 \right) \| \varepsilon (\eta^{\boldsymbol{u}_h}) \|_{0,\Omega}^2 \nonumber \\
   & \quad + \left( \frac{1}{2 \gamma_1} + \frac{1}{\gamma_4} + \frac{5}{2} \right) \nu \| \varepsilon (e^{\boldsymbol{u}_h}) \|_{0,\Omega}^2  + \left( \frac{1}{2} + \frac{\gamma_4}{2} \right) \| \tau^{1/2} \boldsymbol{u}_h \cdot \nabla e^{\boldsymbol{u}_h} \|_{0,\Omega}^2 \nonumber \\
   & \quad + \frac{\gamma_4}{2} \|\tau^{1/2} \nabla e^{p_h} \|_{0,\Omega}^2 \bigg] + h^2 \| \boldsymbol{f}\|_{0,\Omega}
\end{align}
Finally, we substitute \eqref{subs1}, \eqref{subs2}, \eqref{subs3}, and \eqref{subs4} into \eqref{main_a_priori}, and after simplification, we obtain 
\begin{align*}
    &  2 \nu \| \varepsilon (\boldsymbol{I}_h \boldsymbol{u}-\boldsymbol{u}_h ) \|_{0,\Omega}^2 + \sum_{E \in \mathcal{E}_{\mathrm{Nav}}} \frac{\nu}{h_E} \| (\boldsymbol{I}_h \boldsymbol{u}-\boldsymbol{u}_h) \cdot \boldsymbol{n} \|_{0,E}^2 + \sum_{K \in \mathcal{K}_h} \frac{h_K^2}{\nu} \| \nabla (I_h p-p_h)\|_{0,K}^2 \\ & + \| \delta^{1/2} \nabla \cdot (\boldsymbol{I}_h \boldsymbol{u}-\boldsymbol{u}_h) \|_{0,\Omega}^2  + \| \tau^{1/2} \boldsymbol{u}_h \cdot \nabla ( \boldsymbol{I}_h \boldsymbol{u}-\boldsymbol{u}_h) \|_{0,\Omega}^2 
\\& \leq 2 \nu \left( \frac{1}{ \gamma_1} + \frac{C_{\mathrm{tr}}^2}{4 \gamma_1}+ \frac{\beta}{4 \nu \gamma_1} + \frac{1}{2 \gamma_1} H\left(\operatorname{Re}_K-1\right) + M \left( \frac{1}{2 \gamma_1} + \frac{1}{\gamma_4} + \frac{5}{2} \right) \right) \| \varepsilon (e^{\boldsymbol{u}_h}) \|_{0,\Omega}^2  \\
    & \quad + \left( \frac{3 \gamma_1}{2} + \frac{\gamma_1}{4} + \gamma_2 + \gamma_1 C^2_{\mathrm{tr}}   + \frac{\beta \gamma_2}{4 \nu} + M \left( \frac{\gamma_1}{2 } + \frac{1}{\gamma_4} + 1 \right) \right) 2 \nu \| \varepsilon (\eta^{\boldsymbol{u}_h}) \|_{0,\Omega}^2  \\
    & \quad   + \frac{\nu}{h_E} \sum_{E \in \mathcal{E}_{\mathrm{Nav}}} \left( \frac{1}{ \gamma_1} + \frac{1}{2 \gamma_2} \right) \| e^{\boldsymbol{u}_h} \cdot  \boldsymbol{n} \|_{0, \Gamma_{\mathrm{Nav}}}^2   + \frac{\nu}{h_E} \left( \frac {5 \gamma_1}{2} + \gamma_2 \right)   \sum_{E \in \mathcal{E}_{\mathrm{Nav}}} \|\eta^{\boldsymbol{u}_h} \cdot  \boldsymbol{n} \|_{0, \Gamma_{\mathrm{Nav}}}^2 \\ & \quad   + \left( \frac{3 }{2 \gamma_2} + \frac{C_{\mathrm{in}}^2 C_{\mathrm{tr}}^2}{\gamma_2} + \frac{M \gamma_4 }{2} \right) \| \tau^{1/2} \nabla e^{p_h} \|_{0,\Omega}^2  + \left( 2 \gamma_1 + 2 \gamma_2 + \frac{\gamma_2 C_{\mathrm{in}}^2 C_{\mathrm{tr}}^2}{2} \right) \|  \tau^{1/2} \nabla \eta^{p_h} \|_{0,\Omega}^2  \\
    & \quad  + \left(\frac{1}{2 \gamma_2} + \frac{M \gamma_4}{2 } +  \frac{M}{2 }\right) \| \tau^{1/2} \boldsymbol{u}_h \cdot \nabla e^{\boldsymbol{u}_h} \|_{0,\Omega}^2+ ( \gamma_1 + \gamma_2) \| \tau^{1/2} \boldsymbol{u}_h \cdot \nabla \eta^{\boldsymbol{u}_h} \|_{0,\Omega}^2 + \frac{\gamma_2}{2} \| \tau^{-1/2} \eta^{\boldsymbol{u}_h} \|_{0,\Omega}^2 \\
    & \quad + \frac{\gamma_3}{2} \| \delta^{1/2} \nabla \cdot \eta^{\boldsymbol{u}_h} \|_{0,\Omega}^2   + \frac{1}{2 \gamma_3} \| \delta^{1/2} \nabla \cdot e^{\boldsymbol{u}_h} \|_{0,\Omega}^2  +  \sum_{K \in \mathcal{K}_h}\left[\frac{1}{2 \gamma_3} H\left(\operatorname{Re}_K-1\right)\left\|\delta^{1 / 2} \nabla \cdot e^{\boldsymbol{u}_h}\right\|_{0, K}^2  \right. \\
    & \quad  \left. +\frac{\gamma_3}{2} H\left(\operatorname{Re}_K-1\right)\left\|\delta^{-1 / 2} \eta^{p_h}\right\|_{0,K}^2 
    +\frac{\gamma_1}{2} H\left(1-\operatorname{Re}_K\right)(2 \nu)^{-1}\left\|\eta^{p_h}\right\|_{0, K}^2\right] + h^2 \| \boldsymbol{f} \|^2_{0,\Omega}
\end{align*}
We can choose $\gamma_1, \gamma_2, \gamma_3 > 0$ such that the coefficients of the terms involving $(\boldsymbol{I}_h \boldsymbol{u} - \boldsymbol{u}_h)$ and $(I_h p - p_h)$ on the left-hand side are less than 1. Additionally, we select $M$ to be sufficiently small so that, regardless of the value of $\gamma_4$, its coefficient remains less than 1.
\begin{align*}
    &  2 \nu \| \varepsilon (\boldsymbol{I}_h \boldsymbol{u}-\boldsymbol{u}_h ) \|_{0,\Omega}^2 + \sum_{E \in \mathcal{E}_{\mathrm{Nav}}} \frac{\nu}{h_E} \| (\boldsymbol{I}_h \boldsymbol{u}-\boldsymbol{u}_h) \cdot \boldsymbol{n} \|_{0,E}^2 + \sum_{K \in \mathcal{K}_h} \frac{h_K^2}{\nu} \| \nabla (I_h p-p_h)\|_{0,K}^2 \\ & \quad + \| \delta^{1/2} \nabla \cdot (\boldsymbol{I}_h \boldsymbol{u}-\boldsymbol{u}_h) \|_{0,\Omega}^2  + \| \tau^{1/2} \boldsymbol{u}_h \cdot \nabla ( \boldsymbol{I}_h \boldsymbol{u}-\boldsymbol{u}_h) \|_{0,\Omega}^2 
 \lesssim   2 \nu \| \varepsilon (\eta^{\boldsymbol{u}_h} ) \|_{0,\Omega}^2 + \sum_{E \in \mathcal{E}_{\mathrm{Nav}}} \frac{\nu}{h_E} \| \eta^{\boldsymbol{u}_h} \cdot \boldsymbol{n} \|_{0,E}^2 \\& \quad + \sum_{K \in \mathcal{K}_h} \frac{h_K^2}{\nu} \| \nabla \eta^{p_h}\|_{0,K}^2 + \| \tau^{1/2} \boldsymbol{u}_h \cdot \nabla \eta^{\boldsymbol{u}_h} \|_{0,\Omega}^2  + \| \delta^{1/2} \nabla \cdot \eta^{\boldsymbol{u}_h} \|_{0,\Omega}^2 + \| \tau^{-1/2} \eta^{\boldsymbol{u}_h} \|_{0,\Omega}^2 \\& \quad  +  H\left(\operatorname{Re}_K-1\right)\left\|\delta^{-1 / 2} \eta^{p_h}\right\|_{0,K}^2 
    + H\left(1-\operatorname{Re}_K\right)(2 \nu)^{-1}\left\|\eta^{p_h}\right\|_{0, K}^2 + h^2 \| \boldsymbol{f} \|^2_{0,\Omega}
\end{align*}
Now, by using Lemma \ref{interpolation_estimate}, we obtain 
\begin{align*}
    & 2 \nu \| \varepsilon (\boldsymbol{I}_h \boldsymbol{u} - \boldsymbol{u}_h) \|_{0,\Omega}^2 
    + \sum_{E \in \mathcal{E}_{\mathrm{Nav}}} \frac{\nu}{h_E} \| (\boldsymbol{I}_h \boldsymbol{u} - \boldsymbol{u}_h) \cdot \boldsymbol{n} \|_{0,E}^2  + \sum_{K \in \mathcal{K}_h} \frac{h_K^2}{\nu} \| \nabla (I_h p - p_h) \|_{0,K}^2 
    \\ & + \| \delta^{1/2} \nabla \cdot (\boldsymbol{I}_h \boldsymbol{u} - \boldsymbol{u}_h) \|_{0,\Omega}^2 + \| \tau^{1/2} \boldsymbol{u}_h \cdot \nabla ( \boldsymbol{I}_h \boldsymbol{u} - \boldsymbol{u}_h) \|_{0,\Omega}^2 \lesssim  h^2 \| \boldsymbol{f} \|_{0,\Omega}^2 + \\
    &  \left[ \sum_{K \in \mathcal{K}_h} h_K^{2 k} |\boldsymbol{u}|_{K+1, K}^2 
    \left( H\left(\operatorname{Re}_K - 1\right) h_K \sup_{x \in K} |\boldsymbol{u}_h|_p 
    + H\left(1 - \operatorname{Re}_K\right) 2 \nu \right) \right. \\
    & \left. + \sum_{K \in \mathcal{K}_h} h_K^{2 k} |p|_{k+1, K}^2 
    \left( H\left(\operatorname{Re}_K - 1\right) h_K \sup_{x \in K} |\boldsymbol{u}_h|_p^{-1} 
    + H\left(1 - \operatorname{Re}_K\right) h_K^2 (2 \nu)^{-1} \right) \right].
\end{align*}
The result follows simply by using triangle inequality in last inequality and Lemma \ref{interpolation_estimate}.
\end{proof} 
\end{theorem}
\section{A posteriori error estimation for the stationary Navier Stokes problem }\label{section6}
  For each  $ K \in \mathcal{K}_h$ and each $E \in \mathcal{E}_h$, we define element-wise and facet-wise residuals as follows:
 \begin{align*}
    \label{estimator_s}
    \boldsymbol{R}_K &= \left\{\boldsymbol{f}_h +  2 \nu \nabla \cdot \varepsilon(\boldsymbol{u}_h) - \boldsymbol{u}_h \cdot \nabla \boldsymbol{u}_h - \nabla p_h \right\}|_K,  \qquad
    \boldsymbol{R}_E = 
        \frac{1}{2} \llbracket (p_h \boldsymbol{I} - 2 \nu \varepsilon(\boldsymbol{u}_h) )\boldsymbol{n} \rrbracket_E , \quad \text{for } E \in \mathcal{E}_h \backslash \Gamma, \\
    \boldsymbol{R}_{J_K}^1 &= \left\|\sum_{i=1}^{d-1} \left(2 \nu \boldsymbol{n}^t \varepsilon(\boldsymbol{u}_h) \tau^i + \beta \boldsymbol{u}_h \cdot \tau^{i} \right)\right\|_{0,E}^2, \qquad 
    \boldsymbol{R}_{J_K}^2 = \frac{\gamma^2 \nu}{h_E}  \| \boldsymbol{u}_h \cdot \boldsymbol{n} \|_{0,E}^2.
\end{align*}

  Then we introduce the element-wise error estimators $\Psi_K^2 = \Psi_{R_K}^2 + \Psi^2_{R_E} + \Psi_{J_K}^2$ with contributions defined as 
  \begin{align*}
   \Psi^2_{R_K} = \frac{h_K^2}{\nu} \| \boldsymbol{R}_K\|^2_{0,K}, \qquad
   \Psi^2_{R_E} = \sum_{E \in \partial K \backslash \Gamma } \frac{h_E}{\nu} \| \boldsymbol{R}_E\|^2_{0,E}, \qquad
   \Psi^2_{J_K} = \sum_{E \in \Gamma_{\mathrm{Nav}} } \left(   \frac{h_E}{\nu} \boldsymbol{R}_{J_K}^1+ \boldsymbol{R}_{J_K}^2 \right).
  \end{align*}
  so a global a posteriori error estimator for the nonlinear stationary problem is 
  \begin{align}\label{total_estimate}
 \Psi := \left( \sum_{K \in \mathcal{K}_h}  \Psi_K^2  \right)^{1/2}
  \end{align}
    \subsection{Reliability} 
\begin{theorem}\label{GIS}
    (Global inf-sup stability)  Let $\tilde{\boldsymbol{u}} \in \boldsymbol{H}^1(\mathcal{K}_h)$ such that $\|\tilde{\boldsymbol{u}}\|_{1,\mathcal{K}_h} \leq M$, for a sufficiently small $M>0$. For any $\left( \boldsymbol{u}, p \right) \in \tilde{\mathrm{V}} \times \tilde{\Pi} = (\mathrm{V} \cap \boldsymbol{H}^2(\mathcal{K}_h)) \times \Pi \cap \boldsymbol{H}^1(\mathcal{K}_h))$,  with $\left|\!\left|\!\left| (\boldsymbol{v}, q )\right|\!\right|\!\right| \leq 1$ such that
    $$  
\mathcal{A}_h^{{\mathrm{NS}},\tilde{\boldsymbol{u}}}\left(\boldsymbol{u},p ; \boldsymbol{v},q \right) \geq C \left|\!\left|\!\left| (\boldsymbol{u}, p )\right|\!\right|\!\right|.
    $$
\end{theorem}
\begin{proof}
   For any $\left( \boldsymbol{u}, p \right) \in  \tilde{\mathrm{V}} \times \tilde{\Pi}$, there holds 
    $$
   \mathcal{A}_h^{{\mathrm{NS}},\tilde{\boldsymbol{u}}} \left( \boldsymbol{u},p ; \boldsymbol{u},-p \right) \geq \alpha \left|\!\left|\!\left| (\boldsymbol{u}, p )\right|\!\right|\!\right|^2 
    $$                     
    Applying the inf-sup stability of the continuous problem, we get that for any $p \in \tilde{\Pi} $, there exists $\boldsymbol{v} \in \tilde{\mathrm{V}} $ such that $-(p, \nabla \cdot \boldsymbol{v}) \geq \beta_1 \|p\|^2_{0, \Omega}$, where $\beta_1 >0$ is the inf-sup constant depending only on $\Omega$. 
    Then, we have 
    \begin{align*}
\mathcal{A}_h^{{\mathrm{NS}},\tilde{\boldsymbol{u}}}\left(\boldsymbol{u},p;\boldsymbol{v},0\right) & = 2 \nu \left(\varepsilon (\boldsymbol{u}), \varepsilon (\boldsymbol{v}) \right)_{\Omega} +(\tilde{\boldsymbol{u}} \cdot \nabla \boldsymbol{u}, \boldsymbol{v})_{\Omega} -(p, \nabla \cdot \boldsymbol{v})_{\Omega}  + \sum_{E \in \mathcal{E}_{\mathrm{Nav}}}  \int_{E} \beta \sum_{i=1}^{d-1} (\boldsymbol{\tau}^i \cdot \boldsymbol{v})(\boldsymbol{\tau}^i \cdot \boldsymbol{u})\, ds \\& \quad + \sum_{K \in \mathcal{K}_{h}}  \int_K  \tau\bigg(-2 \nu \nabla \cdot  \varepsilon (\boldsymbol{u}) + \boldsymbol{\tilde{u}} \cdot \nabla \boldsymbol{u} + \nabla p \bigg) \bigg(-2 \nu \nabla \cdot  \varepsilon (\boldsymbol{v})  + \boldsymbol{\tilde{u}} \cdot \nabla \boldsymbol{v} \bigg)
\\& \geq 
    \beta_1 \|p\|_{0,\Omega}^2 -|2 \nu \left(\varepsilon (\boldsymbol{u}), \varepsilon (\boldsymbol{v}) \right)_{\Omega}| -| (\tilde{\boldsymbol{u}} \cdot \nabla \boldsymbol{u}, \boldsymbol{v})_{\Omega}| - \sum_{K \in \mathcal{K}_h} \| \tau^{1/2} (-2 \nu \nabla \cdot  \varepsilon (\boldsymbol{u}) + \boldsymbol{\tilde{u}} \cdot \nabla \boldsymbol{u} \\ & \quad + \nabla p )\|_{0,K}  \| \tau^{1/2} (-2 \nu \nabla \cdot  \varepsilon (\boldsymbol{v})  + \boldsymbol{\tilde{u}} \cdot \nabla \boldsymbol{v}) \|_{0,K}  \\& \geq
    \beta_1 \|p\|_{0,\Omega}^2 -  \left(C_a + C_b \|\tilde{\boldsymbol{u}}\|_{1, \mathcal{K}_h} \right) \| \boldsymbol{u}\|_{1,\Omega} \|\boldsymbol{v}\|_{1,\Omega}   -  C \left( \| \nabla \boldsymbol{u} \|_{0,\Omega}^2 + \sum_{K \in \mathcal{K}_h} \frac{h_K^2}{\nu} \|\nabla p\|_{0,K}^2 \right)^{1/2}  \\& \qquad \| \nabla \boldsymbol{v} \|_{0,\Omega} \\&
    \geq \beta_1 \|p\|^2_{0,\Omega} - C \|\boldsymbol{u}\|_{1,\Omega} \|\boldsymbol{v}\|_{1,\Omega} - C \left( \| \nabla \boldsymbol{u} \|_{0,\Omega}^2 + \sum_{K \in \mathcal{K}_h} \frac{h_K^2}{\nu} \|\nabla p\|_{0,K}^2 \right)^{1/2}   \| \nabla \boldsymbol{v} \|_{0,\Omega} \\& \geq \left( \beta_1 - \frac{2}{\epsilon}\right) \|p\|^2_{0,\Omega} - \epsilon C ^2 \|\boldsymbol{u}\|^2_{1,\Omega} - \epsilon C^2 \left( \| \nabla \boldsymbol{u} \|_{0,\Omega}^2 + \sum_{K \in \mathcal{K}_h} \frac{h_K^2}{\nu} \|\nabla p\|_{0,K}^2 \right).
    \end{align*}
    where $\epsilon>0$ is a postive constants. Now, we introduce $\delta_1>0$ such that 
    \begin{align*}
\mathcal{A}_h^{{\mathrm{NS}},\tilde{\boldsymbol{u}}}\left(\boldsymbol{u},p;\boldsymbol{u}+\delta_1 \boldsymbol{v},-p\right)  & = \mathcal{A}_h^{{\mathrm{NS}},\tilde{\boldsymbol{u}}}(\boldsymbol{u},p;\boldsymbol{u},-p)  + \delta_1 \mathcal{A}_h^{{\mathrm{NS}},\tilde{\boldsymbol{u}}}(\boldsymbol{u},p; \boldsymbol{v},0) \\ & \geq(\alpha -2 \delta_1 \epsilon C^2  ) \|\boldsymbol{u}\|^2_{1,\Omega} + \alpha \| p\|_{0,\Omega}^2 + \delta_1 \left( \beta_1 - \frac{2}{\epsilon} \right) \| p\|^2_{0,\Omega} - \epsilon \delta_1 C^2 \frac{h_K^2}{\nu} \|\nabla p\|_{0,K}^2.
    \end{align*}
    Choosing $\epsilon = 4/\beta_1$ and $\delta_1= \alpha/ 4 \epsilon C^2$, we obtain 
    \begin{align*}
    \mathcal{A}_h^{{\mathrm{NS}},\tilde{\boldsymbol{u}}}\left(\boldsymbol{u},p;\boldsymbol{u}+\delta_1 \boldsymbol{v},-p \right)    & \gtrsim   \|\boldsymbol{u}\|^2_{1,\Omega} +  \| p\|^2_{0,\Omega}.
    \end{align*}
    Finally, using the triangle inequality, we can assert that 
 \begin{align*}
 \left|\!\left|\!\left|  (\boldsymbol{u}+\delta_1 \boldsymbol{v},-p) \right|\!\right|\!\right|^2 & = \|\boldsymbol{u}+\delta_1 \boldsymbol{v}\|^2_{1,\Omega} +\|p\|^2_{0,\Omega } \\ & \leq 2(\|\boldsymbol{u}\|_{1,\Omega}^2 + \delta_1^2 \|\boldsymbol{v}\|_{1,\Omega}^2) + \|p \|^2_{0,\Omega} \\ & \leq    2(\|\boldsymbol{u}\|_{1,\Omega}^2 + \delta_1^2 \|p\|_{0,\Omega}^2) + \|p \|^2_{0,\Omega} 
  \\ & \leq \max \left\{2,(2 \delta_1^2+1) \right\} \left|\!\left|\!\left|  \boldsymbol{u},p\right|\!\right|\!\right|^2.
 \end{align*} 
This concludes the proof.
 \end{proof}
Since $\mathrm{V}_h$ is not a conforming space. Now, we define the conforming space $\mathrm{V}_h^c = \mathrm{V}_h \cap \mathrm{V}$. Finally, we decompose the approximate velocity  uniquely into $\boldsymbol{u}_h = \boldsymbol{u}_h^c + \boldsymbol{u}_h^r $, where $\boldsymbol{u}_h^c \in \mathrm{V}_h^c $ and $\boldsymbol{u}_h^r \in (\mathrm{V}_h^c)^{\perp} $, and we note that $ \boldsymbol{u}_h^r = \boldsymbol{u}_h - \boldsymbol{u}_h^c \in \mathrm{V}_h $.
 \begin{lemma}\label{SRL1}
     There holds 
     $$
     \| \boldsymbol{u}_h^r\|_{1,\mathcal{K}_h} \leq C_r \left( \sum_{K \in \mathcal{K}_h}  \boldsymbol{R}_{J_K}^2\right)^{1/2}.
     $$
 \end{lemma}
 \begin{proof}
     It follows straightforwadly from the decomposition $ \boldsymbol{u}_h = \boldsymbol{u}_h^r + \boldsymbol{u}_h^c$
     and from the facet residual as given in \cite{ MR2532864}. 
 \end{proof}
\begin{lemma}\label{SRL2}
    If $\| \boldsymbol{u}\|_{1,\infty} < M$ for sufficiently small $M$, then the following estimate holds:
   \begin{align*}
  \frac{C}{2} \left|\!\left|\!\left| ( \boldsymbol{e^u} , e^p) \right|\!\right|\!\right| & \leq \int_{\Omega} ( \boldsymbol{f} - \boldsymbol{f}_h ) \cdot \boldsymbol{v}  + \int_{\Omega} \boldsymbol{f}_h ( \boldsymbol{v} - \boldsymbol{v}_h )   - \mathcal{A}_h^{\mathrm{NS},\boldsymbol{u}_h}(\boldsymbol{u}_h,p_h; \boldsymbol{v}-\boldsymbol{v}_h,q ) + (1+C) C_r \left( \sum_{K \in \mathcal{K}_h} \boldsymbol{R}_{J_K}^2  \right)^{1/2} \\& \quad  + \sum_{K \in \mathcal{K}_h}  \tau \left( \boldsymbol{f} - \boldsymbol{f}_h ,-2 \nu \nabla \cdot \varepsilon (\boldsymbol{v}) + \boldsymbol{u} \cdot \nabla \boldsymbol{v} - \nabla q \right) + \sum_{K \in \mathcal{K}_h}  \tau  \left( \boldsymbol{f}_h, (\boldsymbol{u}- \boldsymbol{u}_h) \cdot \nabla \boldsymbol{v} \right) \\ & \quad  + \sum_{K \in \mathcal{K}_h}  \tau  \left( \boldsymbol{f}_h, - 2 \nu \nabla \cdot \varepsilon ( \boldsymbol{v}- \boldsymbol{v}_h)+ \boldsymbol{u}_h \cdot \nabla  ( \boldsymbol{v} - \boldsymbol{v}_h) - \nabla q \right) ,
 \end{align*}
  where $\boldsymbol{e}^{\boldsymbol{u}} := \boldsymbol{u} - \boldsymbol{u}_h$, $e^p:= p-p_h$.
    Moreover, $\boldsymbol{v}_h$ denotes the Cl\'{e}ment interpolation \cite{MR400739} of $\boldsymbol{v} \in \tilde{\mathrm{V}}$.
\end{lemma}
\begin{proof}
    Using $\boldsymbol{u}_h = \boldsymbol{u}_h^c + \boldsymbol{u}_h^r$, $\boldsymbol{e}_c^{\boldsymbol{u}} = \boldsymbol{u} - \boldsymbol{u}_h^c$ and the triangle inequality implies 
   $$
     \left|\!\left|\!\left|( \boldsymbol{e^u} , e^p) \right|\!\right|\!\right| \leq \left|\!\left|\!\left| (\boldsymbol{e}_c^{\boldsymbol{u}} , e^p )\right|\!\right|\!\right| + \|\boldsymbol{u}_h^r\|_{1,\mathcal{K}_h} \leq  \left|\!\left|\!\left| (\boldsymbol{e}_c^{\boldsymbol{u}}, e^p) \right|\!\right|\!\right| + C_r \left( \sum_{K \in \mathcal{K}_h} \boldsymbol{R}_{J_K}^2 \right)^{1/2}  
   $$
   where $(\boldsymbol{e}_c^{\boldsymbol{u}}, e^p) \in\mathrm{V}\times \Pi$. Then, Theorem \ref{GIS} gives 
   \begin{align*}
       C \left|\!\left|\!\left| (\boldsymbol{e_c^u}, e^p) \right|\!\right|\!\right| & \leq\mathcal{A}_h^{\mathrm{NS},\boldsymbol{u}_h} (\boldsymbol{e_c^u}, e^p; \boldsymbol{v},q)  \\ & \leq  \mathcal{A}_h^{\mathrm{NS},\boldsymbol{u}_h} (\boldsymbol{e^u}, e^p; \boldsymbol{v},q) + \mathcal{A}_h^{\mathrm{NS},\boldsymbol{u}_h} (\boldsymbol{u}_h^r, 0; \boldsymbol{v},q) \\ & \leq \mathcal{A}_h^{\mathrm{NS},\boldsymbol{u}_h} (\boldsymbol{e^u}, e^p; \boldsymbol{v},q) + C_r \left( \sum_{K \in \mathcal{K}_h} \boldsymbol{R}_{J_K}^2 \right)^{1/2}, 
   \end{align*}
   with $\left|\!\left|\!\left|(\boldsymbol{v},q) \right|\!\right|\!\right| \leq   1 $. Owing to the relation 
   $$
  \mathcal{A}_h^{\mathrm{NS},\boldsymbol{u}_h}(\boldsymbol{u},p;\boldsymbol{v},q) = \mathcal{A}_h^{\mathrm{NS},\boldsymbol{u}}(\boldsymbol{u},p;\boldsymbol{v},q)- \left( \boldsymbol{e^u} \cdot \nabla \boldsymbol{u}, \boldsymbol{v}\right)
   $$
   we then have 
   \begin{align*}
       C \left|\!\left|\!\left|( \boldsymbol{e^u}, e^p ) \right|\!\right|\!\right| & \leq \left|\!\left|\!\left|( \boldsymbol{e}_c^{\boldsymbol{u}}, e^p ) \right|\!\right|\!\right| + C_r \left( \sum_{K \in \mathcal{K}_h} \Psi_{J_K} ^2 \right)^{1/2} 
    \\   & \leq  \mathcal{A}_h^{\mathrm{NS},\boldsymbol{u}}(\boldsymbol{u},p; \boldsymbol{v},q ) - \left( \boldsymbol{e^u} \cdot \nabla \boldsymbol{u}, \boldsymbol{v}\right) - \mathcal{A}_h^{\mathrm{NS},\boldsymbol{u}_h}(\boldsymbol{u}_h,p_h; \boldsymbol{v},q ) +(1+C) C_r \left( \sum_{K \in \mathcal{K}_h} \boldsymbol{R}_{J_K}^2 \right)^{1/2}
   \end{align*}
   while using the properties $ \left( \boldsymbol{e^u} \cdot \nabla \boldsymbol{u}, \boldsymbol{v}\right) \leq C_1 M \|\boldsymbol{e^u}\|_{1, \mathcal{K}_h}$ yields the bounds
   $$
  C \left|\!\left|\!\left|(\boldsymbol{e^u}, e^p )\right|\!\right|\!\right|  \leq \mathcal{A}_h^{\mathrm{NS},\boldsymbol{u}}(\boldsymbol{u},p; \boldsymbol{v},q ) - \mathcal{A}_h^{\mathrm{NS},\boldsymbol{u}_h}(\boldsymbol{u}_h,p_h; \boldsymbol{v},q ) +(1+C) C_r \left( \sum_{K \in \mathcal{K}_h} \boldsymbol{R}_{J_K}^2 \right)^{1/2} - C_1 M  \left|\!\left|\!\left|(\boldsymbol{e^u}, e^p )\right|\!\right|\!\right|.
   $$
   Moreover, from \eqref{bb} we have 
   \begin{align*}
C_2 \left|\!\left|\!\left| ( \boldsymbol{e^u}, e^p ) \right|\!\right|\!\right| & \leq \int_{\Omega} \boldsymbol{f} \boldsymbol{v} + \sum_{K \in \mathcal{K}_h}  \tau \left( \boldsymbol{f} ,-2 \nu \nabla \cdot \varepsilon (\boldsymbol{v}) + \boldsymbol{u} \cdot \nabla \boldsymbol{v} - \nabla q \right) \\ & \quad  -  \mathcal{A}_h^{\mathrm{NS},\boldsymbol{u}_h}(\boldsymbol{u}_h,p_h; \boldsymbol{v},q )  +(1+C) C_r \left( \sum_{K \in \mathcal{K}_h} \boldsymbol{R}_{J_K}^2 \right)^{1/2}.
\end{align*}
By adding and subtracting \( \int_{\Omega} \boldsymbol{f}_h \boldsymbol{v} \) and \( \mathcal{A}_h^{\mathrm{NS},\boldsymbol{u}_h}\left(\boldsymbol{u}_h, p_h; \boldsymbol{v}_h, 0 \right) \), respectively, and utilizing the definition of the discrete weak formulation \eqref{DSF}, we obtain:
  \begin{align*}
  \frac{C}{2} \left|\!\left|\!\left| ( \boldsymbol{e^u} , e^p) \right|\!\right|\!\right| & \leq \int_{\Omega} ( \boldsymbol{f} - \boldsymbol{f}_h ) \cdot \boldsymbol{v}  + \int_{\Omega} \boldsymbol{f}_h ( \boldsymbol{v} - \boldsymbol{v}_h )   - \mathcal{A}_h^{\mathrm{NS},\boldsymbol{u}_h}(\boldsymbol{u}_h,p_h; \boldsymbol{v}-\boldsymbol{v}_h,q ) + (1+C) C_r \left( \sum_{K \in \mathcal{K}_h} \boldsymbol{R}_{J_K}^2  \right)^{1/2} \\& \quad  - \sum_{K \in \mathcal{K}_h}  \tau  \left( \boldsymbol{f}_h, - 2 \nu \nabla \cdot \varepsilon (  \boldsymbol{v}_h)+ \boldsymbol{u}_h \cdot \nabla  \boldsymbol{v}_h \right) + \sum_{K \in \mathcal{K}_h}  \tau \left( \boldsymbol{f} ,-2 \nu \nabla \cdot \varepsilon (\boldsymbol{v}) + \boldsymbol{u} \cdot \nabla \boldsymbol{v} - \nabla q \right) ,
 \end{align*}
 By adding and subtracting $\sum_{K \in \mathcal{K}_h}  \tau \left( \boldsymbol{f}_h ,-2 \nu \nabla \cdot \varepsilon (\boldsymbol{v}) + \boldsymbol{u} \cdot \nabla \boldsymbol{v} - \nabla q \right)$ and $\sum_{K \in \mathcal{K}_h}  \tau \left( \boldsymbol{f}_h, \boldsymbol{u}_h \cdot \nabla \boldsymbol{v} \right)$ in above equation, we get 
      \begin{align*}
  \frac{C}{2} \left|\!\left|\!\left| ( \boldsymbol{e^u} , e^p) \right|\!\right|\!\right| & \leq \int_{\Omega} ( \boldsymbol{f} - \boldsymbol{f}_h ) \cdot \boldsymbol{v}  + \int_{\Omega} \boldsymbol{f}_h ( \boldsymbol{v} - \boldsymbol{v}_h )   - \mathcal{A}_h^{\mathrm{NS},\boldsymbol{u}_h}(\boldsymbol{u}_h,p_h; \boldsymbol{v}-\boldsymbol{v}_h,q ) + (1+C) C_r \left( \sum_{K \in \mathcal{K}_h} \boldsymbol{R}_{J_K}^2  \right)^{1/2} \\& \quad  + \sum_{K \in \mathcal{K}_h}  \tau \left( \boldsymbol{f} - \boldsymbol{f}_h ,-2 \nu \nabla \cdot \varepsilon (\boldsymbol{v}) + \boldsymbol{u} \cdot \nabla \boldsymbol{v} - \nabla q \right) + \sum_{K \in \mathcal{K}_h}  \tau  \left( \boldsymbol{f}_h, (\boldsymbol{u}- \boldsymbol{u}_h) \cdot \nabla \boldsymbol{v} \right) \\ & \quad + \sum_{K \in \mathcal{K}_h}  \tau  \left( \boldsymbol{f}_h, - 2 \nu \nabla \cdot \varepsilon ( \boldsymbol{v}- \boldsymbol{v}_h)+ \boldsymbol{u}_h \cdot \nabla  ( \boldsymbol{v} - \boldsymbol{v}_h) - \nabla q \right) ,
 \end{align*}
\end{proof}
\begin{lemma}\label{SRL3}
    For $\left( \boldsymbol{v},q \right) \in \tilde{\mathrm{V}} \times \tilde{\Pi} $, there is $(\boldsymbol{u}_h, p_h) \in \mathrm{V}_h \times \Pi_h$ such that 
    \begin{align}\label{SRL3_eq1}
  &  \int_{\Omega} ( \boldsymbol{f} - \boldsymbol{f}_h ) \cdot \boldsymbol{v}  + \int_{\Omega} \boldsymbol{f}_h ( \boldsymbol{v} - \boldsymbol{v}_h )   - \mathcal{A}_h^{\mathrm{NS},\boldsymbol{u}_h}(\boldsymbol{u}_h,p_h; \boldsymbol{v}-\boldsymbol{v}_h,q )   + \sum_{K \in \mathcal{K}_h}  \tau \left( \boldsymbol{f} - \boldsymbol{f}_h ,-2 \nu \nabla \cdot \varepsilon (\boldsymbol{v}) + \boldsymbol{u} \cdot \nabla \boldsymbol{v} - \nabla q \right) \nonumber  \\&  + \sum_{K \in \mathcal{K}_h}  \tau  \left( \boldsymbol{f}_h, - 2 \nu \nabla \cdot \varepsilon ( \boldsymbol{v}- \boldsymbol{v}_h)+ \boldsymbol{u}_h \cdot \nabla  ( \boldsymbol{v} - \boldsymbol{v}_h) - \nabla q \right)
 \leq C \max\left\{\frac{m_K}{8}, \frac{\|\boldsymbol{u}_h\|_{0,\Omega}}{\nu} \right\}  \left( \Psi + \| \boldsymbol{f}-\boldsymbol{f}_h\|_{0,\Omega}   \right.\nonumber   \\ & \left. + \sum_{K \in \mathcal{K}_h} \frac{h_K^2}{\nu^{3/2}} \| \nabla \cdot \boldsymbol{u}_h \|_{0,K} \right) \left|\!\left|\!\left| \boldsymbol{v}, q \right|\!\right|\!\right|.
    \end{align}
\end{lemma}
\begin{proof}
    Using integration  by parts elementwise gives 
    \begin{align}\label{sum_eq}
             &  \int_{\Omega} ( \boldsymbol{f} - \boldsymbol{f}_h ) \cdot \boldsymbol{v}  + \int_{\Omega} \boldsymbol{f}_h ( \boldsymbol{v} - \boldsymbol{v}_h )   - \mathcal{A}_h^{\mathrm{NS},\boldsymbol{u}_h}(\boldsymbol{u}_h,p_h; \boldsymbol{v}-\boldsymbol{v}_h,q )   + \sum_{K \in \mathcal{K}_h}  \tau \left( \boldsymbol{f} - \boldsymbol{f}_h ,-2 \nu \nabla \cdot \varepsilon (\boldsymbol{v}) + \boldsymbol{u} \cdot \nabla \boldsymbol{v} - \nabla q \right) \nonumber  \\&  + \sum_{K \in \mathcal{K}_h}  \tau  \left( \boldsymbol{f}_h, - 2 \nu \nabla \cdot \varepsilon ( \boldsymbol{v}- \boldsymbol{v}_h)+ \boldsymbol{u}_h \cdot \nabla  ( \boldsymbol{v} - \boldsymbol{v}_h) - \nabla q \right)  = T_1 + T_2 + T_3 + T_4 +T_5, 
    \end{align}
    where we define the terms
    \begin{align*}
    &T_1 = \sum_{K \in \mathcal{K}_h} \int_K \left( \boldsymbol{f}_h + 2 \nu \nabla \cdot \varepsilon ( \boldsymbol{u}_h) - \nabla p_h - \boldsymbol{u}_h \cdot \nabla \boldsymbol{u}_h  \right) \left( \boldsymbol{v} - \boldsymbol{v}_h\right) +\int_\Omega \left( \boldsymbol{f} - \boldsymbol{f}_h\right) \boldsymbol{v} \\ &
    T_2 = \sum_{K \in \mathcal{K}_h} \int_K (\nabla \cdot \boldsymbol{u}_h) q  \\&
    T_3 = \sum_{K \in \mathcal{K}_h} \int_{\partial K_{in} } \bigg((p_h \boldsymbol{I} - 2 \nu \varepsilon (\boldsymbol{u}_h))   \cdot \boldsymbol{n} \bigg) (\boldsymbol{v}-\boldsymbol{v}_h) +  \sum_{E \in \Gamma_{\mathrm{Nav}}} \left[ \theta \int_E \boldsymbol{n}^t \left( 2 \nu \varepsilon (\boldsymbol{v}-\boldsymbol{v}_h) - q I\right) \boldsymbol{n} (\boldsymbol{n} \cdot \boldsymbol{u}_h)  \, ds  \right.  \\& \qquad \qquad \left. - \int_E \{\boldsymbol{n}^t(2 \nu \varepsilon (\boldsymbol{u}_h) -p_h I) \cdot \tau^i + \beta \sum_{i=1}^{d-1} (\tau^i \cdot \boldsymbol{u}_h) \} (\boldsymbol{v}- \boldsymbol{v}_h) \cdot \tau^i   \, ds - \frac{\gamma \nu}{h_E} \int_E (\boldsymbol{u}_h \cdot \boldsymbol{n})(\boldsymbol{v}-\boldsymbol{v}_h) \cdot \boldsymbol{n}  \, ds \right] \\ & 
    T_4 = \sum_{K \in \mathcal{K}_h} \int_K  \tau  (\boldsymbol{f}_h + 2\nu \nabla \cdot \varepsilon  (\boldsymbol{u}_h) - \boldsymbol{u}_h \cdot
    \nabla \boldsymbol{u}_h - \nabla p_h ) (-2 \nu \nabla \cdot \varepsilon (\boldsymbol{v} - \boldsymbol{v}_h) + \boldsymbol{u}_h \cdot \nabla (\boldsymbol{v}- \boldsymbol{v}_h ) - \nabla q) \\ & \qquad + \sum_{K \in \mathcal{K}_h}  \int_K \tau \left( \boldsymbol{f} - \boldsymbol{f}_h \right) \left(-2 \nu \nabla \cdot \varepsilon (\boldsymbol{v}) + \boldsymbol{u} \cdot \nabla \boldsymbol{v} - \nabla q \right) \\ & 
    T_5 =  \sum_{K \in \mathcal{K}_h} \int_K  \delta \nabla \cdot \boldsymbol{u}_h \nabla \cdot (\boldsymbol{v} - \boldsymbol{v}_h)
    \end{align*}
Applying the Cauchy-Schwarz inequality and Cl\'{e}ment interpolation estimates \cite{MR400739} to $T_1$ implies
\begin{align*}
    T_1 & \leq \left( \sum_{K \in \mathcal{K}_h} \frac{h_K^2}{\nu} \| \boldsymbol{R}_K \|_{0,K}^2 \right)^{1/2} \left( \sum_{K \in \mathcal{K}_h} \nu h_K^{-2} \|\boldsymbol{v} - \boldsymbol{v}_h \|_{0,K}^2 \right)^{1/2} +\int_\Omega \left( \boldsymbol{f} - \boldsymbol{f}_h\right) \boldsymbol{v}_h\\ &  \leq C \left( \sum_{K \in \mathcal{K}_h} \frac{h_K^2}{\nu} \| \boldsymbol{R}_K \|_{0,K}^2 \right)^{1/2} \| \nabla \boldsymbol{v}\|_{0,\Omega} +\| \boldsymbol{f} -\boldsymbol{f}_h \|_{0,\Omega} \|\boldsymbol{v} \|_{0,\Omega}.
\end{align*}
The bound for $T_2$ is defined as 
$ T_2 \leq \| \nabla \cdot \boldsymbol{u}_h\|_{0,\Omega} \| q\|_{0,\Omega}$.

Next, we rewrite \( T_3 \) in terms of a sum over interior facets and then apply the Cauchy-Schwarz inequality, the trace inequality, and the inverse inequality, we have
\begin{align*}
    T_3 
    &\leq \left( \sum_{E \in \mathcal{E}_{\Omega}}  \frac{h_E}{\nu} \|\boldsymbol{R}_E\|^2_{0,E}\right)^{1/2} \left( \sum_{E \in \mathcal{E}_{\Omega}} \frac{\nu}{h_E}  \|\boldsymbol{v}-\boldsymbol{v}_h\|_{0,E}^2 \right)^{1/2} + \theta \left(\sum_{E \in \mathcal{E}_{\mathrm{Nav}}} \frac{\nu}{h_E} \|\boldsymbol{v} - \boldsymbol{v}_h \|^2_{0,E} \right)^{1/2} \\
    &\quad \times  \left( \sum_{E \in \mathcal{E}_{\mathrm{Nav}}} \nu  \| \boldsymbol{u}_n \cdot \boldsymbol{n} \|^2_{0,E} \right)^{1/2}  + \| \boldsymbol{u}_h \cdot \boldsymbol{n} \|_{0,\Gamma_\mathrm{Nav}} \|q\|_{0,\Gamma_\mathrm{Nav}}  + \left( \sum_{E \in \mathcal{E}_\mathrm{Nav}} \frac{h_E}{\nu} \| \boldsymbol{R}_{J_K}^1 \|_{0,E}^2 \right)^{1/2} \\
    &\quad \times \left( \sum_{E \in \mathcal{E}_{\mathrm{Nav}}} \frac{\nu}{h_E}  \|\boldsymbol{v} - \boldsymbol{v}_h \|_{0,E}^2\right)^{1/2}  + \left(\frac{\gamma^2 \nu} {h_E} \sum_{E \in \mathcal{E}_\mathrm{Nav}}  \|\boldsymbol{u}_h \cdot \boldsymbol{n} \|^2_{0,E} \right)^{1/2}  \left(\sum_{E \in \mathcal{E}_\mathrm{Nav}}  \frac{\nu}{h_E}  \|\boldsymbol{v} - \boldsymbol{v}_h \|_{0,E}^2 \right)^{1/2} \\
    &\leq \left(\sum_{E \in \mathcal{E}_{\Omega}}  \frac{h_E}{\nu} \|\boldsymbol{R}_E\|^2_{0,E} + \sum_{E \in \mathcal{E}_\mathrm{Nav}} \bigg( \frac{h_E}{\nu} \| \boldsymbol{R}_{J_K}^1 \|_{0,E}^2 + \| \boldsymbol{R}_{J_K}^2\|^2_{0,E} \bigg)\right)^{1/2}  \left( \| \nabla \boldsymbol{v}\|_{0,\Omega}^2 + \|q\|^2_{0,\Omega}  \right)^{1/2}
\end{align*}
To establish the bound for \(T_4\), we apply the Cauchy-Schwarz inequality, Lemma \ref{H1Hinfty}, and the given condition \(\tau \leq \frac{m_K h_K^2}{8 \nu}\). This yields
\begin{align*}
T_4 & \leq \sum_{K \in \mathcal{K}_h} \frac{m_K h_K^2}{8 \nu} \|\boldsymbol{R}_K\|_{0,K} \left( h_K^{-2} \|\boldsymbol{v} - \boldsymbol{v}_h\|_{0,K} + h_K^{-1} \|\boldsymbol{u}_h\|_{\infty,K} \|\boldsymbol{v} - \boldsymbol{v}_h\|_{0,K} + \|\nabla q\|_{0,K} \right) \\
&  + \sum_{K \in \mathcal{K}_h} \frac{m_K h_K^2}{8 \nu} \|\boldsymbol{f} - \boldsymbol{f}_h\|_{0,K} \| - 2 \nu \nabla \cdot \varepsilon (\boldsymbol{v}) + \boldsymbol{u} \cdot \nabla \boldsymbol{v} - \nabla q \|_{0,K} \\
& \leq \sum_{K \in \mathcal{K}_h} \frac{m_K h_K^2}{8 \nu} \|\boldsymbol{R}_K\|_{0,K} \left( h_K^{-2} \|\boldsymbol{v} - \boldsymbol{v}_h\|_{0,K} + h_K^{-2} \|\boldsymbol{u}_h\|_{0,K} \|\boldsymbol{v} - \boldsymbol{v}_h\|_{0,K} + \|\nabla q\|_{0,K} \right) \\
&  + \sum_{K \in \mathcal{K}_h} \frac{m_K h_K^2}{8 \nu} \|\boldsymbol{f} - \boldsymbol{f}_h\|_{0,K} \| - 2 \nu \nabla \cdot \varepsilon (\boldsymbol{v}) + \boldsymbol{u} \cdot \nabla \boldsymbol{v} - \nabla q \|_{0,K} \\
& \leq C \max\left\{\frac{m_K}{8}, \frac{\|\boldsymbol{u}_h\|_{0,\Omega}}{\nu} \right\} \Bigg[ \left( \sum_{K \in \mathcal{K}_h} \frac{h_K^2}{\nu} \|\boldsymbol{R}_K\|_{0,K}^2 \right)^{1/2} \\
&  \times \left( \sum_{K \in \mathcal{K}_h} h_K^{-2} \nu \|\boldsymbol{v} - \boldsymbol{v}_h\|_{0,K}^2 + \sum_{K \in \mathcal{K}_h} h_K^{-2} \nu \|\boldsymbol{v} - \boldsymbol{v}_h\|_{0,K}^2 + \sum_{K \in \mathcal{K}_h} \frac{h_K^2}{\nu} \|\nabla q\|_{0,K}^2 \right)^{1/2} \\
&  + \left( \sum_{K \in \mathcal{K}_h} \frac{h_K^2}{\nu} \|\boldsymbol{f} - \boldsymbol{f}_h\|_{0,K}^2 \right)^{1/2} \left( \|\nabla \boldsymbol{v}\|^2_{0,\Omega} + \sum_{K \in \mathcal{K}_h} \frac{h_K^2}{\nu} \|\nabla q\|_{0,K}^2 \right)^{1/2} \Bigg] \\
& \leq C \max\left\{\frac{m_K}{8}, \frac{\|\boldsymbol{u}_h\|_{0,\Omega}}{\nu} \right\} \left( \sum_{K \in \mathcal{K}_h} \frac{h_K^2}{\nu} \|\boldsymbol{R}_K\|_{0,K}^2 + \frac{h_K^2}{\nu} \|\boldsymbol{f} - \boldsymbol{f}_h\|_{0,K}^2 \right)^{1/2} \\
&  \times \left( \|\nabla \boldsymbol{v}\|^2_{0,\Omega} + \sum_{K \in \mathcal{K}_h} \frac{h_K^2}{\nu} \|\nabla q\|_{0,K}^2 \right)^{1/2}.
\end{align*}
We need to calculate the bound for \(T_5\). Using the Cauchy-Schwarz inequality and the bound \(\delta \leq \frac{\lambda M^2 h_K^2 m_K^2}{4 \nu}\), we obtain
\begin{align*}
    T_5 & \leq  \sum_{K \in \mathcal{K}_h} \int_K \frac{\lambda M^2 m_K h_K^2}{4 \nu} \| \nabla \cdot \boldsymbol{u}_h \|_{0,K} \| \nabla \cdot (\boldsymbol{v} - \boldsymbol{v}_h) \|_{0,K} 
   \\ & \leq C \left( \frac{h_K^4}{\nu^3} \|\nabla \cdot \boldsymbol{u}_h \|^2_{0,K}\right)^{1/2} \left( \nu h_K^{-2} \| \boldsymbol{v}-\boldsymbol{v}_h \|^2_{0,K} \right)^{1/2}
   \\ & \leq C \left( \frac{h_K^4}{\nu^3} \|\nabla \cdot \boldsymbol{u}_h \|^2_{0,K}\right)^{1/2} \| \nabla \boldsymbol{v} \|_{0,\Omega}
\end{align*}
Finally, \eqref{SRL3_eq1} is obtained by combining the bounds derived for $T_1, T_2, T_3, T_4, T_5$ along with \eqref{sum_eq}.
\end{proof}
\begin{theorem}\label{reliability}
  Let $\left( \boldsymbol{u},p \right)$ be the regular solution to \eqref{1} and $\left( \boldsymbol{u}_h, p_h \right)$ a solution to \eqref{DSF}. Let $\Psi$ be the a  posteriori error estimator defined in \eqref{total_estimate}. If $\|\boldsymbol{u}\|_{1,\infty} < M $ for sufficiently small $M$, then the following estimate holds:
   \begin{align*}
   \left|\!\left|\!\left| (\boldsymbol{u}-\boldsymbol{u}_h , p -p_h)\right|\!\right|\!\right| & \leq C \max\left\{\frac{m_K}{8}, \frac{\|\boldsymbol{u}_h\|_{0,\Omega}}{\nu} \right\}  \left( \Psi + \| \boldsymbol{f}-\boldsymbol{f}_h\|_{0,\Omega} + \sum_{K \in \mathcal{K}_h} h_K \| \boldsymbol{f}-\boldsymbol{f}_h\|_{0,K} \right. \\ & \left. \quad  + \sum_{K \in \mathcal{K}_h} \frac{h_K^2}{\nu^{3/2}} \| \nabla \cdot \boldsymbol{u}_h \|_{0,K} \right),
  \end{align*}
    where $C>0$ is a constant independent of $h$.
    \begin{proof}
        Combining Lemmas \eqref{SRL2} and \eqref{SRL3} implies the stated result.
    \end{proof}
\end{theorem}
\subsection{Efficiency}
The efficiency of the estimator is studied using the standard technique of polynomial bubble functions, as outlined in \cite{MR2182149}. To this end, we introduce an interior bubble function $b_K$, which is defined and supported on an element $K$. For an internal edge $E$ shared by two elements $K$ and $K^{\prime}$, we define the patch $\omega_E = K \cup K^{\prime}$. Correspondingly, an edge bubble function $b_E$ is defined on $E$, ensuring it is positive within the patch interior and vanishes on its boundary. These bubble functions satisfy the following properties:
$$
b_K \in \mathrm{H}_{0}^1(K), \quad b_E \in H_{0}^1\left(\omega_E\right), \quad \text { and } \quad\left\|b_K\right\|_{L^{\infty}(K)}=\left\|b_E\right\|_{L^{\infty}(\omega_E)}=1 .
$$
The following estimates are well-known; see Verf\"urth
 \cite{MR3059294}.
\begin{lemma}\label{bubble}
For each element $K$ and edge $E$, the following estimates hold:
\begin{subequations}
\begin{align} 
    \left\|b_K \boldsymbol{v}\right\|_{0,K} &\leq C\| \boldsymbol{v} \|_{0,K}, \label{B1}\\ 
 \| \boldsymbol{v} \|_{0,K} &\leq C\| b_K^{1/2} \boldsymbol{v} \|_{0,K}, \label{B2}\\
\left\|\nabla (b_K \boldsymbol{v} )\right\|_{0,K} &\leq C h_K^{-1}\| \boldsymbol{v} \|_{0,K}, \label{B3} \\ 
\| \boldsymbol{v} \|_{0,E} & \leq C \| b_E^{1/2} \boldsymbol{v} \|_{0,E}, \label{B4}  \\
 \left\|b_E \boldsymbol{v} \right\|_{0,K} & \leq C h_E^{1/2}\| \boldsymbol{v} \|_{0,E},  \qquad \forall K \in \omega_E, \label{B5} \\
 \left\|\nabla (b_E \boldsymbol{v} )\right\|_{0,K} &\leq C h_E^{-1/2}\|\boldsymbol{v} \|_{0,E}, \qquad \forall K \in \omega_E \label{B6} ,
 \end{align}
\end{subequations}
where $ \boldsymbol{v} $ denotes a vector-valued polynomial function defined on the elements $K$ and the edges(facets) $E$.
\end{lemma}
In this lemma, we establish the efficiency bound for the internal residual $ \Psi_{R_K}$.
\begin{lemma}\label{efficiency_1}
Let $K$ be an element of $\mathcal{K}_h$.   The local equilibrium residual satisfies
 $$
 \Psi_{R_K}^2 \leq C \left( \| \boldsymbol{u} - \boldsymbol{u}_h \|^2_{1,K} + \|p-p_h\|_{0,K}^2 + h_K^2 \| \boldsymbol{f}- \boldsymbol{f}_h \|_{0,K} \right).
 $$
\end{lemma}
\begin{proof}
    For each $K \in \mathcal{K}_h$, we define $\boldsymbol{W}_b = b_K \boldsymbol{R}_K$. Then, using \eqref{B2}, we have 
    \begin{align*}
    \frac{1}{C^2} \| \boldsymbol{R}_K \|_{0,K}^2 &\leq \| b_K^{1/2} \boldsymbol{R}_K \|^2_{0,K} = \int_K \boldsymbol{R}_K \cdot  \boldsymbol{W}_b  \\ & 
    = \int_K \left( \boldsymbol{f}_h + 2 \nu \nabla \cdot \varepsilon (\boldsymbol{u}_h) - \boldsymbol{u}_h \cdot \nabla \boldsymbol{u}_h - \nabla p_h  \right) \cdot \boldsymbol{W}_b  
    \end{align*}
Noting that $\left(\boldsymbol{f} + 2 \nu \nabla \cdot \varepsilon ( \boldsymbol{u}) - \boldsymbol{u} \cdot \nabla \boldsymbol{u} - \nabla p \right)|_K = 0$ for the exact solution $(\boldsymbol{u}, p)$, we simply subtract, then integrate by parts and note that $\boldsymbol{W}_b|_{\partial K} = \boldsymbol{0}$, to give
\begin{align*}
 \frac{1}{C^2} \| \boldsymbol{R}_K \|_{0,K}^2 &\leq \int_K \left( \boldsymbol{f}_h - \boldsymbol{f} \right) \cdot \boldsymbol{W}_b \,  
    + \int_K \left[ -2 \nu \nabla \cdot \varepsilon  (\boldsymbol{u} - \boldsymbol{u}_h) + \nabla (p - p_h) \right] \cdot \boldsymbol{W}_b \,   \\ & \quad +  \int_K \left[  \left( \boldsymbol{u} - \boldsymbol{u}_h \right) \cdot \nabla  \boldsymbol{u} +  \boldsymbol{u}_h \cdot \nabla  \left( \boldsymbol{u} - \boldsymbol{u}_h \right) \right] \cdot \boldsymbol{W}_b \,  \\ & 
  \leq   T_1 + T_2
\end{align*}
where 
\begin{align*}
      T_1 &= \int_K \bigg( 2 \nu \varepsilon (\boldsymbol{u}- \boldsymbol{u}_h ) :  \varepsilon ({W}_b) - (p-p_h) \nabla \cdot \boldsymbol{W}_b \bigg)  + \int_K ( \boldsymbol{f}_h - \boldsymbol{f} ) \cdot \boldsymbol{W}_b  \\
      T_2 &=  \int_K \left[ \left( \boldsymbol{u} - \boldsymbol{u}_h \right) \cdot \nabla  \boldsymbol{u} +  \boldsymbol{u}_h \cdot \nabla  \left( \boldsymbol{u} - \boldsymbol{u}_h \right) \right] \cdot \boldsymbol{W}_b \, 
\end{align*}
Using Cauchy-Schwarz inequality and Lemma \ref{bubble}, we have 
\begin{align*}
T_1 &\leq  C_1 \left( \| \boldsymbol{u} - \boldsymbol{u}_h \|_{1,K} + \|p-p_h\|_{0,K} + h_K \|\boldsymbol{f}- \boldsymbol{f}_h \|_{0,K} \right) h_K^{-1} \| \boldsymbol{R}_K\|_{0,K} \\&
\leq  C_1 \left( \| \boldsymbol{u} - \boldsymbol{u}_h \|_{1,K}^2 + \|p-p_h\|_{0,K}^2 + h_K^2 \|\boldsymbol{f}- \boldsymbol{f}_h \|_{0,K}^2 \right)^{1/2} \left(h_K^{-2} \| \boldsymbol{R}_K\|_{0,K}^2\right)^{1/2} \\
 T_2 & \leq C_2 \| \boldsymbol{u} - \boldsymbol{u}_h \|_{1,K} h_K^{-1} \|\boldsymbol{R}_{K} \|_{0,K},
\end{align*}
and combining these bounds leads to the stated result. 
\end{proof}
In this lemma, we establish the efficiency bound for the edge residual $ \Psi_{R_e}$.
\begin{lemma}\label{efficiency_2}
    Let $K$ be an element of $\mathcal{K}_h$. The jump residual satisfies 
    $$
\Psi^2_{R_E} \leq \sum_{K \in \omega_E} \left( \| \boldsymbol{u} - \boldsymbol{u}_h \|^2_{1,K} + \|p-p_h \|_{0,K}^2 + h_K^2 \|\boldsymbol{f} - \boldsymbol{f}_h \|_{0,K}^2\right).
    $$
\end{lemma}
\begin{proof}
   Suppose $E$ be an interior facet (edge) and recall that the classical solution $(\boldsymbol{u},p)$ satisfies \\ $ \llbracket \left(p \boldsymbol{I} - 2 \nu \varepsilon (\boldsymbol{u}) \right) \boldsymbol{n}\rrbracket |_E =0. $ Now, define the localised jump term $\boldsymbol{W}_E =
   \sum_{E \in \partial K} \frac{h_E}{2 \nu} \boldsymbol{R}_E b_E$. Using Lemma \ref{bubble} gives 
 \begin{align*}
    \frac{h_E}{\nu} \|\boldsymbol{R}_E \|_{0,E}^2 &\leq C \left( \llbracket p_h \boldsymbol{I} - 2 \nu \varepsilon (\boldsymbol{u}_h )\rrbracket , \boldsymbol{W}_E \right)_E
    \\ & \leq C \left( \llbracket \left(p_h \boldsymbol{I} - 2 \nu \varepsilon (\boldsymbol{u}_h) \right) \boldsymbol{n}\rrbracket - \llbracket \left(p \boldsymbol{I} - 2 \nu \varepsilon (\boldsymbol{u}) \right) \boldsymbol{n}\rrbracket , \boldsymbol{W}_E \right)_E.
\end{align*}
Using integration by parts on each element of patch $\omega_e$ implies
\begin{align}\label{jump_bound}
\left( \llbracket \left(p_h \boldsymbol{I} - 2 \nu \varepsilon (\boldsymbol{u}_h) \right) \boldsymbol{n}\rrbracket - \llbracket \left(p \boldsymbol{I} - 2 \nu \varepsilon (\boldsymbol{u}) \right) \boldsymbol{n}\rrbracket , \boldsymbol{W}_E \right)_E &= \sum_{K \in \omega_E} \left\{ \int_K \left[ - 2 \nu \nabla \cdot \varepsilon (\boldsymbol{u}- \boldsymbol{u}_h) + \nabla (p-p_h) \right] \cdot \boldsymbol{W}_E \right. \nonumber  \\
& \quad + \left. \int_K \left[ -2 \nu \varepsilon (\boldsymbol{u} - \boldsymbol{u}_h ) + (p-p_h) \boldsymbol{I} \right] : \nabla \boldsymbol{W}_E \right\}
\end{align}
Since the exact solution $(\boldsymbol{u},p)$ satisfies $-2 \nu \nabla \cdot \varepsilon ( \boldsymbol{u}) + \boldsymbol{u} \cdot \nabla \boldsymbol{u} +  \nabla p |_K = \boldsymbol{f}|_K$, we have 
\begin{align}\label{bound_eq}
     \frac{h_E}{\nu} \|\boldsymbol{R}_E \|_{0,E}^2 & =  \sum_{K \in \omega_E} \int_K \{ \boldsymbol{f}_h + 2 \nu \nabla \cdot \varepsilon (\boldsymbol{u}_h) - \boldsymbol{u}_h \cdot \nabla \boldsymbol{u}_h - \nabla p_h \} \cdot \boldsymbol{W}_E + \sum_{K \in \omega_e} \int_K (\boldsymbol{f} - \boldsymbol{f}_h) \cdot \boldsymbol{W}_E \nonumber  \\ & \quad + \sum_{K \in \omega_E} \int_K \{ - 2 \nu \varepsilon (  \boldsymbol{u} - \boldsymbol{u}_h) +(p-p_h) \boldsymbol{I} \} : \nabla \boldsymbol{W}_E + \sum_{K \in \omega_e} \int_K \{ \boldsymbol{u}_h \cdot \nabla \boldsymbol{u}_h - \boldsymbol{u} \cdot \nabla \boldsymbol{u}\} \cdot \boldsymbol{W}_E \nonumber  \\ &   =
     T_1 + T_2 + T_3 +T_4 
\end{align}
These four terms will be bounded separately.
First, using the definition of $\boldsymbol{R}_K$, and then combining the Cauchy- Schwarz inequality with Lemma \ref{bubble}, gives
\begin{align*}
    T_1 &\leq C_1 \left(\sum_{K \in \omega_e} \frac{h_K^2}{\nu} \| \boldsymbol{R}_K  \|^2_{0,K} \right)^{1/2} \left( \sum_{K \in \omega_e} \nu  h_K^{-2} \|\boldsymbol{W}_E \|^2_{0,K} \right)^{1/2} \\
    T_2 & \leq C_2\left( \sum_{K \in \omega_e} h_K^2 \| \boldsymbol{f} - \boldsymbol{f}_h \|_{0,K}^2  \right)^{1/2} \left( \sum_{K \in \omega_e} h_K^{-2} \|\boldsymbol{W}_E \|^2_{0,K} \right)^{1/2}
\end{align*}
Next, given the shape regularity of the grid, using the definition of $\boldsymbol{W}_b$ and Lemma \ref{bubble} gives
$$
h_K^{-2} \|\boldsymbol{W}_E \|^2_{0,K}  \leq h_E^{-2} \| \boldsymbol{W}_E \|_{0,K}^2 \leq h_E^{-1} \left\| \frac{h_E}{\nu} \boldsymbol{R}_E \right\|_{0,E}^2.  
$$
Hence, the following estimate holds 
\begin{align*}
    T_1 &\leq C_1 \left(\sum_{K \in \omega_e} \frac{h_K^2}{\nu} \| \boldsymbol{R}_K  \|^2_{0,K} \right)^{1/2} \left( \sum_{K \in \omega_e} \frac{h_E}{\nu} \|\boldsymbol{R}_E \|^2_{0,K} \right)^{1/2}
    \\ & \leq  C_1 \left(\sum_{K \in \omega_e}  \| \boldsymbol{u} - \boldsymbol{u}_h \|^2_{1,K} + \|p-p_h\|_{0,K}^2  \right) \left( \sum_{K \in \omega_e} \frac{h_E}{\nu} \|\boldsymbol{R}_E \|^2_{0,K} \right)^{1/2} \\
    T_2 & \leq C_2 \left(\sum_{K \in \omega_e} \frac{h_K^2}{\nu} \| \boldsymbol{f} - \boldsymbol{f}_h \|_{0,K}^2  \right) \left( \sum_{K \in \omega_e} \frac{h_E }{\nu} \|\boldsymbol{R}_E \|^2_{0,K} \right)^{1/2}
\end{align*}
Hence, 
\begin{align*}
    T_1 +T_2 \lesssim \left(\sum_{K \in \omega_e}  \| \boldsymbol{u} - \boldsymbol{u}_h \|^2_{1,K} + \|p-p_h\|_{0,K}^2  + h_K^2 \| \boldsymbol{f} - \boldsymbol{f}_h \|_{0,K}^2  \right) \left( \sum_{K \in \omega_e} \frac{h_E}{\nu} \|\boldsymbol{R}_E \|^2_{0,K} \right)^{1/2}
\end{align*}
Similarly, 
\begin{align*}
 T_3 &\leq C_3 \left( \sum_{K \in \omega_e} \| \boldsymbol{u} - \boldsymbol{u}_h \|^2_{1,K} + \|p-p_h\|_{0,K}^2   \right)^{1/2}  \left(  \sum_{K \in \omega_e} \|\nabla \boldsymbol{W}_E\|^2_{0,E} \right)^{1/2}
   \end{align*}
   where this time the second term is bounded using \eqref{B6} i.e.
$\|\nabla \boldsymbol{W}_E\|^2_{0,K} \lesssim h_E^{-1} \left\| \frac{h_E}{\nu} \boldsymbol{R}_E\right\|^2_{0,E} $. This implies 
\begin{align*}
T_3 & \leq C_3 \left( \sum_{K \in \omega_e} \| \boldsymbol{u} - \boldsymbol{u}_h \|^2_{1,K} + \|p-p_h\|_{0,K}^2   \right)^{1/2}  \left(  \sum_{K \in \omega_e} \frac{h_E}{\nu} \| \boldsymbol{R}_E\|^2_{0,E} \right)^{1/2} \\
T_4 &\leq C_4 \left( \sum_{K \in \omega_e} \|\boldsymbol{u} - \boldsymbol{u}_h \|_{1,K}^2 \right)^{1/2} \left(  \sum_{K \in \omega_e} \frac{h_E}{\nu} \| \boldsymbol{R}_E\|^2_{0,E} \right)^{1/2}
   \end{align*}
   Combining the bounds of $T_1$, $T_2$, $T_3$ and $T_4$ with \eqref{jump_bound} and \eqref{bound_eq} implies the stated result. 
\end{proof}
In this lemma, we establish the efficiency bound for the trace residual $ \Psi_{J_K}$.
\begin{lemma}\label{efficiency_3}
Let $K$ be an element of $\mathcal{K}_h$. The local trace residual satisfies
$$
 \Psi_{J_K}^2 \lesssim \left( \| \boldsymbol{u} - \boldsymbol{u}_h \|^2_{1,K} + \|p-p_h\|_{0,K}^2  \right).
$$
\begin{proof}
Noting that the conditions 
\[
\left. \boldsymbol{u} \cdot \boldsymbol{n} \right|_E = 0 \quad \text{and} \quad 
\left. \left(2 \nu \boldsymbol{n}^\top \varepsilon(\boldsymbol{u}) \tau^i + \beta \, \boldsymbol{u} \cdot \tau^i \right) \right|_E = 0
\]
holds for all \( E \in \Gamma_{\mathrm{Nav}} \), where \( \boldsymbol{u} \) is the regular solution of \eqref{1}, it follows that
$$
\begin{aligned}
\Psi_{J_K}^2: & =
\sum_{E \in \Gamma_{\mathrm{Nav}}} \left(   \frac{h_E}{\nu} \boldsymbol{R}_{J_K}^1+ \boldsymbol{R}_{J_K}^2 \right) \\& = \sum_{E \in \Gamma_{\mathrm{Nav}} } \left(\frac{h_E}{\nu}  \|\sum_{i=1}^{d-1} \left(2 \nu \boldsymbol{n}^t \varepsilon (\boldsymbol{u}_h) \tau^i + \beta \boldsymbol{u}_h \cdot \tau^{i} \right)\|_{0,E}^2 + \frac{\gamma^2 \nu }{h_E}  \| \boldsymbol{u}_h \cdot \boldsymbol{n} \|_{0,E}^2 \right) \\
& = \sum_{E \in \Gamma_{\mathrm{Nav}} } \left(\frac{h_E}{\nu}  \|\sum_{i=1}^{d-1} \left(2 \nu \boldsymbol{n}^t \varepsilon (\boldsymbol{u}_h - \boldsymbol{u}) \tau^i + \beta (\boldsymbol{u}_h - \boldsymbol{u}) \cdot \tau^{i} \right)\|_{0,E}^2 + \frac{\gamma^2 \nu }{h_E}  \| (\boldsymbol{u}_h - \boldsymbol{u}) \cdot \boldsymbol{n} \|_{0,E}^2 \right) 
\end{aligned}
$$
Moreover, we have
$$
 \Psi_{J_K}^2 \lesssim \left( \| \boldsymbol{u} - \boldsymbol{u}_h \|^2_{1,K} + \|p-p_h\|_{0,K}^2  \right).
$$
\end{proof}
\end{lemma}
\begin{theorem}
  Let $(\boldsymbol{u}, p)$ and $\left(\boldsymbol{u}_h, p_h \right)$ be the unique solutions of problems \eqref{1} and \eqref{DSF}, respectively. Let $\Psi$ be defined as in (3.3). Then there exists a constant $C>0$ that is independent of $h$ such that
$$
\Psi \leq C\left( \left|\!\left|\!\left| (\boldsymbol{u}-\boldsymbol{u}_h , p -p_h)\right|\!\right|\!\right| +\left(\sum_{K \in \mathcal{K}_h} h_K^2\left\|\boldsymbol{f}-\boldsymbol{f}_h\right\|_{0, K}^2\right)^{1 / 2}\right) .
$$
\begin{proof}
     Combining Lemmas \eqref{efficiency_1}, \eqref{efficiency_2} and \eqref{efficiency_3} implies the stated result.
\end{proof}
\end{theorem}
\section{Numerical experiments}\label{section8}
All routines have been implemented using the open-source finite element library FEniCS \cite{alnaes2015fenics}, and the solvers used in this work are monolithic. We used the MUMPS distributed direct solver \cite{amestoy2000mumps} for the linear systems in all examples and employed the PETSc Nonlinear Solvers (SNES) in Example 4.
 In some experiments, we employ uniform meshes, while in others, we use adaptive mesh refinement.  Specifically, starting with an initial mesh $\mathcal{K}_{0,\Omega}$, we apply the iterative refinement loop
\begin{align*} 
\text{Solve} \rightarrow \text{Estimate} \rightarrow \text{Mark} \rightarrow \text{Refine} \end{align*}
to generate a sequence of (nested) regular meshes $\left\{\mathcal{K}_{\ell}\right\}$ with mesh size $h_{\ell}$. 
At each step, we compute the local error estimators $\Psi_{K}$ for all $K$ over the previous mesh $\mathcal{K}_h$ and refine those elements $K \in \mathcal{K}_h$ according to
\begin{align*} 
\Psi_{K} \geq \tilde{\theta} \max\{\Psi_{K} : K \in \mathcal{K}_h\}, 
\end{align*}
where $\tilde{\theta}  \in (0, 1)$ is a prescribed parameter.  We assess the quality of the a posteriori error estimator through the so-called effectivity index, which is required to remain bounded as $h$ approaches zero and is defined by
\begin{align*}
    \mathrm{Effec} := \frac{\Psi}{\|(u - \boldsymbol{u}_h, p - p_h)\|}. 
\end{align*}
\begin{remark}
  In the case where $\nu \ll 1$, the numerical algorithm requires a continuation strategy to reach the target viscosity. This strategy begins with a relatively large viscosity value, which is gradually reduced to the desired value. This process improves the initial guess and ensures the convergence of the Newton solver. In all numerical experiments, we set $\lambda = 1$ and $m_K = 0.0814814$, as specified in \cite{MR1186727}, to satisfy the stabilization parameters.
\end{remark}
\subsection{Analytic solution}
The computational domain is $\Omega := (0,1) \times (0,1)$, and we consider two viscosity values, $\nu = \{1, 0.01\}$. The function \(\boldsymbol{f}\) is chosen such that the exact solution is given by 
\begin{align*}
\boldsymbol{u}(x, y) &= \left(-256 x^2 (x-1)^2 y (y-1)(2y-1), \, 256 x^2 (x-1)^2 y (y-1)(2y-1) \right), \\
p(x, y) &= 150 (x - 0.5)(y - 0.5).
\end{align*}

The slip boundary condition is imposed on $y = 1$, while the essential boundary condition is enforced on the remainder of the boundary. Tables~\ref{table1} and~\ref{table2} present the results for $\mathbb{P}_1^2-\mathbb{P}_1$, and Tables~\ref{table3} and~\ref{table4} for $\mathbb{P}_2^2-\mathbb{P}_2$ under uniform refinement. These results demonstrate that the method maintains accuracy even for small viscosity coefficients. The method achieves the optimal order of convergence for both $\mathbb{P}_1^2-\mathbb{P}_1$ and $\mathbb{P}_2^2-\mathbb{P}_2$ spaces with parameters $\gamma = 10$ and $\beta = 10$. Tables~\ref{table1}--\ref{table4} further show that the effectivity index ($E$) remains bounded as $h \to 0$ for various values of $\nu$ and $\theta$.

The total error (T.E.) is defined as
\[
\left|\!\left|\!\left| (\boldsymbol{u} - \boldsymbol{u}_h, p - p_h) \right|\!\right|\!\right| = \nu \|\varepsilon(\boldsymbol{v} - \boldsymbol{v}_h) \|^2_{0,\Omega} + \sum_{E \in \mathcal{E}_{\mathrm{Nav}}} \frac{\nu}{h_E} \|(\boldsymbol{v} - \boldsymbol{v}_h) \cdot \boldsymbol{n} \|_{0,E}^2 +  \| p - p_h \|^2_{0,\Omega}.
\]

Table~\ref{table5} shows the error in the $L_2$ norm for the slip condition on $\Gamma_\mathrm{Nav}$. The results indicate that there are no significant differences in error between the symmetric and skew-symmetric variants. However, it shows that the larger the Nitsche parameter  $\gamma$ is, the smaller is the error on the slip condition, for both variants. 
\begin{table}[t!]
    \centering
    \begin{footnotesize}
    \caption{Analytical solution for $\nu = 1$, $\gamma = 10$, $\beta = 10$, and $\mathbb{P}_1^2 - \mathbb{P}_1$ elements under uniform refinement for different values of $\theta$.}
    \begin{tabular}{cccccccccccc}
    \toprule
    \multirow{2}{*}{$h$} & \multicolumn{2}{c}{$\|p-p_h\|_{0,\Omega}$} & \multicolumn{2}{c}{$\|\boldsymbol{u}-\boldsymbol{u}_h\|_{0,\Omega}$} & \multicolumn{2}{c}{$\|\boldsymbol{u}-\boldsymbol{u}_h\|_{1,\Omega}$} & \multicolumn{2}{c}{T.E.} & \multicolumn{2}{c}{$\Psi$} & \multirow{2}{*}{Effec} \\
    \cmidrule(lr){2-3} \cmidrule(lr){4-5} \cmidrule(lr){6-7} \cmidrule(lr){8-9} \cmidrule(lr){10-11}
     & Error & Rate & Error & Rate & Error & Rate & Error & Rate & Error & Rate & \\
    \midrule
    \multicolumn{12}{c}{$\theta=1$} \\
    \midrule
    0.3536 & 4.99e+00 & 0.00 & 3.96e-01 & 0.00 & 5.38e+00 & 0.00 & 6.99e+00 & 0.00 & 4.92e+01 & 0.00 & 7.03 \\
    0.1768 & 1.76e+00 & 1.50 & 8.95e-02 & 2.14 & 2.48e+00 & 1.12 & 3.19e+00 & 1.13 & 2.33e+01 & 1.08 & 7.30 \\
    0.0884 & 5.94e-01 & 1.56 & 2.17e-02 & 2.05 & 1.25e+00 & 0.99 & 1.56e+00 & 1.03 & 1.14e+01 & 1.03 & 7.30 \\
    0.0442 & 2.03e-01 & 1.55 & 5.24e-03 & 2.05 & 6.20e-01 & 1.01 & 7.70e-01 & 1.02 & 5.63e+00 & 1.02 & 7.32 \\
    0.0221 & 7.15e-02 & 1.50 & 1.27e-03 & 2.04 & 3.09e-01 & 1.01 & 3.81e-01 & 1.01 & 2.79e+00 & 1.01 & 7.32 \\
    0.0110 & 2.57e-02 & 1.47 & 3.10e-04 & 2.02 & 1.54e-01 & 1.00 & 1.90e-01 & 1.01 & 1.39e+00 & 1.01 & 7.33 \\
    0.0055 & 9.50e-03 & 1.43 & 8.00e-05 & 2.00 & 7.69e-02 & 1.00 & 9.48e-02 & 1.00 & 6.95e-01 & 1.00 & 7.34 \\
    \midrule
    \multicolumn{12}{c}{$\theta=-1$} \\
    \midrule
    0.3536 & 1.16e+01 & 0.00 & 7.43e-01 & 0.00 & 6.63e+00 & 0.00 & 1.11e+01 & 0.00 & 7.54e+01 & 0.00 & 6.78 \\
    0.1768 & 2.92e+00 & 1.99 & 1.72e-01 & 2.11 & 2.72e+00 & 1.29 & 4.28e+00 & 1.38 & 2.77e+01 & 1.44 & 6.47 \\
    0.0884 & 9.50e-01 & 1.62 & 4.28e-02 & 2.01 & 1.29e+00 & 1.07 & 2.01e+00 & 1.09 & 1.25e+01 & 1.15 & 6.22 \\
    0.0442 & 3.25e-01 & 1.55 & 1.07e-02 & 2.00 & 6.31e-01 & 1.04 & 9.77e-01 & 1.04 & 5.91e+00 & 1.08 & 6.05 \\
    0.0221 & 1.14e-01 & 1.51 & 2.67e-03 & 2.00 & 3.11e-01 & 1.02 & 4.82e-01 & 1.02 & 2.86e+00 & 1.05 & 5.93 \\
    0.0110 & 4.05e-02 & 1.49 & 6.70e-04 & 1.99 & 1.54e-01 & 1.01 & 2.40e-01 & 1.01 & 1.41e+00 & 1.02 & 5.87 \\
    0.0055 & 1.46e-02 & 1.47 & 1.70e-04 & 1.99 & 7.70e-02 & 1.00 & 1.20e-01 & 1.00 & 6.99e-01 & 1.01 & 5.84 \\
    \midrule
    \multicolumn{12}{c}{$\theta=0$} \\
    \midrule
    0.3536 & 5.72e+00 & 0.00 & 4.50e-01 & 0.00 & 5.43e+00 & 0.00 & 7.94e+00 & 0.00 & 5.49e+01 & 0.00 & 6.92 \\
    0.1768 & 1.70e+00 & 1.75 & 1.04e-01 & 2.12 & 2.48e+00 & 1.13 & 3.50e+00 & 1.18 & 2.41e+01 & 1.19 & 6.89 \\
    0.0884 & 5.59e-01 & 1.60 & 2.52e-02 & 2.04 & 1.24e+00 & 0.99 & 1.72e+00 & 1.03 & 1.16e+01 & 1.06 & 6.74 \\
    0.0442 & 1.88e-01 & 1.57 & 6.17e-03 & 2.03 & 6.19e-01 & 1.01 & 8.52e-01 & 1.02 & 5.68e+00 & 1.03 & 6.68 \\
    0.0221 & 6.59e-02 & 1.51 & 1.52e-03 & 2.03 & 3.08e-01 & 1.01 & 4.23e-01 & 1.01 & 2.80e+00 & 1.02 & 6.63 \\
    0.0110 & 2.37e-02 & 1.48 & 3.80e-04 & 2.01 & 1.54e-01 & 1.00 & 2.11e-01 & 1.00 & 1.39e+00 & 1.01 & 6.61 \\
    0.0055 & 8.84e-03 & 1.42 & 9.00e-05 & 2.00 & 7.69e-02 & 1.00 & 1.05e-01 & 1.00 & 6.96e-01 & 1.00 & 6.60 \\
    \bottomrule
    \end{tabular}
    \label{table1}
    \end{footnotesize}
    \end{table}
\begin{table}[t!]
    \centering
    \begin{footnotesize}
    \caption{Analytical solution for $\nu = 0.01$, $\gamma = 10$, $\beta = 10$, and $\mathbb{P}_1^2 - \mathbb{P}_1$ elements under uniform refinement for different values of $\theta$.
}
    \begin{tabular}{cccccccccccc}
    \toprule
    \multirow{2}{*}{$h$} & \multicolumn{2}{c}{$\|p-p_h\|_{0,\Omega}$} & \multicolumn{2}{c}{$\|\boldsymbol{u}-\boldsymbol{u}_h\|_{0,\Omega}$} & \multicolumn{2}{c}{$\|\boldsymbol{u}-\boldsymbol{u}_h\|_{1,\Omega}$} & \multicolumn{2}{c}{T.E.} & \multicolumn{2}{c}{$\Psi$} & \multirow{2}{*}{Effec} \\
    \cmidrule(lr){2-3} \cmidrule(lr){4-5} \cmidrule(lr){6-7} \cmidrule(lr){8-9} \cmidrule(lr){10-11}
     & Error & Rate & Error & Rate & Error & Rate & Error & Rate & Error & Rate & \\
     \midrule
    \multicolumn{12}{c}{$\theta=1$} \\
   \midrule
    0.3536 &  1.03e+00 &  0.00 &  1.13e+00 &  0.00 &  1.01e+01 &  0.00 &  1.25e+01 &  0.00 &  5.59e+01 &  0.00 &  4.47 \\  
    0.1768 &  4.49e-01 &  1.20 &  6.84e-01 &  0.73 &  8.25e+00 &  0.29 &  8.51e+00 &  0.55 &  1.65e+01 &  1.76 &  1.94   \\
    0.0884 &  7.49e-02 &  2.59 &  6.21e-02 &  3.46 &  1.54e+00 &  2.42 &  1.62e+00 &  2.39 &  4.02e+00 &  2.04 &  2.49   \\
    0.0442 &  1.80e-02 &  2.06 &  1.36e-02 &  2.19 &  6.51e-01 &  1.24 &  6.62e-01 &  1.29 &  1.13e+00 &  1.83 &  1.71   \\
    0.0221 &  4.29e-03 &  2.07 &  2.93e-03 &  2.21 &  3.11e-01 &  1.07 &  3.13e-01 &  1.08 &  3.74e-01 &  1.60 &  1.20   \\
    0.0110 &  1.06e-03 &  2.01 &  6.90e-04 &  2.10 &  1.54e-01 &  1.01 &  1.54e-01 &  1.02 &  1.53e-01 &  1.29 &  0.99   \\
    0.0055 &  2.70e-04 &  1.97 &  1.70e-04 &  2.03 &  7.69e-02 &  1.00 &  7.69e-02 &  1.00 &  7.13e-02 &  1.10 &  0.93   \\
   \midrule
    \multicolumn{12}{c}{\textbf{$\theta=-1$}} \\
   \midrule
    0.3536 &  1.04e+00 &  0.00 &  1.14e+00 &  0.00 &  1.01e+01 &  0.00 &  1.25e+01 &  0.00 &  5.59e+01 &  0.00 &  4.47 \\  
    0.1768 &  4.51e-01 &  1.20 &  6.89e-01 &  0.72 &  8.30e+00 &  0.29 &  8.56e+00 &  0.55 &  1.65e+01 &  1.76 &  1.93   \\
    0.0884 &  7.59e-02 &  2.57 &  6.38e-02 &  3.43 &  1.55e+00 &  2.42 &  1.63e+00 &  2.40 &  4.02e+00 &  2.04 &  2.47   \\
    0.0442 &  1.82e-02 &  2.06 &  1.40e-02 &  2.19 &  6.52e-01 &  1.25 &  6.64e-01 &  1.29 &  1.13e+00 &  1.83 &  1.71   \\
    0.0221 &  4.35e-03 &  2.07 &  3.04e-03 &  2.21 &  3.11e-01 &  1.07 &  3.13e-01 &  1.08 &  3.74e-01 &  1.60 &  1.19   \\
    0.0110 &  1.08e-03 &  2.01 &  7.11e-04 &  2.09 &  1.54e-01 &  1.01 &  1.54e-01 &  1.02 &  1.53e-01 &  1.29 &  0.99   \\
    0.0055 &  2.81e-04 &  1.97 &  1.72e-04 &  2.03 &  7.69e-02 &  1.00 &  7.69e-02 &  1.00 &  7.13e-02 &  1.10 &  0.93   \\
    \midrule
    \multicolumn{12}{c}{$\theta=0$} \\
   \midrule
     0.3536 &  1.04e+00 &  0.00 &  1.13e+00 &  0.00 &  1.01e+01 &  0.00 &  1.25e+01 &  0.00 &  5.59e+01 &  0.00 &  4.47 \\  
    0.1768 &  4.50e-01 &  1.20 &  6.86e-01 &  0.72 &  8.28e+00 &  0.29 &  8.53e+00 &  0.55 &  1.65e+01 &  1.76 &  1.94   \\
    0.0884 &  7.54e-02 &  2.58 &  6.29e-02 &  3.45 &  1.55e+00 &  2.42 &  1.62e+00 &  2.40 &  4.02e+00 &  2.04 &  2.48   \\
    0.0442 &  1.81e-02 &  2.06 &  1.38e-02 &  2.19 &  6.52e-01 &  1.25 &  6.63e-01 &  1.29 &  1.13e+00 &  1.83 &  1.71   \\
    0.0221 &  4.32e-03 &  2.07 &  2.98e-03 &  2.21 &  3.11e-01 &  1.07 &  3.13e-01 &  1.08 &  3.74e-01 &  1.60 &  1.20   \\
    0.0110 &  1.07e-03 &  2.01 &  7.02e-04 &  2.09 &  1.54e-01 &  1.01 &  1.54e-01 &  1.02 &  1.53e-01 &  1.29 &  0.99   \\
    0.0055 &  2.74e-04 &  1.97 &  1.72e-04 &  2.03 &  7.69e-02 &  1.00 &  7.69e-02 &  1.00 &  7.13e-02 &  1.10 &  0.93   \\
   \bottomrule
    \end{tabular}
    \label{table2}
    \end{footnotesize}
\end{table}
\begin{table}[t!]
    \centering
    \begin{footnotesize}
    \caption{Analytical solution for $\nu = 1$, $\gamma = 10$, $\beta = 10$, and $\mathbb{P}_2^2 - \mathbb{P}_2$ elements under uniform refinement for different values of $\theta$.}
    \begin{tabular}{cccccccccccc}
    \toprule
    \multirow{2}{*}{$h$} & \multicolumn{2}{c}{$\|p-p_h\|_{0,\Omega}$} & \multicolumn{2}{c}{$\|\boldsymbol{u}-\boldsymbol{u}_h\|_{0,\Omega}$} & \multicolumn{2}{c}{$\|\boldsymbol{u}-\boldsymbol{u}_h\|_{1,\Omega}$} & \multicolumn{2}{c}{T.E.} & \multicolumn{2}{c}{$\Psi$} & \multirow{2}{*}{Effec} \\
    \cmidrule(lr){2-3} \cmidrule(lr){4-5} \cmidrule(lr){6-7} \cmidrule(lr){8-9} \cmidrule(lr){10-11}
     & Error & Rate & Error & Rate & Error & Rate & Error & Rate & Error & Rate & \\
     \midrule
    \multicolumn{12}{c}{$\theta=1$} \\
   \midrule
    0.3536 &  1.06e+00 &  0.00 &  4.23e-02 &  0.00 &  1.11e+00 &  0.00 &  1.90e+00 &  0.00 &  1.17e+01 &  0.00 &  6.15 \\  
    0.1768 &  2.58e-01 &  2.04 &  1.16e-02 &  1.86 &  3.30e-01 &  1.75 &  4.54e-01 &  2.06 &  3.09e+00 &  1.92 &  6.81   \\
    0.0884 &  6.45e-02 &  2.00 &  3.07e-03 &  1.92 &  8.50e-02 &  1.96 &  1.12e-01 &  2.02 &  7.95e-01 &  1.96 &  7.08   \\
    0.0442 &  1.64e-02 &  1.98 &  7.80e-04 &  1.98 &  2.14e-02 &  1.99 &  2.83e-02 &  1.99 &  2.02e-01 &  1.97 &  7.15   \\
    0.0221 &  4.14e-03 &  1.98 &  1.90e-04 &  2.03 &  5.34e-03 &  2.00 &  7.18e-03 &  1.98 &  5.10e-02 &  1.99 &  7.11   \\
    0.0110 &  1.04e-03 &  1.99 &  5.00e-06 &  2.00 &  1.33e-03 &  2.00 &  1.82e-03 &  1.98 &  1.28e-02 &  2.00 &  7.05   \\
    0.0055 &  2.60e-04 &  2.00 &  1.00e-07 &  1.98 &  3.30e-04 &  2.00 &  4.60e-04 &  1.99 &  3.21e-03 &  1.99 &  7.02  \\
        \midrule
        \multicolumn{12}{c}{\textbf{$\theta=-1$}} \\
        \midrule
    0.3536 &  9.50e-01 &  0.00 &  6.17e-02 &  0.00 &  1.07e+00 &  0.00 &  1.45e+00 &  0.00 &  1.07e+01 &  0.00 &  7.41 \\
    0.1768 &  2.52e-01 &  1.92 &  1.18e-02 &  2.39 &  3.27e-01 &  1.71 &  4.03e-01 &  1.85 &  3.06e+00 &  1.81 &  7.59 \\
    0.0884 &  6.44e-02 &  1.97 &  3.07e-03 &  1.94 &  8.46e-02 &  1.95 &  1.04e-01 &  1.95 &  7.96e-01 &  1.94 &  7.65 \\
    0.0442 &  1.64e-02 &  1.98 &  7.80e-04 &  1.98 &  2.13e-02 &  1.99 &  2.64e-02 &  1.98 &  2.03e-01 &  1.98 &  7.68 \\
    0.0221 &  4.14e-03 &  1.98 &  1.90e-04 &  2.03 &  5.34e-03 &  2.00 &  6.64e-03 &  1.99 &  5.10e-02 &  1.99 &  7.69 \\
    0.0110 &  1.04e-03 &  1.99 &  5.00e-05 &  2.00 &  1.33e-03 &  2.00 &  1.67e-03 &  1.99 &  1.28e-02 &  2.00 &  7.69 \\
    0.0055 &  2.60e-04 &  2.00 &  1.00e-05 &  1.98 &  3.30e-04 &  2.00 &  4.20e-04 &  1.99 &  3.21e-03 &  1.99 &  7.68 \\
       \midrule
       \multicolumn{12}{c}{$\theta=0$} \\
      \midrule
  0.3536 &  8.44e-01 &  0.00 &  5.05e-02 &  0.00 &  9.98e-01 &  0.00 &  1.48e+00 &  0.00 &  1.04e+01 &  0.00 &  7.06 \\  
    0.1768 &  2.44e-01 &  1.79 &  1.16e-02 &  2.12 &  3.24e-01 &  1.62 &  4.09e-01 &  1.85 &  3.03e+00 &  1.78 &  7.42   \\
   0.0884 &  6.36e-02 &  1.95 &  3.06e-03 &  1.92 &  8.43e-02 &  1.94 &  1.05e-01 &  1.96 &  7.92e-01 &  1.94 &  7.52   \\
   0.0442 &  1.63e-02 &  1.97 &  7.80e-04 &  1.98 &  2.13e-02 &  1.99 &  2.67e-02 &  1.98 &  2.02e-01 &  1.97 &  7.57   \\
   0.0221 &  4.12e-03 &  1.98 &  1.90e-04 &  2.03 &  5.33e-03 &  2.00 &  6.73e-03 &  1.99 &  5.10e-02 &  1.99 &  7.58   \\
  0.0110 &  1.04e-03 &  1.99 &  5.00e-05 &  2.00 &  1.33e-03 &  2.00 &  1.69e-03 &  1.99 &  1.28e-02 &  2.00 &  7.57   \\
  0.0055 &  2.60e-04 &  2.00 &  1.00e-05 &  1.98 &  3.30e-04 &  2.00 &  4.20e-04 &  1.99 &  3.21e-03 &  1.99 &  7.57  \\
    \bottomrule
    \end{tabular}
    \label{table3}
    \end{footnotesize}
\end{table}
\begin{table}[t!]
    \centering
    \begin{footnotesize}
    \caption{Analytical solution for $\nu = 0.01$, $\gamma = 10$, $\beta = 10$, and $\mathbb{P}_2^2 - \mathbb{P}_2$ elements under uniform refinement for different values of $\theta$.}
  \begin{tabular}{cccccccccccc}
    \toprule
    \multirow{2}{*}{$h$} & \multicolumn{2}{c}{$\|p-p_h\|_{0,\Omega}$} & \multicolumn{2}{c}{$\|\boldsymbol{u}-\boldsymbol{u}_h\|_{0,\Omega}$} & \multicolumn{2}{c}{$\|\boldsymbol{u}-\boldsymbol{u}_h\|_{1,\Omega}$} & \multicolumn{2}{c}{T.E.} & \multicolumn{2}{c}{$\Psi$} & \multirow{2}{*}{Effec} \\
    \cmidrule(lr){2-3} \cmidrule(lr){4-5} \cmidrule(lr){6-7} \cmidrule(lr){8-9} \cmidrule(lr){10-11}
     & Error & Rate & Error & Rate & Error & Rate & Error & Rate & Error & Rate & \\
     \midrule
    \multicolumn{12}{c}{$\theta=1$} \\
   \midrule
    0.3536 &  3.58e-02 &  0.00 &  6.25e-02 &  0.00 &  1.21e+00 &  0.00 &  1.23e+00 &  0.00 & 4.66e+01 &  0.00 & 3.79e+01 \\
    0.1768 &  1.50e-02 &  1.25 &  1.91e-02 &  1.71 &  4.07e-01 &  1.57 &  4.08e-01 &  1.59 &  1.16e+01 &  2.01 & 2.83e+01 \\
    0.0884 &  5.04e-03 &  1.58 &  5.93e-03 &  1.69 &  1.11e-01 &  1.87 &  1.11e-01 &  1.87 &  2.93e+00 &  1.98 & 2.63e+01 \\
    0.0442 &  1.31e-03 &  1.94 &  1.53e-03 &  1.95 &  2.81e-02 &  1.98 &  2.82e-02 &  1.98 &  7.34e-01 &  2.00 & 2.61e+01 \\
    0.0221 &  3.31e-04 &  2.01 &  3.80e-04 &  2.01 &  7.00e-03 &  2.01 &  7.01e-03 &  2.01 &  1.83e-01 &  2.00 & 2.62e+01 \\
    0.0110 &  8.00e-05 &  2.00 &  1.00e-04 &  2.00 &  1.75e-03 &  2.00 &  1.75e-03 &  2.00 &  4.58e-02 &  2.00 & 2.62e+01 \\
    0.0055 &  2.00e-05 &  1.99 &  2.00e-05 &  1.99 &  4.40e-04 &  2.00 &  4.40e-04 &  1.99 &  1.15e-02 &  2.00 & 2.61e+01 \\
    \midrule
       \multicolumn{12}{c}{\textbf{$\theta=-1$}} \\
   \midrule
    0.3536 &  3.75e-02 &  0.00 &  6.46e-02 &  0.00 &  1.22e+00 &  0.00 &  1.24e+00 &  0.00 & 4.66e+01 &  0.00 & 3.75e+01 \\
    0.1768 &  1.50e-02 &  1.32 &  1.91e-02 &  1.76 &  4.07e-01 &  1.59 &  4.08e-01 &  1.61 &  1.16e+01 &  2.01 & 2.84e+01 \\
    0.0884 &  5.03e-03 &  1.58 &  5.92e-03 &  1.69 &  1.11e-01 &  1.87 &  1.11e-01 &  1.87 &  2.93e+00 &  1.98 & 2.63e+01 \\
    0.0442 &  1.31e-03 &  1.94 &  1.53e-03 &  1.95 &  2.81e-02 &  1.98 &  2.82e-02 &  1.98 &  7.34e-01 &  2.00 & 2.61e+01 \\
    0.0221 &  3.31e-04 &  2.01 &  3.80e-04 &  2.01 &  7.00e-03 &  2.01 &  7.01e-03 &  2.01 &  1.83e-01 &  2.00 & 2.62e+01 \\
    0.0110 &  8.00e-05 &  2.00 &  1.00e-04 &  2.00 &  1.75e-03 &  2.00 &  1.75e-03 &  2.00 &  4.58e-02 &  2.00 & 2.62e+01 \\
    0.0055 &  2.00e-05 &  1.99 &  2.00e-05 &  1.99 &  4.40e-04 &  2.00 &  4.40e-04 &  1.99 &  1.15e-02 &  2.00 & 2.61e+01 \\
     \midrule
       \multicolumn{12}{c}{$\theta=0$} \\
     \midrule
    0.3536 &  3.65e-02 &  0.00 &  6.34e-02 &  0.00 &  1.21e+00 &  0.00 &  1.23e+00 &  0.00 & 4.66e+01 &  0.00 & 3.77e+01 \\
    0.1768 &  1.50e-02 &  1.28 &  1.91e-02 &  1.73 &  4.07e-01 &  1.58 &  4.08e-01 &  1.60 &  1.16e+01 &  2.01 & 2.83e+01 \\
    0.0884 &  5.04e-03 &  1.58 &  5.92e-03 &  1.69 &  1.11e-01 &  1.87 &  1.11e-01 &  1.87 &  2.93e+00 &  1.98 & 2.63e+01 \\
    0.0442 &  1.31e-03 &  1.94 &  1.53e-03 &  1.95 &  2.81e-02 &  1.98 &  2.82e-02 &  1.98 &  7.34e-01 &  2.00 & 2.61e+01 \\
    0.0221 &  3.31e-04 &  2.01 &  3.80e-04 &  2.01 &  7.00e-03 &  2.01 &  7.01e-03 &  2.01 &  1.83e-01 &  2.00 & 2.62e+01 \\
    0.0110 &  8.00e-05 &  2.00 &  1.00e-04 &  2.00 &  1.75e-03 &  2.00 &  1.75e-03 &  2.00 &  4.58e-02 &  2.00 & 2.62e+01 \\
    0.0055 &  2.00e-05 &  1.99 &  2.00e-05 &  1.99 &  4.40e-04 &  2.00 &  4.40e-04 &  1.99 &  1.15e-02 &  2.00 & 2.61e+01 \\
    \bottomrule
    \end{tabular}
    \label{table4}
    \end{footnotesize}
\end{table}
\begin{table}[t!]
    \centering
    \begin{footnotesize}
   \caption{Values of $\left\|\boldsymbol{u}_h \cdot \boldsymbol{n}\right\|_{0, \Gamma_{\mathrm{Nav}}}$ computed for various $\theta$ and $\gamma$ with parameters $\nu = 1$, $\beta = 10$, and $\mathbb{P}_1^2 - \mathbb{P}_1$ elements under uniform mesh refinement.}
    \begin{tabular}{lccc|ccc}
    \toprule
    & \multicolumn{3}{c}{$\theta=-1$} & \multicolumn{3}{|c}{$\theta=1$} \\
    \cmidrule(lr){2-4} \cmidrule(lr){5-7}
    $h$ & $\gamma=10^{-3}$ & $\gamma=1$ & $\gamma=10^3$ & $\gamma=10^{-3}$ & $\gamma=1$ & $\gamma=10^3$ \\
    \midrule
    0.3536 & 3.739618 & 2.703788 & 0.008016 & 1.672831 & 1.241215 & 0.007919 \\
    0.1768 & 8.029179 & 1.281138 & 0.001739 & 0.471996 & 0.334333 & 0.001718  \\
    0.0884 & 1.479223 & 0.701043 & 0.000421 & 0.115706 & 0.082970 & 0.000416  \\
    0.0442 & 0.476380 & 0.086847 & 0.000103 & 0.028373 & 0.020496 & 0.000102  \\
    0.0221 & 0.124253 & 0.025162 & 0.000025 & 0.007058 & 0.005110 & 0.000025  \\
    0.0110 & 0.031169 & 0.005674 & 0.000006 & 0.001763 & 0.001277 & 0.000006  \\
    0.0055 & 0.007796 & 0.001387 & 0.000002 & 0.000441 & 0.000319 & 0.000001 \\ 
    \bottomrule 
    \end{tabular}
    \label{table5}
    \end{footnotesize}
\end{table}
\subsection{Lid driven cavity problem}
The lid-driven cavity problem serves as a well-established benchmark in computational fluid mechanics (see \cite{medic1999nsike, GhiaShin1982}). The computational domain is defined as a square region, $\Omega = (0,1)^2$, with negligible body forces and a single moving boundary. A uniform velocity, $\boldsymbol{u} = (1,0)^{\mathrm{T}}$, is applied along the top boundary at $y = 1$. 
Two cases are considered: in the first case, a homogeneous slip boundary condition with a friction coefficient $\beta = 0$ and Nitsche parameter $\gamma = 10$ is imposed on the remaining three sides, while in the second case, no-slip boundary conditions are enforced on the same boundaries. The Reynolds number is defined as $Re = 1/\nu$, and computations are conducted for $Re = 5000$. The velocity streamline plots obtained are similar to those presented in \cite{MR4812237}, exhibiting one primary vortex and three secondary vortices in the no-slip case. In contrast, in the slip case, the viscous forces near the boundaries are reduced, leading to decreased flow instabilities and the absence of secondary vortices. The final adapted mesh for the $\mathbb{P}_1^2-\mathbb{P}_1$ element pair and the corresponding velocity streamlines for the symmetric formulation ($\theta = 1$) are shown in Figures~\ref{figure2} for both slip and no-slip boundary conditions. Notably, the mesh refinement is concentrated within the primary vortex in the no-slip boundary condition case, contributing to accurate solution approximation. In the slip case, the refinement is concentrated near the boundaries. It is also observed that increasing the friction coefficient \(\beta\) yields results comparable to those obtained with no-slip boundary conditions. Furthermore, Table \ref{tab:comparison} demonstrates that the position of the primary vortex center, determined using the stabilized finite element method \eqref{DSF}, aligns well with the results obtained by Ghia and Shin \cite{GhiaShin1982} and Medic and Mohammadi \cite{medic1999nsike}.
\begin{table}[t!]
    \centering
    \begin{footnotesize}
   \caption{Comparison of results using various schemes and boundary conditions.}
      \begin{tabular}{@{}lcc@{}}
    \toprule
    Scheme & Re = 5000 \\ \midrule
    Ghia et al.  \cite{GhiaShin1982} & $x = 0.5117; \ y = 0.5352$ \\
    Medic et al.  \cite{medic1999nsike} & $x = 0.53; \ y = 0.53$ \\
    Stabilized FE (with slip bc)  $\mathbb{P}_1^2 \times \mathbb{P}_1$ (adapted mesh) & $x = 0.5018; \ y = 0.5000$ \\
    Stabilized FE (with no-slip bc)  $\mathbb{P}_1^2 \times \mathbb{P}_1$ (adapted mesh) & $x = 0.5172; \ y = 0.5351$ \\
    \bottomrule
    \end{tabular}
\label{tab:comparison}
\end{footnotesize}
\end{table}
\begin{figure}[H]
\centering
\subfloat{\includegraphics[width=0.25\linewidth]{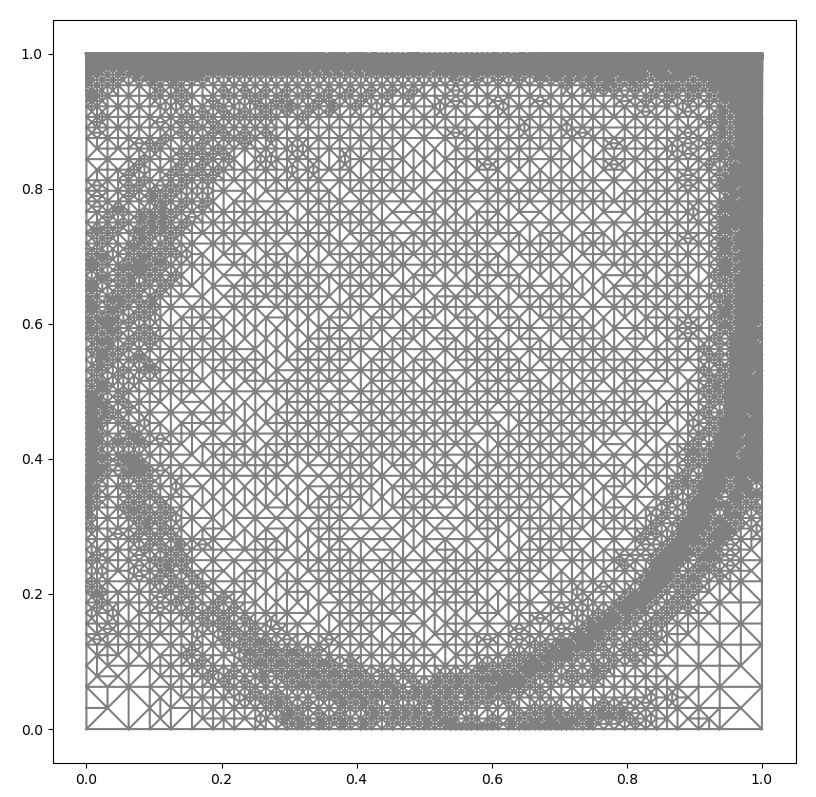}} 
 \subfloat{\includegraphics[width=0.25\linewidth]{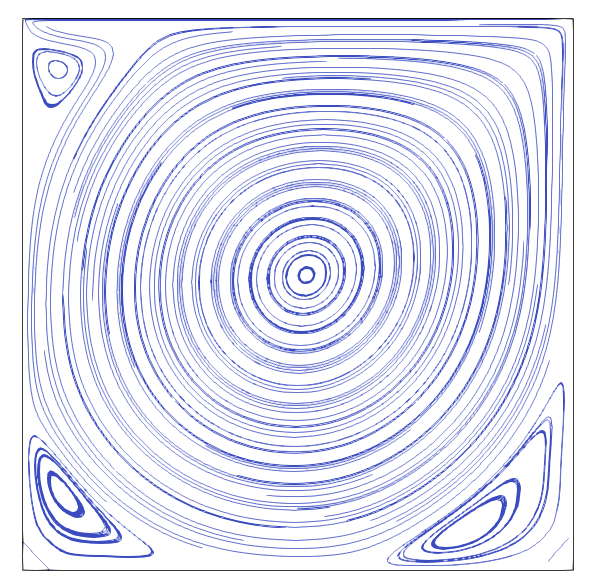}} 
 \subfloat{\includegraphics[width=0.25\linewidth]{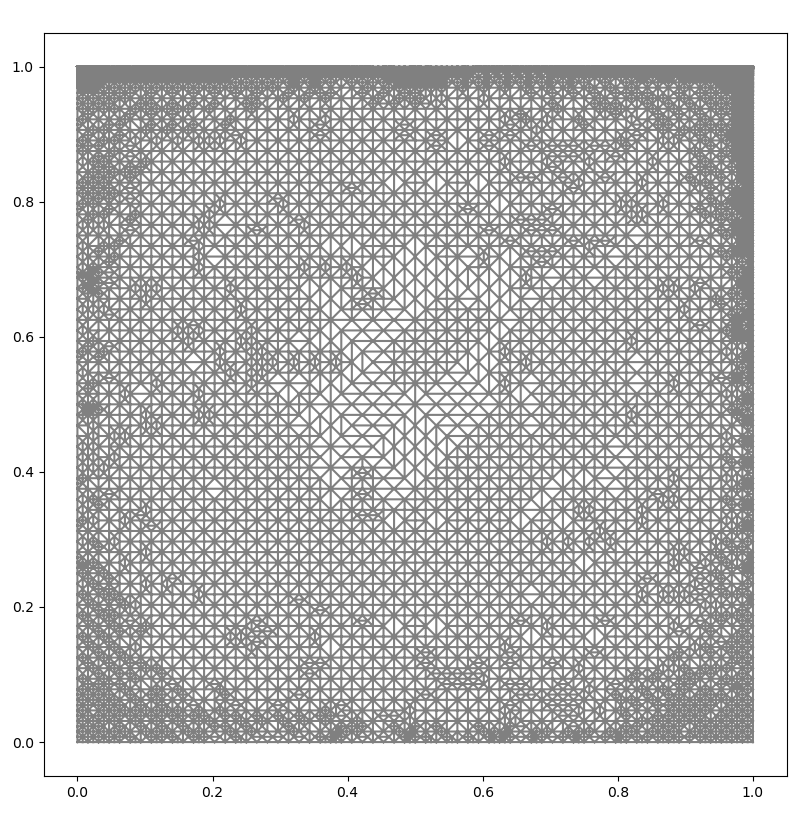}} 
 \subfloat{\includegraphics[width=0.25\linewidth]{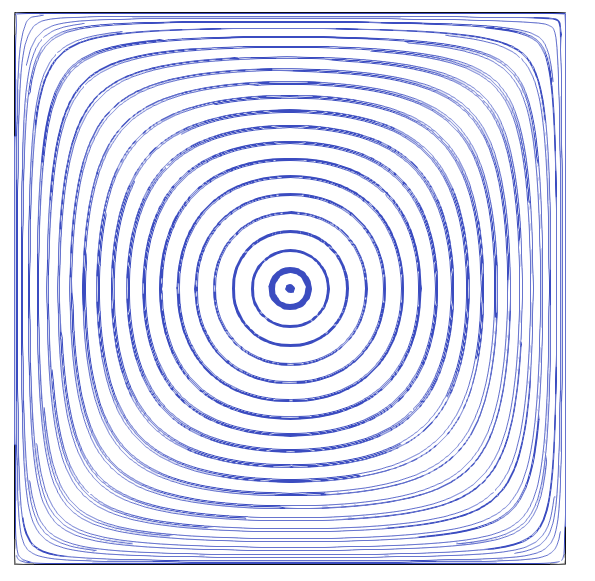}}
\caption{Refined mesh and streamlines with no-slip (top row) and slip (bottom row) boundary conditions at $Re = 5000.$}
\label{figure2}
\end{figure}
\subsection{Convergence in the case of adaptive mesh refinement (non-convex domain)}
We conduct a test to recover optimal convergence rates through adaptive mesh refinement. This process is guided by the proposed \emph{a posteriori} error estimator and follows the maximum marking strategy. 

The convergence rates are computed using the alternative formula
\begin{align*}
\text{rate} = -2 \log(e / \widehat{e}) \big[\log(\mathrm{DoFs} / \widehat{\mathrm{DoF}})\big]^{-1}.
\end{align*}

We consider a non-convex rotated L-shaped domain $\Omega = (-1, 1)^2 \setminus (-1, 0)^2$ and employ manufactured velocity and fluid pressure solutions with sharp gradients near the domain re-entrant corner
\begin{align*}
\boldsymbol{u}(r, s) &= \frac{r^\chi}{2a} 
\begin{pmatrix}
-(\chi+1) \cos ([\chi+1] s) + (M_2-\chi-1) M_1 \cos ([\chi-1] s) \\
(\chi+1) \sin ([\chi+1] s) + (M_2+\chi-1) M_1 \sin ([\chi-1] s)
\end{pmatrix}, \\
p(r, s) &= r^{1/3} \sin \left(\frac{1}{3} \left(\frac{\pi}{2} + s\right)\right),
\end{align*}
where the polar coordinates and parameters are defined as $M_1 = \frac{-\cos ([\chi+1] \omega)}{\cos ([\chi-1] \omega)}, ~ M_2 = \frac{2(a+2b)}{a+b}, ~
r = \sqrt{x_1^2 + x_2^2}, ~ s = \arctan(x_2, x_1), ~
\omega = \frac{3\pi}{4},~ \chi = 0.54448373, ~ a = 10^3, ~  b = 10.$ These computations were performed using the $\mathbb{P}_1^2-\mathbb{P}_1$ element pair with the parameters $\nu = 1$, $\beta = 0$, $\theta = 1$, and $\gamma = 10$. The boundary conditions are defined as follows: \(\Gamma_{\mathrm{Nav}}\) represents the segments where \(x = 0\) and \(y = 0\), while \(\Gamma_{\mathrm{D}}\) covers the rest of the boundary. The forcing terms are derived from the manufactured solution.

It is important to note that the exact pressure belongs to the space \(H^{4/3 - \epsilon}(\Omega)\) for any \(\epsilon > 0\). This indicates sufficient regularity to achieve optimal convergence. However, the pressure gradient has a singularity at the re-entrant corner, leading us to expect a convergence rate of approximately \(\mathcal{O}(h^{1/3})\). The numerical results for this test are summarized in Table \ref{conv_adap}. We observe the expected sub-optimal convergence when using uniform mesh refinement, while optimal convergence across all variables is achieved in the adaptive case due to localized mesh refinement. Notably, this optimal convergence is reached with a significantly smaller number of degrees of freedom compared to the uniformly refined case. The last column of the table again confirms the reliability and efficiency of the \emph{a posteriori} error estimator. Furthermore, in Figure~\ref{figure_shaped} sample triangulations obtained after several adaptive refinement steps, which illustrate the anticipated clustering of vertices near the re-entrant corner. The error plots in Figure \ref{figure_error_plots} reveal that the error in the different norms converges suboptimally under uniform refinement, while in the case of adaptive refinement, we achieve an optimal rate of convergence.
\begin{table}[ht]
\centering
\begin{footnotesize}
\label{conv_adap}
\caption{Comparison of results obtained using uniform and adaptive mesh refinement with $\mathbb{P}_1^2-\mathbb{P}_1$ elements. }
\begin{tabular}{ccccc ccc ccc c c}
\toprule
\multirow{2}{*}{DoFs} & \multirow{2}{*}{$h$} & \multicolumn{2}{c}{$\|p-p_h\|_{0,\Omega}$} & \multicolumn{2}{c}{$\|\boldsymbol{u}-\boldsymbol{u}_h\|_{0,\Omega}$} & \multicolumn{2}{c}{$\|\boldsymbol{u}-\boldsymbol{u}_h\|_{1,\Omega}$} & \multicolumn{2}{c}{T.E.} & \multicolumn{2}{c}{$\Psi$} & \multirow{2}{*}{Effec} \\
\cmidrule{3-12}
& & Error & Rate & Error & Rate & Error & Rate & Error & Rate & Error & Rate & \\
\midrule
\multicolumn{13}{c}{Uniform Mesh Refinement} \\
\midrule
25 & 1.4142 & 4.6e-01 & 0.00 & 3.7e-02 & 0.00 & 1.6e-01 & 0.00 & 2.0e-01 & 0.00 & 1.7e+00 & 0.00 & 8.09 \\
64 & 0.7071 & 2.0e-01 & 1.19 & 1.3e-02 & 1.49 & 1.1e-01 & 0.46 & 1.2e-01 & 0.72 & 1.1e+00 & 0.61 & 8.74 \\
196 & 0.3536 & 8.2e-02 & 1.31 & 6.1e-03 & 1.09 & 8.1e-02 & 0.47 & 8.5e-02 & 0.54 & 7.6e-01 & 0.52 & 8.91 \\
676 & 0.1768 & 4.5e-02 & 0.86 & 2.9e-03 & 1.06 & 5.7e-02 & 0.51 & 5.9e-02 & 0.53 & 5.3e-01 & 0.51 & 9.02 \\
2500 & 0.0884 & 2.7e-02 & 0.74 & 1.5e-03 & 1.00 & 3.9e-02 & 0.53 & 4.1e-02 & 0.53 & 3.7e-01 & 0.53 & 9.08 \\
9604 & 0.0442 & 1.7e-02 & 0.67 & 7.6e-04 & 0.95 & 2.7e-02 & 0.54 & 2.8e-02 & 0.54 & 2.5e-01 & 0.53 & 9.11 \\
37636 & 0.0221 & 1.1e-02 & 0.63 & 4.0e-04 & 0.91 & 1.9e-02 & 0.54 & 1.9e-02 & 0.54 & 1.7e-01 & 0.54 & 9.12 \\
\midrule
\multicolumn{13}{c}{Adaptive Mesh Refinement} \\
\midrule
25 & 1.4142 & 4.6e-01 & 0.00 & 3.7e-02 & 0.00 & 1.6e-01 & 0.00 & 2.0e-01 & 0.00 & 1.7e+00 & 0.00 & 8.09 \\
55 & 1.0000 & 1.6e-01 & 2.76 & 1.7e-02 & 1.99 & 1.1e-01 & 0.78 & 1.2e-01 & 1.29 & 1.1e+00 & 1.06 & 8.83 \\
118 & 0.7071 & 9.3e-02 & 1.35 & 8.3e-03 & 1.86 & 8.7e-02 & 0.73 & 9.0e-02 & 0.80 & 7.8e-01 & 0.85 & 8.67 \\
169 & 0.7071 & 5.9e-02 & 2.53 & 6.6e-03 & 1.22 & 7.1e-02 & 1.10 & 7.3e-02 & 1.17 & 6.3e-01 & 1.20 & 8.63 \\
220 & 0.7071 & 4.4e-02 & 2.19 & 6.1e-03 & 0.60 & 6.3e-02 & 0.96 & 6.4e-02 & 0.97 & 5.5e-01 & 1.12 & 8.47 \\
304 & 0.7071 & 4.4e-02 & -0.01 & 5.1e-03 & 1.13 & 5.4e-02 & 0.96 & 5.6e-02 & 0.91 & 4.7e-01 & 0.90 & 8.49 \\
451 & 0.5000 & 2.8e-02 & 2.26 & 3.6e-03 & 1.75 & 4.5e-02 & 0.91 & 4.6e-02 & 0.93 & 3.9e-01 & 0.99 & 8.40 \\
604 & 0.5000 & 2.2e-02 & 1.83 & 3.3e-03 & 0.69 & 4.0e-02 & 0.87 & 4.1e-02 & 0.88 & 3.3e-01 & 1.06 & 8.18 \\
907 & 0.3536 & 1.5e-02 & 1.85 & 1.4e-03 & 4.17 & 3.1e-02 & 1.15 & 3.2e-02 & 1.19 & 2.7e-01 & 1.03 & 8.47 \\
1162 & 0.3536 & 1.3e-02 & 1.39 & 1.3e-03 & 0.55 & 2.7e-02 & 1.10 & 2.8e-02 & 1.11 & 2.4e-01 & 1.04 & 8.54 \\
1912 & 0.3536 & 1.1e-02 & 0.40 & 1.1e-03 & 0.79 & 2.2e-02 & 0.88 & 2.2e-02 & 0.88 & 1.9e-01 & 0.99 & 8.32 \\
2548 & 0.2500 & 8.2e-03 & 2.25 & 6.9e-04 & 3.05 & 1.8e-02 & 1.18 & 1.9e-02 & 1.20 & 1.6e-01 & 1.04 & 8.51 \\
3694 & 0.2500 & 7.0e-03 & 0.87 & 5.8e-04 & 1.00 & 1.6e-02 & 0.86 & 1.6e-02 & 0.86 & 1.3e-01 & 0.96 & 8.35 \\
5389 & 0.1768 & 4.9e-03 & 1.85 & 3.0e-04 & 3.37 & 1.3e-02 & 1.02 & 1.3e-02 & 1.04 & 1.0e-01 & 1.10 & 8.43 \\
\bottomrule
\end{tabular}
\end{footnotesize}
\end{table}
\begin{figure}[H]
\centering
\subfloat{\includegraphics[width=0.33\linewidth]{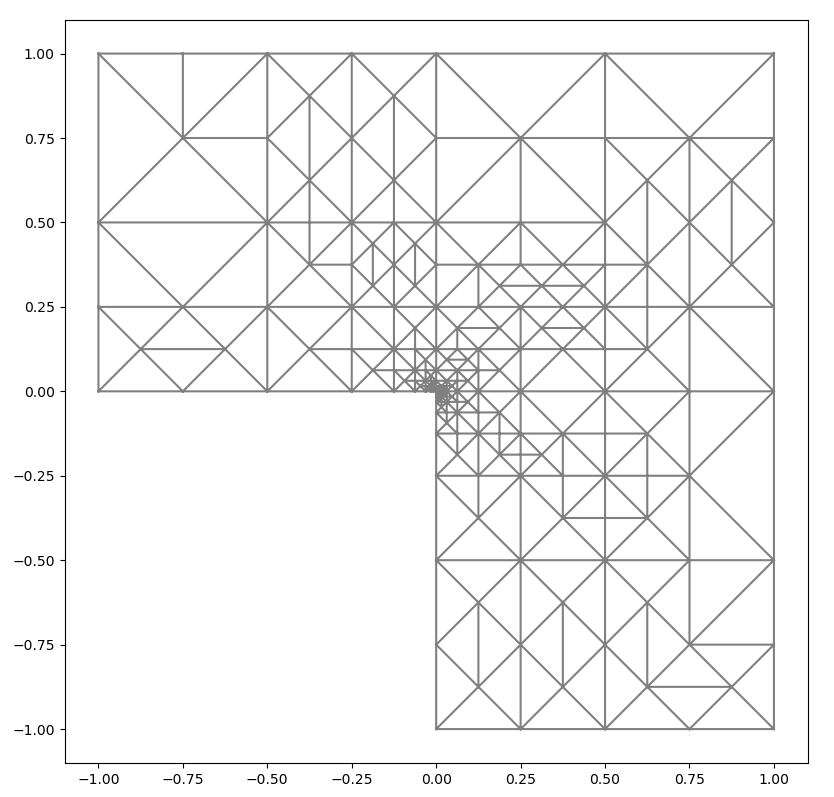}} 
 \subfloat{\includegraphics[width=0.33\linewidth]{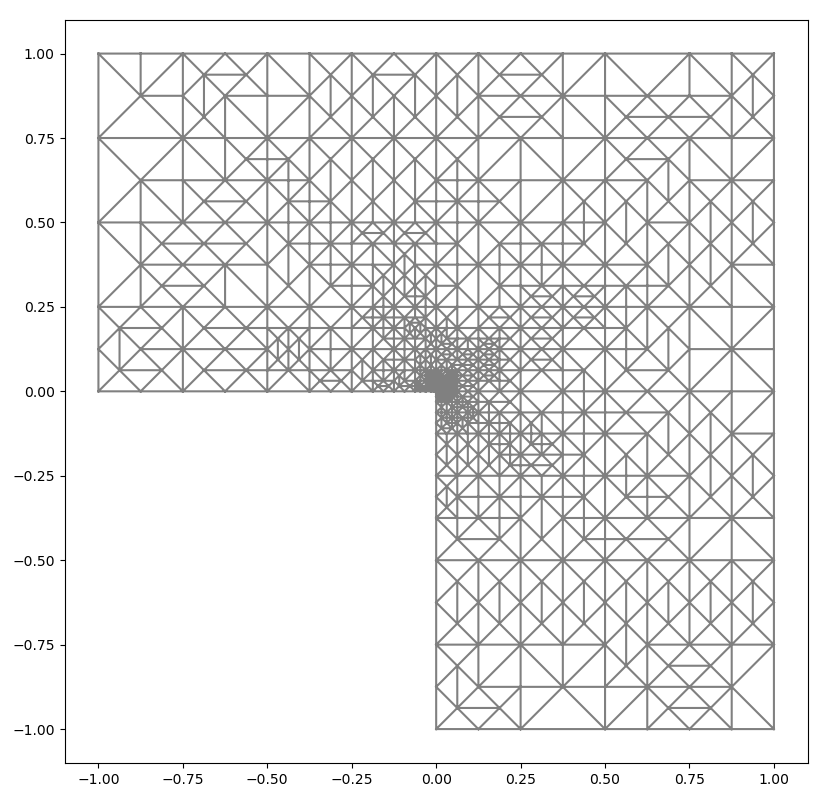}} 
 \subfloat{\includegraphics[width=0.33\linewidth]{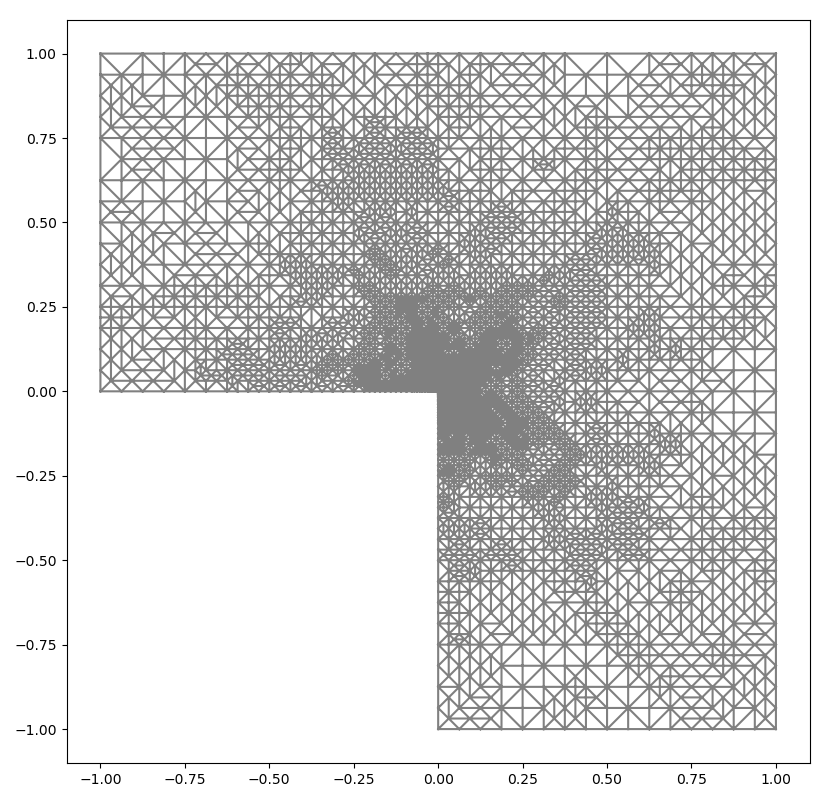}} 
\caption{The refined meshes obtained using the adaptive strategy with $\tilde{\theta} = 0.5$, corresponding to $451$, $1912$, and $11146$ degrees of freedom, respectively. 
}
\label{figure_shaped}
\end{figure}
\begin{figure}[H]
\centering
\subfloat{\includegraphics[width=0.4\linewidth]{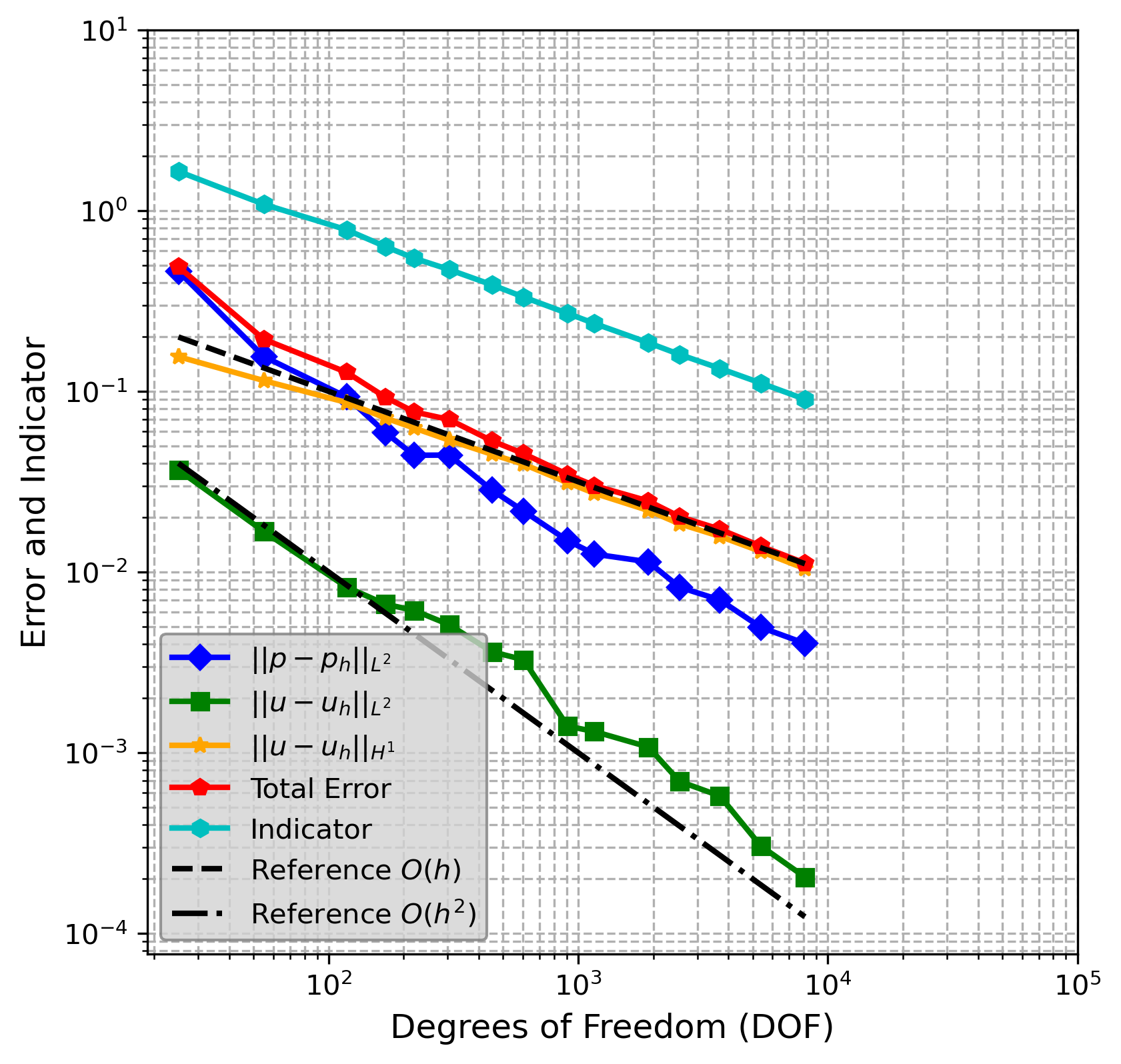}} 
 \subfloat{\includegraphics[width=0.4\linewidth]{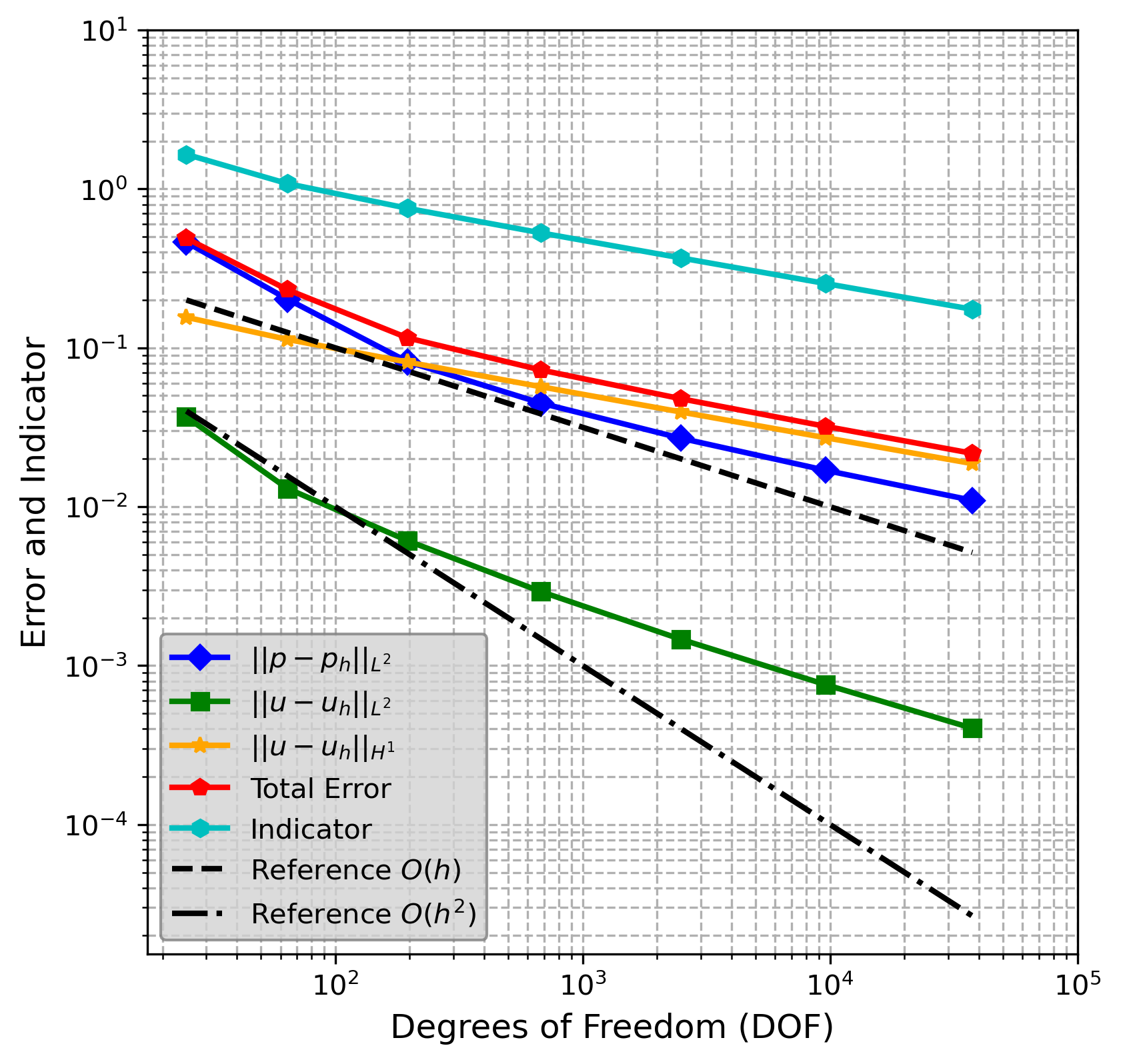}} 
\caption{Comparison of errors and error indicators for adaptive refinement (left) and uniform refinement (right) with $\tilde{\theta} = 0.5$. }
\label{figure_error_plots}
\end{figure}
\subsection{Flow past through a 3D circular cylinder problem}
The geometrical settings of the domain are taken from \cite{MR4744101, Bayraktar2012, BraackRichter2006}. The domain $\Omega$ is the region $] 0,2.5[\times$ $] 0, H[\times] 0, H[$, with $H=0.41 \mathrm{~m}$, without a cylinder of diameter $D=0.1 \mathrm{~m}$. 
In this problem, we impose no-slip boundary conditions on all the lateral walls of the box, and we use do-nothing boundary conditions at the outflow plane. For the surface of the cylinder, we consider two cases: (1) with a homogeneous slip condition and (2) with a no-slip condition. The final adapted meshes obtained using our adaptive scheme for both cases are shown in Figure ~\ref{cylinder_slip}. As expected, most of the refinement is concentrated near the cylinder.
\begin{figure}[H]
\centering
    \subfloat{\includegraphics[width=0.5\linewidth]{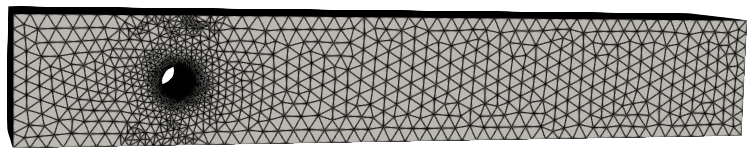}} \\
 \subfloat{\includegraphics[width=0.5\linewidth]{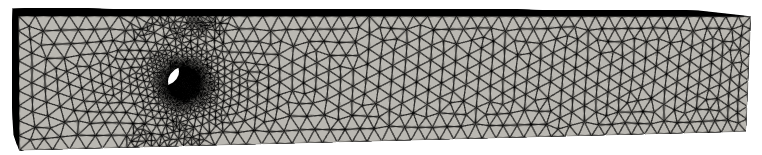}} 
	\caption{Surface view of the final adapted mesh with $\tilde{\theta} = 0.5$ (976,763 elements), showing no-slip (top) and slip (bottom) conditions on the inner cylinder.}
	\label{cylinder_slip}
\end{figure}	
Finally, the inflow condition from \cite{Bayraktar2012} is given by
$\boldsymbol{u}_D:=\left(\frac{16 U  y z(H-y)(H-z)}{H^4}, 0,0\right)^T,$
with $U=0.45 \mathrm{~m} / \mathrm{s}$, the fluid viscosity is given by $\nu=10^{-3}$ and the right-hand side of the momentum equation vanishes, i.e. $\boldsymbol{f}=\boldsymbol{0}$. We have
chosen the following values for the remaining parameters: $\theta = -1$, $\beta = 0$, and $\gamma = 100$. 
The quantities to compute are the following three: the pressure difference $\Delta p$ between the points $(0.55,0.2,0.205)$ and ( $0.45,0.2,0.205)$, and the drag and lift coefficients defined as follows:
$$
C_{\text {drag }}:=\frac{2 F_{\text {drag }}}{\rho \bar{\boldsymbol{u}}^2 D H} \quad \text { and } \quad C_{\text {lift }}:=\frac{2 F_{\text {lift }}}{\rho \bar{\boldsymbol{u}}^2 D H}
$$
where $\rho=1$ and $\bar{u}=0.2$, are the density of the fluid and the mean inflow, respectively, and
$$
F_{\mathrm{drag}}:=\int_S\left(\rho \nu \frac{\partial \boldsymbol{u}_t}{\partial \boldsymbol{n}} n_y-p n_x\right) \mathrm{d} S \quad \text { and } \quad F_{\text {lift }}:=\int_S\left(\rho \nu \frac{\partial \boldsymbol{u}_t}{\partial \boldsymbol{n}} n_x-p n_y\right) \mathrm{d} S
$$
are the drag and lift forces. Here $S$ is the surface of the cylinder, $\boldsymbol{n}=\left(n_x, n_y, n_z\right)$ the inward pointing unit vector with respect to $\Omega, \boldsymbol{t}$ a tangential vector on $S$ and $\boldsymbol{u}_t=\boldsymbol{u} \cdot \boldsymbol{t}$ is the projection of the velocity into the direction $\boldsymbol{t}$.

The stabilized method with adaptive mesh is used to compute the drag coefficient, lift coefficient, and pressure drop. The computed values are $C_{\text{drag}} = 6.00488 $, $C_{\text{lift}} =  0.09602$, and $\Delta p = 0.16557 $. These results are obtained by imposing the no-slip boundary condition on the inner cylinder. The computed values are in good agreement with the reference intervals reported in the literature (see \cite{MR1928952}): $C_{\text{drag}} \in [6.05, 6.25]$, $C_{\text{lift}} \in [0.008, 0.1]$, and $\Delta p \in [0.165, 0.175]$.

When a slip boundary condition is applied to the inner cylinder, while the other conditions are kept the same, the computed values are \( C_{\text{drag}} = 5.88612  \), \( C_{\text{lift}} = 0.07487 \), and \( \Delta p = 0.16245 \). These results indicate a lower drag coefficient, a lower lift coefficient, and a reduced pressure difference, suggesting a more efficient flow around the cylinder compared to the no-slip case.

The numerical solutions for fluid velocity and pressure under both the no-slip and slip boundary conditions are presented in Figure~\ref{cylinder_noslip_velocity}. Comparing the results, we observe that the flow velocity is higher with the slip boundary condition compared to the no-slip boundary condition.
\begin{figure}[H]
\centering
    \subfloat{\includegraphics[width=0.5\linewidth]{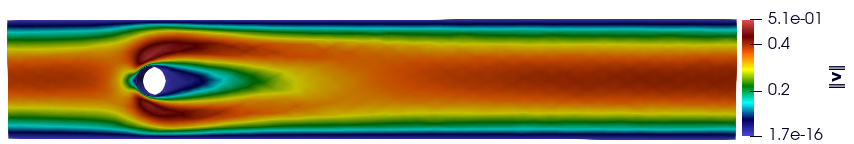}} 
 \subfloat{\includegraphics[width=0.5\linewidth]{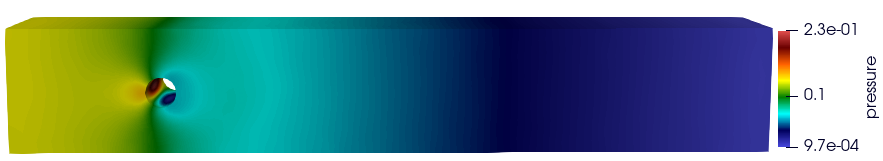}} \\
    \subfloat{\includegraphics[width=0.5\linewidth]{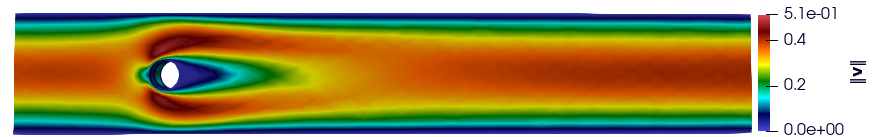}} 
 \subfloat{\includegraphics[width=0.5\linewidth]{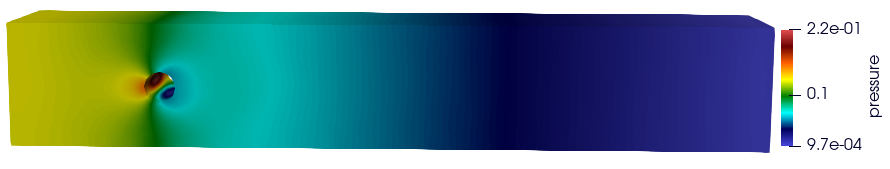}} 
	\caption{Velocity magnitude and Pressure isovalues on the final adapted mesh with no-slip (top row) and slip (bottom row) boundary conditions.}
	\label{cylinder_noslip_velocity}
\end{figure}

\section{Conclusion}
This study extends the equal-order stabilized scheme proposed by Franca et al. to the stationary Navier-Stokes equations with slip boundary conditions. We develop a reliable method based on Nitsche's approach to effectively address boundary conditions on domains with complex geometries. We prove that the discrete problem is well-posed under mild assumptions and show that the approximation error achieves optimal convergence rates. Furthermore, we rigorously evaluate the efficiency and reliability of residual-based a posteriori error estimators for the discrete problem. The theoretical results are validated through several benchmark numerical tests to confirm that the proposed scheme works well. We also compare the results with slip and no-slip boundary conditions in the lid-driven cavity test and the flow past a 3D cylinder, demonstrating the robustness of the slip boundary conditions. The extension of this scheme to the unsteady Navier-Stokes equations is left for future work, as it will require designing suitable stabilization techniques. Additionally, developing parameter-robust block preconditioners to improve the computational efficiency and scalability will be a key area for further research.

\section*{Acknowledgments}
AB was supported by the Ministry of Education, Government of India - MHRD. NAB was supported by Centro de Modelamiento Matematico (CMM), Proyecto Basal FB210005, and by ANID, Proyecto 3230326.

\section*{Conflict of interest statement}
 The authors affirm that they do not have any conflicts  of interests.

\section*{Data availability statement}
 Data generated during the research discussed in the paper will be made available upon reasonable request. 
\bibliographystyle{siam}
\bibliography{main}
\end{document}